\documentclass[11pt]{article}

\usepackage[utf8]{inputenc}
\usepackage[T1]{fontenc}
\usepackage{lmodern}
\usepackage{newunicodechar}
\newunicodechar{，}{,}

\usepackage[margin=1in]{geometry}
\usepackage{microtype}

\usepackage{amsmath, amssymb, amsfonts, amsthm, mathtools, bm}
\usepackage{enumitem}
\usepackage{cases}
\usepackage{xcolor}
\usepackage{tikz}
\usetikzlibrary{arrows.meta,positioning,calc}
\usepackage{placeins}
\usepackage{csquotes}

\usepackage[
    colorlinks=true,
    linkcolor=blue!50!black,
    citecolor=blue!50!black,
    urlcolor=blue!50!black
]{hyperref}
\usepackage[nameinlink,noabbrev]{cleveref}
\usepackage{aliascnt}

\numberwithin{equation}{section}

\theoremstyle{plain}
\newtheorem{theorem}{Theorem}[section]

\newaliascnt{lemma}{theorem}
\newtheorem{lemma}[lemma]{Lemma}
\aliascntresetthe{lemma}

\newaliascnt{proposition}{theorem}
\newtheorem{proposition}[proposition]{Proposition}
\aliascntresetthe{proposition}

\newaliascnt{corollary}{theorem}
\newtheorem{corollary}[corollary]{Corollary}
\aliascntresetthe{corollary}

\newaliascnt{conjecture}{theorem}

\aliascntresetthe{conjecture}

\newaliascnt{claim}{theorem}

\aliascntresetthe{claim}

\theoremstyle{definition}
\newaliascnt{definition}{theorem}
\newtheorem{definition}[definition]{Definition}
\aliascntresetthe{definition}

\theoremstyle{remark}
\newaliascnt{remark}{theorem}
\newtheorem{remark}[remark]{Remark}
\aliascntresetthe{remark}

\crefname{theorem}{Theorem}{Theorems}
\crefname{lemma}{Lemma}{Lemmas}
\crefname{proposition}{Proposition}{Propositions}
\crefname{corollary}{Corollary}{Corollaries}
\crefname{definition}{Definition}{Definitions}
\crefname{conjecture}{Conjecture}{Conjectures}
\crefname{claim}{Claim}{Claims}
\crefname{remark}{Remark}{Remarks}
\crefname{researchproblem}{Research Problem}{Research Problems}

\Crefname{theorem}{Theorem}{Theorems}
\Crefname{lemma}{Lemma}{Lemmas}
\Crefname{proposition}{Proposition}{Propositions}
\Crefname{corollary}{Corollary}{Corollaries}
\Crefname{definition}{Definition}{Definitions}
\Crefname{conjecture}{Conjecture}{Conjectures}
\Crefname{claim}{Claim}{Claims}
\Crefname{remark}{Remark}{Remarks}
\Crefname{researchproblem}{Research Problem}{Research Problems}

\setlist{itemsep=2pt, topsep=6pt, parsep=2pt}

\newcommand{\N}{\mathbb{N}}
\newcommand{\Z}{\mathbb{Z}}

\newcommand{\1}{\mathbf{1}}

\definecolor{roadmaproot}{RGB}{246,246,255}
\definecolor{roadmapsplit}{RGB}{248,248,248}
\definecolor{roadmapminor}{RGB}{241,247,254}
\definecolor{roadmapmajor}{RGB}{254,249,236}
\definecolor{roadmapmerge}{RGB}{242,250,243}

\newcommand{\roadmapkey}[1]{\textcolor{blue!60!black}{\textbf{#1}}}

\tikzset{
  roadmap arrow/.style={
    ->,
    >=Stealth,
    line width=.65pt,
    draw=black!78
  },
  roadmap box/.style={
    draw=black!72,
    rounded corners=3pt,
    line width=.55pt,
    align=center,
    inner xsep=6pt,
    inner ysep=6pt,
    font=\small
  },
  roadmap root/.style={
    roadmap box,
    fill=roadmaproot,
    text width=12.6cm
  },
  roadmap split/.style={
    roadmap box,
    fill=roadmapsplit,
    text width=8.2cm
  },
  roadmap minor/.style={
    roadmap box,
    fill=roadmapminor,
    text width=6.05cm
  },
  roadmap major/.style={
    roadmap box,
    fill=roadmapmajor,
    text width=6.05cm
  },
  roadmap merge/.style={
    roadmap box,
    fill=roadmapmerge,
    text width=11.8cm
  },
  roadmap final/.style={
    roadmap box,
    fill=roadmaproot,
    text width=12.8cm
  },
  roadmap edge label/.style={
    fill=white,
    inner sep=1.2pt,
    font=\footnotesize,
    text=black!75
  }
}

\title{Primes with Restricted-Digit Differences}
\author{Rui Han, Yaghoub Rahimi, Fan Yang}

\date{}

\begin{document}

\maketitle

\begin{abstract}
Fix a base \(b\ge2\) and a digit set
\(\mathcal D\subset\{0,1,\dots,b-1\}\), and let \(\mathcal C_k\) be the set of
integers in \([0,b^k)\) whose \(k\) base-\(b\) digits all lie in
\(\mathcal D\).  We study prime pairs and three-term arithmetic progressions
in primes whose difference, or common difference, belongs to
\(\mathcal C_k\).  Under a natural digit non-resonance condition, we establish a localized
Fourier criterion guaranteeing asymptotic formulae with explicit leading
constants for both problems.  Using a transfer-operator argument, we verify
this criterion in explicit parameter ranges for several natural families of
digit sets.
\end{abstract}

\tableofcontents

\section{Introduction}
\label{sec:introduction}

Additive patterns in the primes form a central theme of analytic number theory.
Hardy--Littlewood heuristics, developed through the circle method, predict
asymptotic formulae for many additive configurations in the primes, with main
terms governed by products of local density factors
\cite{HardyLittlewood,Vaughan}. In the special case of arithmetic
progressions, van der Corput proved that the primes contain infinitely many
three-term progressions \cite{VanderCorput}. Later work revealed a much richer
structure: Green proved a Roth theorem inside the primes \cite{GreenRoth},
Green and Tao proved that the primes contain arbitrarily long arithmetic
progressions \cite{GreenTaoAP}, and developed a general asymptotic theory for
systems of affine-linear forms in the primes \cite{GreenTaoLinear}.

In this paper we study a difference-set variant of these problems.  Rather than
restricting the prime variables themselves, we impose a sparse arithmetic
constraint on the \emph{difference}.  More precisely, we ask for prime pairs
and prime three-term arithmetic progressions whose difference, or common
difference, belongs to a restricted-digit set.

Our Fourier-analytic approach closely follows Maynard's work on primes and polynomial values with restricted digits \cite{MaynardRestrictedDigitsInv,MaynardRestrictedDigits}. Our setting is
closely related but distinct: the prime variables are unrestricted, while the
restricted-digit condition is imposed on their difference.

The study of sparse sets of allowable differences goes back to S\'ark\"ozy's
work on intersective sequences and difference sets
\cite{SarkozyDifferenceSetsIII}.  For the shifted-prime set
\(\{p-1:p\in\mathbb P\}\), quantitative density bounds were subsequently
obtained and refined by Lucier, Ruzsa--Sanders, Wang, and Green
\cite{LucierShiftedPrimes,RuzsaSanders,WangShiftedPrimes,
GreenSarkozyShiftedPrimes}.  These works ask when a dense set must realize a
prescribed difference and, quantitatively, how large a set avoiding such
differences can be.  By contrast, our goal is a quantitative
Hardy--Littlewood asymptotic for prime configurations after averaging over all
restricted-digit differences of a given length.

Fix a base \(b\ge2\), let
\[
        \mathcal D\subset\{0,1,\dots,b-1\},
        \qquad
        r:=|\mathcal D|,
\]
with \(r\ge2\), and put \(X:=b^k\).  We write
\(e(t):=\exp(2\pi i t)\).

{
\begin{definition}[Digit non-resonance]
\label{def:digit-nonresonance}
We say that \(\mathcal D\) is non-resonant with the base \(b\) if every
prime divisor of
\[
        g_{\mathcal D}:=\gcd\{a-a':a,a'\in\mathcal D\}
\]
divides \(b\).
\end{definition}
Throughout the paper we assume this digit non-resonance condition; see
\cref{rem:digit-resonances} for the obstruction it excludes.}

{Define}
\[
        \mathcal C_k
        :=
        \left\{
        \sum_{j=0}^{k-1} d_j b^j : d_j\in\mathcal D
        \right\}
        \subset [0,X),
\]
and 
\[
        \mathcal C
        :=
        \bigcup_{k\ge1}\mathcal C_k
        \subset \mathbb Z_{\ge0}.
\]
Clearly $|\mathcal{C}_k|=r^k$.

\begin{remark}
The set \(\mathcal C_k\) should not always be interpreted as the truncation of
the infinite restricted-digit set at height \(b^k\).  When \(0\in\mathcal D\),
the sets \(\mathcal C_k\) are nested and
\[
        \mathcal C_k=\mathcal C\cap[0,b^k).
\]
However, when \(0\notin\mathcal D\), the sets \(\mathcal C_k\) are fixed-length
layers rather than nested truncations.  In that case the layers are disjoint and
\[
        \mathcal C\cap[0,b^k)
        =
        \bigsqcup_{1\le j\le k}\mathcal C_j.
\]
Thus, in this case, \(\mathcal C_k\) does not include the smaller
restricted-digit numbers of shorter length.  
In fact, if
\(a_*:=\min\mathcal D\), then the smallest element of \(\mathcal C_k\) is
\[
        a_*\frac{b^k-1}{b-1}.
\]

We keep the fixed-length convention because it gives the exact product formula
for the Fourier transform,
\[
        \widehat C_k(\theta)
        =
        \prod_{j=0}^{k-1}
        \sum_{d\in\mathcal D} e(db^j\theta),
\]
and hence is the natural object for the circle-method and transfer-operator
estimates used below.
\end{remark}

The arithmetic of integers with restricted digits has a substantial earlier
literature.  Their distribution in residue classes was initiated by
Erd\H{o}s, Mauduit, and S\'ark\"ozy and refined by Konyagin
\cite{ErdosMauduitSarkozyI,KonyaginMissingDigits}.  Related work treats prime
factors, almost-primes, character sums, and more general arithmetic properties
of digitally restricted sets
\cite{ErdosMauduitSarkozyII,DartygeMauduitAlmostPrimes,
DartygeMauduitDensityZero,BanksConflittiShparlinski,
BanksShparlinskiRestrictedDigits}.  Drmota and Mauduit developed Weyl-sum
estimates for affine digit restrictions \cite{DrmotaMauduit}, while the recent
work of Saavedra-Araya gives a general residue-class criterion via Markov
chains \cite{SaavedraAraya}.  These results emphasize that the base and the
congruence structure of the allowed digits can create genuine local
obstructions, a point reflected in the non-resonance hypothesis used below.

There is also a growing additive-combinatorial theory of integer Cantor sets.
Walker and Walker studied arithmetic progressions lying inside restricted-digit
sets \cite{WalkerWalker}, and Yu investigated additive intersections between
sets defined by different digital expansions \cite{HanYuRestrictedDigits}.
Glasscock, Moreira, and Richter placed such examples in a broader class of
fractal sets in the integers \cite{GlasscockMoreiraRichter}.  More recently,
Burgin, Fragkos, Lacey, Mena, and Reguera studied the intersective,
uniform-distribution, ergodic, and pair-correlation properties of integer
Cantor sets \cite{BurginEtAlIntegerCantor}; in a companion work they proved a
Szemer\'edi theorem along Cantor sets, giving arithmetic progressions in
positive-density sets whose common differences belong to an integer Cantor set
\cite{BurginEtAlSzemerediCantor}. The latter is especially close to our setting, since it also requires the common difference to lie in an integer Cantor set. However, it is a qualitative density theorem, whereas the present paper seeks asymptotic counts in the sparse ambient set of primes. 
Green's work on Waring's problem with restricted digits gives another recent
circle-method application in which the summation variables themselves have
restricted expansions \cite{GreenWaringRestrictedDigits}.

Digital restrictions have also been imposed directly on prime variables.  The
literature on primes with prescribed or restricted digits includes work of
Wolke, Harman, Harman--K\'atai, Bourgain, and Swaenepoel
\cite{WolkePreassignedDigits,HarmanPreassignedDigits,HarmanKatai,
BourgainBinaryDigitsII,SwaenepoelPreassignedDigits}; related joint digital
statistics along the primes were studied by Mauduit--Rivat and
Martin--Mauduit--Rivat \cite{MauduitRivat,MartinMauduitRivat}.  Maynard proved
the existence and asymptotic distribution of primes and polynomial values with
one or more forbidden digits
\cite{MaynardRestrictedDigitsInv,MaynardRestrictedDigits}.  Building on Maynard's Fourier-analytic method, Burgin proved a
S\'ark\"ozy-type theorem in which the prescribed difference is of the form
\(p-1\), with \(p\) a prime having restricted digits
\cite{BurginSarkozyRestrictedDigits}.  Nath established
Bombieri--Vinogradov-type distribution results for primes with a missing digit
\cite{NathMissingDigitBV}, Leng and Sawhney proved a Vinogradov theorem with all
three prime summands restricted by their digits \cite{LengSawhney}, and Pratt
combined a missing-digit condition with primes represented by a quadratic form
\cite{PrattMissingDigits}.  In a different direction, Cumberbatch proved that
almost all even \emph{target integers} in a restricted-digit set are sums of
two unrestricted primes \cite{CumberbatchGoldbachDigits}; Kim obtained a
conditional short-interval version for targets with a missing digit
\cite{KimGoldbachMissingDigit}.

{In this paper, we study the weighted counting functions} 
\[
        T_k^{(2)}
        :=
        \sum_{0\le n<X}\mathbf 1_{\mathcal C_k}(n)
        \sum_{0\le m<X-n}\Lambda(m)\Lambda(m+n)
\]
{for prime pairs with restricted-digit difference, and}
\[
        T_k^{(3)}
        :=
        \sum_{0\le n\le X/2}\mathbf 1_{\mathcal C_k}(n)
        \sum_{0\le m<X-2n}\Lambda(m)\Lambda(m+n)\Lambda(m+2n)
\]
{for three-term arithmetic progressions in primes with
restricted-digit common difference.}

A central quantitative input is a localized packetwise \(L^1\)-bound for
the restricted-digit Fourier transform.  This type of hybrid estimate was
introduced by Maynard in his work on primes with restricted digits
\cite{MaynardRestrictedDigitsInv,MaynardRestrictedDigits}, and it is a main analytic innovation in studying restricted-digit integers.  We use this
packetwise \(L^1\) framework in a form adapted to the present
configuration.  Related applications appear in the work
of Nath and Burgin
\cite{NathMissingDigitBV,BurginSarkozyRestrictedDigits}; see also
\cite{DrmotaMauduit} for closely related Weyl-sum estimates for digital
restrictions.

{Throughout the rest of the paper, \(d\sim D\) means
\(D\le d<2D\).  For \(B\ge1\), the notation \(|\eta|\sim B\) means
\(B\le|\eta|<2B\), while \(|\eta|\sim0\) means \(|\eta|<1\).
When the endpoint and nonzero shifts are treated on a common scale, we write
\(\max(1,|\eta|)\sim B\) to mean
\(B\le\max(1,|\eta|)<2B\).}

\begin{definition}[\(\alpha_{b,\mathcal D}\)-admissibility]
\label{def:alpha-admissible}
{For \(\alpha_{b,\mathcal D}>0\), we say} that \(\widehat C_k\) is \(\alpha_{b,\mathcal D}\)-admissible if there
exists a constant \(K_{b,\mathcal D}>0\), independent of \(k,D,B,N\), such
that, for every integer \(N\) with \(3X\le N\le4X\), and for all \(D,B\ge1\)
satisfying
\[
        D^2B\le8X,
\]
one has
\begin{equation}\label{eq:hybrid-digit-general}
\sum_{d\sim D}
\sum_{\substack{1\le\ell\le d\\(\ell,d)=1}}
\sum_{\substack{|\eta|<B\\N\ell/d+\eta\in\mathbb Z}}
\left|
\widehat C_k\!\left(\frac{\ell}{d}+\frac{\eta}{N}\right)
\right|
\le
K_{b,\mathcal D}\,
|\mathcal C_k|\,(D^2B)^{\alpha_{b,\mathcal D}}.
\end{equation}
Here \(\eta\) ranges over the shifted lattice
\(\mathbb Z-N\ell/d\), and need not be an integer. The fixed
constant \(8\) only provides room for dyadic endpoints; replacing it by
another fixed constant changes only \(K_{b,\mathcal D}\).
\end{definition}

The exponent \(\alpha_{b,\mathcal D}\) measures the strength of the available
hybrid control on the digit side.  {In the one-missing-digit
setting, Maynard's work provides the underlying \(b\)-adic \(\ell^1\)
mechanism and the explicit benchmark
\[
\frac{\log\!\left(C_b\,\dfrac{b}{b-1}\log b\right)}{\log b},
\qquad C_b\le1+\frac3{\log b},
\]
on the natural \(b\)-adic grid
\cite{MaynardRestrictedDigitsInv,MaynardRestrictedDigits}.  The uniform
shifted-grid admissibility used here is established in
\cref{sec:transfer-model}.}

For \(d\ge1\) and \(a\in\mathbb Z\), let
\begin{equation}\label{eq:ramanujan-sum}
        c_d(a)
        :=
        \sum_{\substack{1\le \ell\le d\\(\ell,d)=1}}
        e\!\left(\frac{\ell a}{d}\right)
\end{equation}
denote the Ramanujan sum.

Our first theorem gives an asymptotic formula for prime pairs with
restricted-digit difference.

\begin{theorem}[Prime pairs with restricted-digit difference]
\label{thm:intro-two-prime}
Let \(X=b^k\).  {Assume that \(\mathcal D\) is non-resonant
with \(b\) in the sense of \cref{def:digit-nonresonance}, and that}
\(\widehat C_k\) is \(\alpha_{b,\mathcal D}\)-admissible with
\[
        0<\alpha_{b,\mathcal D}<\frac25.
\]
Then, for every fixed \(A>0\),
\[
        T_k^{(2)}
        =
        {\left(
        1-\frac{\sum_{a\in\mathcal D}a}{(b-1)|\mathcal D|}
        \right)
        \overline{\mathfrak S}^{(2)}_{b,\mathcal D}\,
        X|\mathcal C_k|}
        +
        O_{b,\mathcal D,A}\!\left(
        |\mathcal C_k|X(\log X)^{-A}
        \right),
\]
where
\[
        \overline{\mathfrak S}^{(2)}_{b,\mathcal D}
        :=
        \frac{1}{|\mathcal D|}
        \sum_{a\in\mathcal D}
        \sum_{d\mid b}
        \frac{\mu(d)^2}{\phi(d)^2}c_d(a).
\]
Moreover, {the
leading coefficient is positive if and only if either \(b\) is odd or
\(\mathcal D\) contains an even digit.}
\end{theorem}

The two-term problem is comparatively favorable.  On the Fourier side, the
restricted-digit condition appears through a single factor \(\widehat C_k\),
and the minor-arc analysis reduces to estimating packets in which the prime
side is controlled by classical exponential-sum bounds.  The threshold
\(\alpha_{b,\mathcal D}<2/5\) is exactly the point at which these packet
estimates are strong enough to yield a power-saving total minor-arc
contribution.

Our main theorem concerns three-term arithmetic progressions in primes.
Here the same restricted-digit Fourier transform appears, but the prime side is
encoded by a genuinely new three-frequency kernel, and the analysis becomes
substantially more delicate.

\begin{theorem}[Three-term progressions in primes with restricted-digit common difference]
\label{thm:intro-three-prime}
Let \(X=b^k\).  {Assume that \(\mathcal D\) is non-resonant
with \(b\) in the sense of \cref{def:digit-nonresonance}, and that}
\(\widehat C_k\) is \(\alpha_{b,\mathcal D}\)-admissible with
\[
        0<\alpha_{b,\mathcal D}<\frac16.
\]
Then, for every fixed \(A>0\),
\[
        T_k^{(3)}
        =
        \kappa_{b,\mathcal D}\, \overline{\mathfrak S}^{(3)}_{b,\mathcal D}\,X|\mathcal C_k|
        +
        O_{b,\mathcal D,A}\!\left(
        |\mathcal C_k|X(\log X)^{-A}
        \right),
\]
{where}
\[
        \overline{\mathfrak S}^{(3)}_{b,\mathcal D}
        :=
        \frac{1}{|\mathcal D|}
        \sum_{a\in\mathcal D}\mathfrak S_b^{(3)}(a)
\]
is the averaged projected three-prime local factor, and
\(\kappa_{b,\mathcal D}\) is the archimedean factor arising from the restriction
\(m,m+n,m+2n<X\).  These constants are evaluated explicitly in
\cref{subsec:averaging-three-term-local-factor}.  Moreover,
\[
        \kappa_{b,\mathcal D}\, \overline{\mathfrak S}^{(3)}_{b,\mathcal D}>0
\]
if and only if
\[
        \min\mathcal D < \frac{b-1}{2}
        \qquad\text{and}\qquad
        \mathcal D\cap \gcd(b,6)\mathbb Z\neq\emptyset.
\]
\end{theorem}

{

The leading coefficients in
\cref{thm:intro-two-prime,thm:intro-three-prime}
split naturally into an archimedean factor and an arithmetic factor.  In the
two-term theorem,
\[
        1-\frac{\sum_{a\in\mathcal D}a}{(b-1)|\mathcal D|}
\]
is the normalized archimedean factor, while
\(\overline{\mathfrak S}^{(2)}_{b,\mathcal D}\) is the averaged arithmetic
factor.  In the three-term theorem, \(\kappa_{b,\mathcal D}\) is the
archimedean factor, while
\(\overline{\mathfrak S}^{(3)}_{b,\mathcal D}\) is the averaged arithmetic
factor.  The arithmetic factors record the local congruence densities of the
prime configurations.  The archimedean factors arise from requiring all
prime variables to lie in \([0,X)\): for a fixed difference \(n\), there are
\(X-n\) admissible starting points in the two-term problem and \(X-2n\) in the
three-term problem.  Thus only \(n<X/2\) contributes in the three-term count,
which is also the source of the condition
\[
        \min\mathcal D<\frac{b-1}{2}
\]
in its non-vanishing criterion.
}

The main technical difficulty in the paper lies in the proof of
\cref{thm:intro-three-prime}.  It is already visible in the exact Fourier
identity
\[
        T_k^{(3)}
        =
        \frac1N\sum_{\alpha\in\mathcal G_N}
        \widehat H_\Lambda(\alpha)\widehat C_k(\alpha),
\]
where
\[
        \mathcal G_N:=N^{-1}\mathbb Z/\mathbb Z
        =
        \left\{\frac{a}{N}:0\le a<N\right\},
\]
and
\[
        \widehat H_\Lambda(\alpha)
        :=
        \frac1N
        \sum_{\beta\in\mathcal G_N}
        \widehat\Lambda_X(\beta)
        \widehat\Lambda_X(\alpha-2\beta)
        \widehat\Lambda_X(\beta-\alpha).
\]
Thus the prime-side kernel involves the three interacting frequencies
\[
        \beta,
        \qquad
        \alpha-2\beta,
        \qquad
        \beta-\alpha.
\]
The key minor-arc input is a three-point lemma, Lemma \ref{lem:optimal-three-point} showing that, when the outer
frequency \(\alpha\) is minor, at least one of these three prime frequencies
is minor in a quantitatively useful sense.  After Cauchy--Schwarz and
Parseval control the other two factors, the resulting bound is paired with
the packetwise digit estimate.  This yields a summable dyadic contribution
only under
\[
        \alpha_{b,\mathcal D}<\frac16,
\]
which is stricter than the two-term threshold
\(\alpha_{b,\mathcal D}<2/5\).  On the major arcs, the three-term argument
also requires an inner \(\beta\)-decomposition and an approximate decoupling
of the resulting arithmetic and archimedean variables.  These are the two
places where the proof genuinely goes beyond the prime-pair argument.

The two-term proof supplies the basic circle-method template, and the
three-term proof retains this outer architecture while adding the
three-frequency analysis just described.  The two proof roadmaps are shown in
\cref{fig:two-term-roadmap,fig:three-term-roadmap}.

\begin{figure}[p]
\centering
\begin{tikzpicture}[x=1cm,y=1cm]

\node[roadmap root] (pair-root) at (0,0) {
  \textbf{Prime pairs: exact Fourier reduction}\\[2pt]
  \(\displaystyle
    T_k^{(2)}
    =\frac1N\sum_{\alpha\in\mathcal G_N}
      \bigl|\widehat\Lambda_X(\alpha)\bigr|^2
      \widehat C_k(-\alpha)
  \)
};

\node[roadmap split] (pair-split) at (0,-2.25) {
  \textbf{Circle-method decomposition}\\[-1pt]
  \(\displaystyle
    \mathcal G_N
    =\mathfrak m_A^{(N)}\sqcup\mathfrak M_A^{(N)}
  \)
};

\node[roadmap minor] (pair-minor-1) at (-4.0,-4.85) {
  \textbf{Minor arcs}\\[-1pt]
  prime minor-arc decay for
  \(\bigl|\widehat\Lambda_X\bigr|^2\)\\[-1pt]
  \(+\) packetwise \(\alpha_{b,\mathcal D}\)-admissibility
  of \(\widehat C_k\)
};

\node[roadmap major] (pair-major-1) at (4.0,-4.85) {
  \textbf{Major arcs}\\[-1pt]
  non-resonance and bad-denominator decay\\[-1pt]
  \(+\) the prime major-arc approximation
};

\node[roadmap minor] (pair-minor-2) at (-4.0,-7.35) {
  \textbf{Dyadic packet summation}\\[2pt]
  decisive factor\\[2pt]
  \(\displaystyle
     (D^2B)^{\alpha_{b,\mathcal D}-2/5}
  \)
};

\node[roadmap major] (pair-major-2) at (4.0,-7.35) {
  \textbf{Arithmetic reduction}\\[1pt]
  the surviving denominator \(d\) satisfies \(d\mid b\)
};

\node[roadmap minor] (pair-minor-3) at (-4.0,-9.95) {
  \(\displaystyle \alpha_{b,\mathcal D}<\frac25\)\\[2pt]
  \(\displaystyle
    T_k^{(2)}\!\left(\mathfrak m_A^{(N)}\right)
    =O_A\!\left(|\mathcal C_k|X(\log X)^{-A}\right)
  \)
};

\node[roadmap major] (pair-major-3) at (4.0,-9.95) {
  \textbf{Complete the archimedean \(\eta\)-sum}\\[2pt]
  \(\displaystyle
    \sum_{0\le n<X}
    \mathbf 1_{\mathcal C_k}(n)(X-n)
    \mathfrak S_b^{(2)}(n)
  \)
};

\node[roadmap merge] (pair-merge) at (0,-12.75) {
  \textbf{Average the projected local factor over \(\mathcal C_k\)}\\[-1pt]
  \(\mathfrak S_b^{(2)}(n)\) depends only on the least significant digit,
  while \(X-n\) supplies the normalized archimedean factor; hence\\[2pt]
  \(\displaystyle
     \left(
       1-\frac{\sum_{a\in\mathcal D}a}{(b-1)|\mathcal D|}
     \right)
     \overline{\mathfrak S}^{(2)}_{b,\mathcal D}
  \)
};

\node[roadmap final] (pair-final) at (0,-15.30) {
  \(\displaystyle
  T_k^{(2)}
  =
  \left(
       1-\frac{\sum_{a\in\mathcal D}a}{(b-1)|\mathcal D|}
  \right)
  \overline{\mathfrak S}^{(2)}_{b,\mathcal D}
  X|\mathcal C_k|
  +O_{b,\mathcal D,A}\!\left(|\mathcal C_k|X(\log X)^{-A}\right)
  \)
};

\draw[roadmap arrow] (pair-root) -- (pair-split);
\draw[roadmap arrow] (pair-split) --
  node[roadmap edge label,pos=.45] {minor} (pair-minor-1);
\draw[roadmap arrow] (pair-split) --
  node[roadmap edge label,pos=.45] {major} (pair-major-1);
\draw[roadmap arrow] (pair-minor-1) -- (pair-minor-2);
\draw[roadmap arrow] (pair-minor-2) -- (pair-minor-3);
\draw[roadmap arrow] (pair-major-1) -- (pair-major-2);
\draw[roadmap arrow] (pair-major-2) -- (pair-major-3);
\draw[roadmap arrow] (pair-minor-3) -- (pair-merge);
\draw[roadmap arrow] (pair-major-3) -- (pair-merge);
\draw[roadmap arrow] (pair-merge) -- (pair-final);

\end{tikzpicture}
\caption{Roadmap of the two-term argument.  This is the basic
circle-method skeleton reused in the three-term proof.  The minor arcs combine
prime-side decay with the localized hybrid estimate for \(\widehat C_k\),
while the major arcs reduce to squarefree divisors of the base and then to a
projected local factor in physical space.}
\label{fig:two-term-roadmap}
\end{figure}
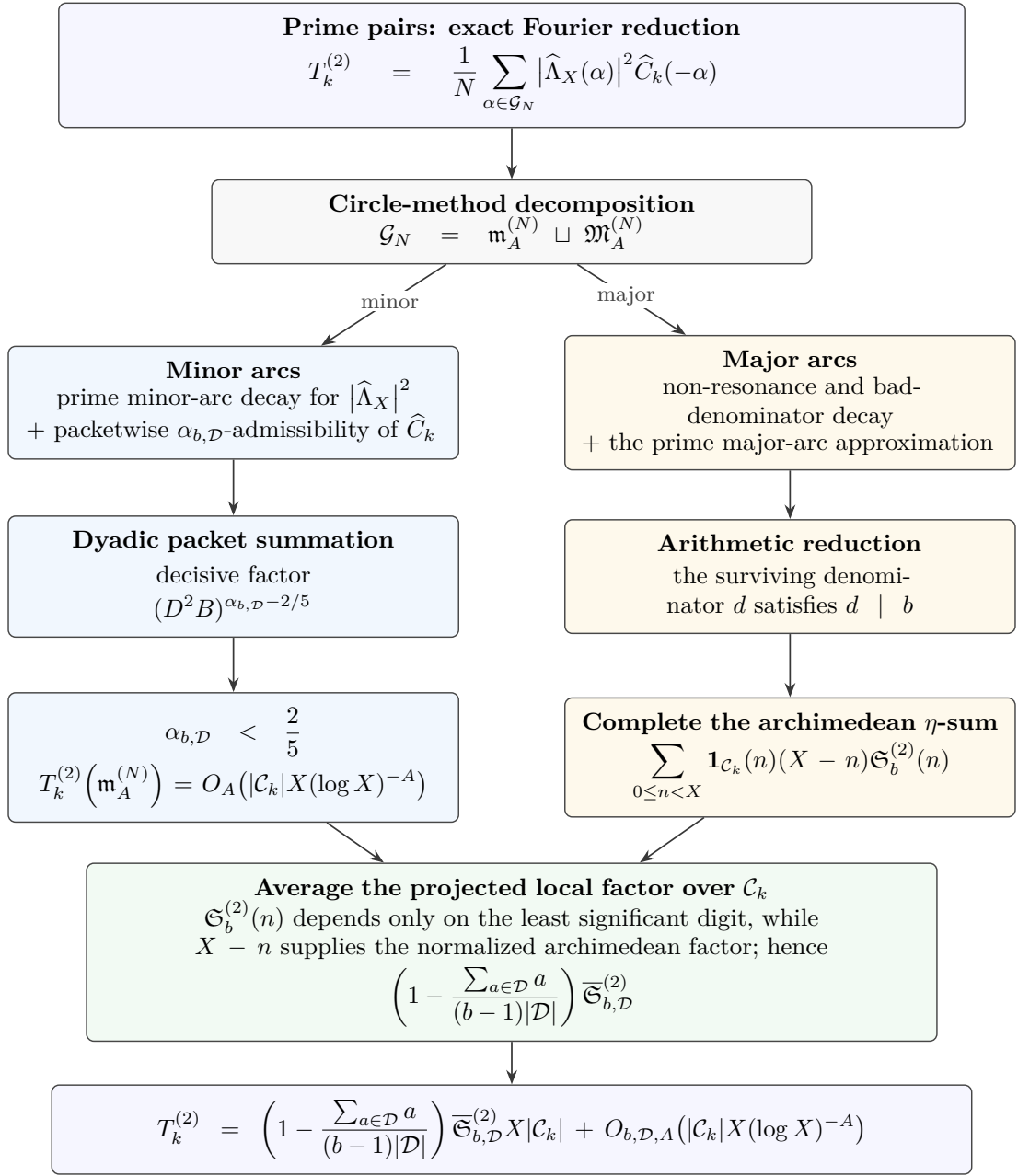

\begin{figure}[p]
\centering
\begin{tikzpicture}[x=1cm,y=1cm]

\node[roadmap root] (three-root) at (0,0) {
  \textbf{Three-term progressions: exact Fourier reduction}\\[2pt]
  \(\displaystyle
    T_k^{(3)}
    =\frac1N\sum_{\alpha\in\mathcal G_N}
      \widehat H_\Lambda(\alpha)\widehat C_k(\alpha)
  \),\\[2pt]
  \(\displaystyle
    \widehat H_\Lambda(\alpha)
    =\frac1N\sum_{\beta\in\mathcal G_N}
      \widehat\Lambda_X(\beta)
      \widehat\Lambda_X(\alpha-2\beta)
      \widehat\Lambda_X(\beta-\alpha)
  \)
};

\node[roadmap split] (three-split) at (0,-2.55) {
  \textbf{Outer circle-method decomposition for $\alpha$}\\[-1pt]
  \(\displaystyle
    \mathcal G_N
    =\mathfrak m_A^{(N)}\sqcup\mathfrak M_A^{(N)}
  \)
};

\node[roadmap minor] (three-minor-1) at (-4.0,-5.15) {
  \textbf{Minor arcs: three prime frequencies}\\[2pt]
  \(\displaystyle
    \beta,\qquad \alpha-2\beta,\qquad \beta-\alpha
  \)
};

\node[roadmap major] (three-major-1) at (4.0,-5.15) {
  \textbf{Major arcs}\\[-1pt]
  discard bad outer denominators\\[-1pt]
  \(+\) split the inner \(\beta\)-sum into major and minor arcs
};

\node[roadmap minor] (three-minor-2) at (-4.0,-8.05) {
  \roadmapkey{Three-point lemma}\\[-1pt]
  one of the three prime sums gains\\[1pt]
  \(\displaystyle X^{-1/5}+(DB)^{-1/3+\epsilon}\)
};

\node[roadmap major] (three-major-2) at (4.0,-7.80) {
  inner minor arcs are negligible;\\[-1pt]
  inner major arcs give \(\widehat H_{\mathrm{good}}\);\\[-1pt]
  the three M\"obius factors force the outer denominator $d$ to be
  squarefree, hence \(d\mid b\)
};

\node[roadmap minor] (three-minor-3) at (-4.0,-11.40) {
  \textbf{Dyadic packet summation}\\[2pt]
  decisive factor\\[2pt]
  \(\displaystyle
  (DB)^{2\alpha_{b,\mathcal D}-1/3+\epsilon}
  \)\\[2pt]
  \(\displaystyle \alpha_{b,\mathcal D}<\frac16\)
  \(\Longrightarrow\) the \linebreak minor arcs are negligible
};

\node[roadmap major] (three-major-3) at (4.0,-11.40) {
  \roadmapkey{Approximate decoupling}\\[1pt]
  \(\displaystyle
    \widehat H_{\mathrm{good}}\!\left(\frac\ell d+\frac\eta N\right)
    =S_{1,\mathrm{full}}\!\left(\frac\ell d\right)
     S_{2,\mathrm{full}}(\eta)
     +\text{error}
  \)\\[2pt]
  completing the outer \(\eta\)-sum gives\\[-1pt]
  \(\displaystyle
    \sum_{0\le n<X/2}
    \mathbf 1_{\mathcal C_k}(n)(X-2n)
    \mathfrak S_b^{(3)}(n)
  \)
};

\node[roadmap merge] (three-merge) at (0,-14.85) {
  \textbf{Average the projected local factor over \(\mathcal C_k\)}\\[-1pt]
  the arithmetic factor depends on the least significant digit, while the
  cutoff \(n<X/2\) and the weight \(X-2n\) determine the archimedean factor;
  hence\\[2pt]
  \(\displaystyle
     \kappa_{b,\mathcal D}\,
     \overline{\mathfrak S}_{b,\mathcal D}^{(3)}
  \)
};

\node[roadmap final] (three-final) at (0,-17.35) {
  \(\displaystyle
  T_k^{(3)}
  =\kappa_{b,\mathcal D}\,
   \overline{\mathfrak S}_{b,\mathcal D}^{(3)}
   X|\mathcal C_k|
  +O_{b,\mathcal D,A}\!\left(|\mathcal C_k|X(\log X)^{-A}\right)
  \)
};

\draw[roadmap arrow] (three-root) -- (three-split);
\draw[roadmap arrow] (three-split) --
  node[roadmap edge label,pos=.45] {minor} (three-minor-1);
\draw[roadmap arrow] (three-split) --
  node[roadmap edge label,pos=.45] {major} (three-major-1);
\draw[roadmap arrow] (three-minor-1) -- (three-minor-2);
\draw[roadmap arrow] (three-minor-2) -- (three-minor-3);
\draw[roadmap arrow] (three-major-1) -- (three-major-2);
\draw[roadmap arrow] (three-major-2) -- (three-major-3);
\draw[roadmap arrow] (three-minor-3) -- (three-merge);
\draw[roadmap arrow] (three-major-3) -- (three-merge);
\draw[roadmap arrow] (three-merge) -- (three-final);

\end{tikzpicture}
\caption{Roadmap of the three-term argument.  Its outer structure is the same
as in \cref{fig:two-term-roadmap}.  The genuinely new inputs are the
three-point lemma on the minor arcs and, on the major arcs, the inner
\(\beta\)-decomposition together with approximate arithmetic--archimedean
decoupling.  The arithmetic coefficient in the projected local factor is
evaluated in \cref{sec:three-term-local-factor}.}
\label{fig:three-term-roadmap}
\end{figure}
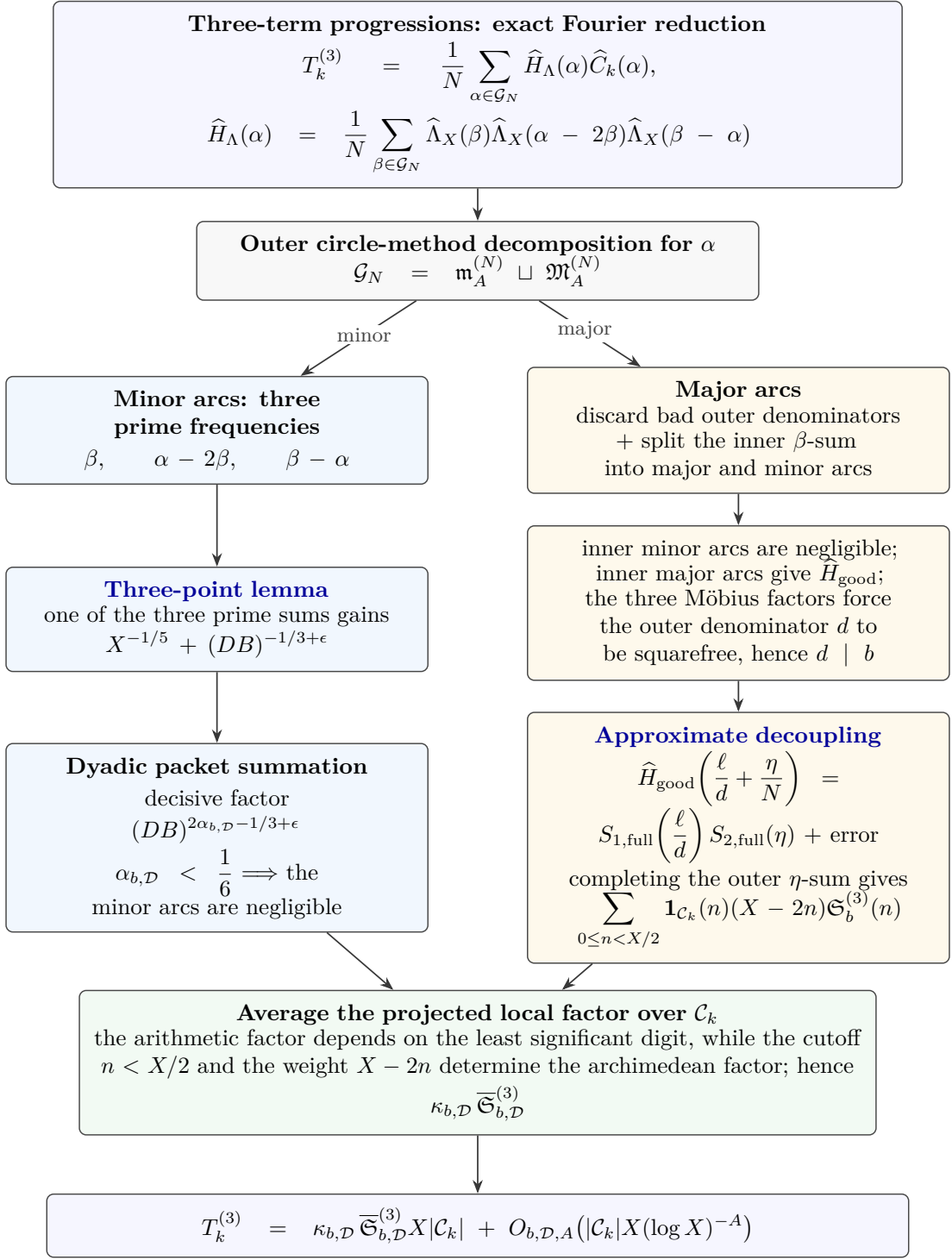
\FloatBarrier

{

We also consider alternative-range variants.  Here only the starting point
\(m\) and the restricted-digit difference \(n\) are required to lie in
\([0,X)\), while the remaining prime variables are allowed to lie in the
larger intervals forced by these choices:
}
\[
        T_{k,\mathrm{alt}}^{(2)}
        :=
        \sum_{0\le n,m<X}\mathbf 1_{\mathcal C_k}(n)
        \Lambda(m)\Lambda(m+n),
\]
and
\[
        T_{k,\mathrm{alt}}^{(3)}
        :=
        \sum_{0\le n,m<X}\mathbf 1_{\mathcal C_k}(n)
        \Lambda(m)\Lambda(m+n)\Lambda(m+2n).
\]
{

For every \(n\in\mathcal C_k\), both alternative-range counts have exactly
\(X\) admissible starting points \(m\).  Hence the normalized archimedean
factor is \(1\), while the arithmetic factors are unchanged.  The resulting
main terms therefore involve only the averaged projected local factors. 
The proofs of the variants are essentially the same as those of the main theorems, so we only sketch the necessary modifications in
\cref{sec:alternative-range}.
}

\begin{theorem}[Prime pairs with restricted-digit difference, alternative range]
\label{thm:intro-two-prime-alternative-range}
Let \(X=b^k\).  {Assume that \(\mathcal D\) is non-resonant
with \(b\) in the sense of \cref{def:digit-nonresonance}, and that}
\(\widehat C_k\) is \(\alpha_{b,\mathcal D}\)-admissible with
\[
        0<\alpha_{b,\mathcal D}<\frac25.
\]
Then, for every fixed \(A>0\),
\[
        T_{k,\mathrm{alt}}^{(2)}
        =
        \overline{\mathfrak S}^{(2)}_{b,\mathcal D}\,X|\mathcal C_k|
        +
        O_{b,\mathcal D,A}\!\left(
        |\mathcal C_k|X(\log X)^{-A}
        \right),
\]
where
\[
        \overline{\mathfrak S}^{(2)}_{b,\mathcal D}
        :=
        \frac{1}{|\mathcal D|}
        \sum_{a\in\mathcal D}
        \sum_{d\mid b}
        \frac{\mu(d)^2}{\phi(d)^2}c_d(a),
\]
with \(c_d(a)\) denoting the Ramanujan sum.  Moreover,
\[
        \overline{\mathfrak S}^{(2)}_{b,\mathcal D}>0
\]
if and only if either \(b\) is odd or \(\mathcal D\) contains an even digit.
\end{theorem}

\begin{theorem}[Three-term progressions, alternative range]
\label{thm:intro-three-prime-alternative-range}
Let \(X=b^k\). {Assume that \(\mathcal D\) is non-resonant
with \(b\) in the sense of \cref{def:digit-nonresonance}, and that}
\(\widehat C_k\) is \(\alpha_{b,\mathcal D}\)-admissible with
\[
        0<\alpha_{b,\mathcal D}<\frac16.
\]
Then, for every fixed \(A>0\),
\[
        T_{k,\mathrm{alt}}^{(3)}
        =
        \overline{\mathfrak S}^{(3)}_{b,\mathcal D}\,X|\mathcal C_k|
        +
        O_{b,\mathcal D,A}\!\left(
        |\mathcal C_k|X(\log X)^{-A}
        \right),
\]
where
\[
        \overline{\mathfrak S}^{(3)}_{b,\mathcal D}
        :=
        \frac{1}{|\mathcal D|}
        \sum_{a\in\mathcal D}\mathfrak S_b^{(3)}(a).
\]
Moreover,
\[
        \overline{\mathfrak S}^{(3)}_{b,\mathcal D}>0
\]
if and only if
\[
        \mathcal D\cap \gcd(b,6)\mathbb Z\neq\emptyset.
\]
\end{theorem}

For each of these four counts, whenever the corresponding leading constant is
positive, configurations in which at least one von Mangoldt argument is a
proper prime power contribute negligibly.  If \(0\in\mathcal C_k\), the
diagonal \(n=0\) is negligible as well; if \(0\notin\mathcal C_k\), it is absent
from the count.  Removing these contributions yields nontrivial configurations
of genuine primes; see \cref{sec:true-prime}.

The paper is organized as follows.  In \cref{sec:prime-pairs} we prove the
prime-pair asymptotic formula.  In \cref{sec:three-term} we prove a
three-point minor-arc lemma and establish the three-term asymptotic formula.
In \cref{sec:alternative-range} we prove the alternative-range variants.
In \cref{sec:true-prime} we remove configurations involving proper prime
powers and, when present, the diagonal, thereby passing from the
von Mangoldt-weighted asymptotic formulae above to
nontrivial configurations of genuine primes.  In
\cref{sec:transfer-model} we derive explicit sufficient bounds for the
admissibility exponent.  The appendices prove the bad-denominator decay
estimate and evaluate the projected three-prime local factor.

\section{Preliminaries}
For a function \(h:\mathbb Z/N\mathbb Z\to\mathbb C\), define its Fourier
transform by
\[
        \widehat h\!\left(\frac{a}{N}\right)
        :=
        \sum_{x\in\mathbb Z/N\mathbb Z}
        h(x)e\!\left(\frac{ax}{N}\right),
        \qquad
        a\in\mathbb Z/N\mathbb Z.
\]
With this convention, Fourier inversion takes the form
\[
        h(x)
        =
        \frac1N
        \sum_{a=0}^{N-1}
        \widehat h\!\left(\frac{a}{N}\right)
        e\!\left(-\frac{ax}{N}\right).
\]

\subsection{Exact Fourier representation for the two-term counting problem}
\label{subsec:fourier-setup-pairs}

If
\[
        \Lambda_X(t)
        :=
        \begin{cases}
        \Lambda(t), & 0\le t<X,\\
        0, & \text{otherwise},
        \end{cases}
\]
then
\[
        T_k^{(2)}
        =
        \sum_{0\le n<X}
        \mathbf 1_{\mathcal C_k}(n)
        \sum_{m\in\mathbb Z}
        \Lambda_X(m)\Lambda_X(m+n).
\]
This form makes the support restriction \(m+n<X\) automatic.

For the two-term problem we take \(N=3X\), and we view \(\Lambda_X\) and
\(\mathbf 1_{\mathcal C_k}\) as functions on \(\mathbb Z/N\mathbb Z\), both
supported in \([0,X)\).  Thus
\[
        \widehat\Lambda_X(\theta)
        :=
        \sum_{0\le m<X}\Lambda(m)e(m\theta),
        \qquad
        \widehat C_k(\theta)
        :=
        \sum_{0\le n<X}\mathbf 1_{\mathcal C_k}(n)e(n\theta).
\]

\begin{proposition}[Exact discrete Fourier representation for prime pairs]
\label{prop:two_fourier-representation}
One has
\begin{align} \label{eq:prime-pair-sum}
        T_k^{(2)}
        =
        \frac1N
        \sum_{a=0}^{N-1}
        \left|
        \widehat\Lambda_X\!\left(\frac{a}{N}\right)
        \right|^2
        \widehat C_k\!\left(-\frac{a}{N}\right).
\end{align}
\end{proposition}

\begin{proof}
Since \(N>2X\), one has the exact cyclic identity
\[
        T_k^{(2)}
        =
        \sum_{m,n\;(\mathrm{mod}\,N)}
        \Lambda_X(m)\Lambda_X(m+n)\mathbf 1_{\mathcal C_k}(n).
\]
Indeed, the factors \(\Lambda_X(m)\) and \(\mathbf 1_{\mathcal C_k}(n)\) force
\(0\le m,n<X\).  Hence \(0\le m+n<2X<N\), so the residue class
\(m+n\pmod N\) is represented by the ordinary integer \(m+n\).  The factor
\(\Lambda_X(m+n)\) is therefore nonzero precisely when \(m+n<X\), which is
exactly the truncation in the definition of \(T_k^{(2)}\).

Applying Fourier inversion to the three factors gives
\begin{align*}
T_k^{(2)}
&=
\frac1{N^3}
\sum_{m,n\;(\mathrm{mod}\,N)}
\sum_{a_1,a_2,a_3=0}^{N-1}
\widehat\Lambda_X\!\left(\frac{a_1}{N}\right)
\widehat\Lambda_X\!\left(\frac{a_2}{N}\right)
\widehat C_k\!\left(\frac{a_3}{N}\right)
\\
&\qquad\qquad\qquad\times
e\!\left(
-\frac{(a_1+a_2)m+(a_2+a_3)n}{N}
\right).
\end{align*}
The sums over \(m\) and \(n\) impose the congruences
\[
        a_1+a_2\equiv0\pmod N,
        \qquad
        a_2+a_3\equiv0\pmod N.
\]
Writing \(a=a_2\), we obtain
\[
        T_k^{(2)}
        =
        \frac1N
        \sum_{a=0}^{N-1}
        \widehat\Lambda_X\!\left(-\frac{a}{N}\right)
        \widehat\Lambda_X\!\left(\frac{a}{N}\right)
        \widehat C_k\!\left(-\frac{a}{N}\right),
\]
and the stated identity follows.
\end{proof}

\subsection{Exact Fourier representation for the three-term counting problem}

For the three-term problem we take \(N=4X\), and again view
\(\mathbf 1_{\mathcal C_k}\) and \(\Lambda_X\) as functions on
\(\mathbb Z/N\mathbb Z\).  Since \(N>3X\), there is no wraparound in the
expressions \(m\), \(m+n\), and \(m+2n\).  Indeed, if \(0\le m,n<X\), then
\[
        0\le m+2n<3X<N,
\]
so the residue classes \(m+n\pmod N\) and \(m+2n\pmod N\) are represented by
the corresponding ordinary integers.

\begin{proposition}[Exact Fourier representation for three-term counting]
\label{prop:fourier-Tk-3AP}
One has
\begin{equation}\label{eq:fourier-Tk-3AP}
        T_k^{(3)}
        =
        \frac1{N^2}
        \sum_{a,b=0}^{N-1}
        \widehat\Lambda_X\!\left(\frac{a}{N}\right)
        \widehat\Lambda_X\!\left(\frac{b-2a}{N}\right)
        \widehat\Lambda_X\!\left(\frac{a-b}{N}\right)
        \widehat C_k\!\left(\frac{b}{N}\right).
\end{equation}
\end{proposition}

\begin{proof}
By Fourier inversion,
\begin{align*}
T_k^{(3)}
&=
\frac1{N^4}
\sum_{m,n\;(\mathrm{mod}\,N)}
\sum_{a_1,a_2,a_3,a_4=0}^{N-1}
\widehat\Lambda_X\!\left(\frac{a_1}{N}\right)
\widehat\Lambda_X\!\left(\frac{a_2}{N}\right)
\widehat\Lambda_X\!\left(\frac{a_3}{N}\right)
\widehat C_k\!\left(\frac{a_4}{N}\right)
\\
&\qquad\qquad\times
e\!\left(
-\frac{(a_1+a_2+a_3)m+(a_2+2a_3+a_4)n}{N}
\right).
\end{align*}
The sums over \(m\) and \(n\) impose
\[
        a_1+a_2+a_3\equiv0\pmod N,
        \qquad
        a_2+2a_3+a_4\equiv0\pmod N.
\]
Writing \(a:=a_1\) and \(b:=a_4\), equivalently
\[
        a_1\equiv a,
        \qquad
        a_2\equiv b-2a,
        \qquad
        a_3\equiv a-b
        \pmod N,
\]
gives \eqref{eq:fourier-Tk-3AP}.
\end{proof}

It is convenient to introduce the prime-side three-term kernel
\begin{equation}\label{eq:def-HLambda}
        \widehat H_\Lambda(\alpha)
        :=
        \frac1N
        \sum_{\beta\in\mathcal G_N}
        \widehat\Lambda_X(\beta)
        \widehat\Lambda_X(\alpha-2\beta)
        \widehat\Lambda_X(\beta-\alpha),
        \qquad
        \alpha\in\mathcal G_N.
\end{equation}
Then \cref{prop:fourier-Tk-3AP} becomes the one-dimensional identity
\begin{equation}\label{eq:Tk3-via-HLambda}
        T_k^{(3)}
        =
        \frac1N
        \sum_{\alpha\in\mathcal G_N}
        \widehat H_\Lambda(\alpha)\widehat C_k(\alpha).
\end{equation}
The restricted-digit set enters only through \(\widehat C_k\), while the prime
variables are encoded {by} \(\widehat H_\Lambda\).

We also record the combinatorial form of \(\widehat H_\Lambda\).

\begin{proposition}[Combinatorial form of \(\widehat H_\Lambda\)]
\label{prop:H-Lambda-combinatorial}
For every \(\alpha\in\mathcal G_N\),
\begin{equation}\label{eq:H-Lambda-combinatorial}
        \widehat H_\Lambda(\alpha)
        =
        \sum_{\substack{0\le m_1,m_2,m_3<X\\ m_1+m_3=2m_2}}
        \Lambda(m_1)\Lambda(m_2)\Lambda(m_3)
        e\!\left((m_2-m_3)\alpha\right).
\end{equation}
\end{proposition}

\begin{proof}
Expanding the definition \eqref{eq:def-HLambda}, we obtain
\begin{align*}
\widehat H_\Lambda(\alpha)
&=
\frac1N
\sum_{\beta\in\mathcal G_N}
\sum_{0\le m_1,m_2,m_3<X}
\Lambda(m_1)\Lambda(m_2)\Lambda(m_3)
\\
&\qquad\qquad\times
e\!\left(
m_1\beta+m_2(\alpha-2\beta)+m_3(\beta-\alpha)
\right)
\\
&=
\sum_{0\le m_1,m_2,m_3<X}
\Lambda(m_1)\Lambda(m_2)\Lambda(m_3)
e\!\left((m_2-m_3)\alpha\right)
\\
&\qquad\qquad\times
\frac1N
\sum_{\beta\in\mathcal G_N}
e\!\left((m_1-2m_2+m_3)\beta\right).
\end{align*}
By orthogonality on \(\mathcal G_N\), the inner sum is \(1\) when
\[
        m_1-2m_2+m_3\equiv0\pmod N
\]
and \(0\) otherwise.  Since \(0\le m_i<X\) and \(N=4X\), one has
\[
        |m_1-2m_2+m_3|<2X<N,
\]
so the congruence is equivalent to the equality \(m_1+m_3=2m_2\).  This gives
\eqref{eq:H-Lambda-combinatorial}.
\end{proof}

\subsection{Major and minor arcs}\label{def:major_minor_arc}

Let \(N\) be a large integer, and let \(D_0\ge1\) be a parameter.  By
Dirichlet's approximation theorem, for each \(\alpha\in\mathbb T\) there is a
reduced fraction \(\ell/d\), interpreted in \(\mathbb R/\mathbb Z\), such that
\[
        1\le d\le D_0,
        \qquad
        1\le\ell\le d,
        \qquad
        (\ell,d)=1,
        \qquad
        \left\|\alpha-\frac{\ell}{d}\right\|_{\mathbb R/\mathbb Z}
        \le\frac1{dD_0}.
\]
For the dyadic decompositions below, choose one such fraction by first
minimizing \(d\) and then choosing the least representative
\(\ell\in\{1,\dots,d\}\).  Then there is a unique
\(\eta\in(-N/2,N/2]\) for which
\[
        \alpha=\frac\ell d+\frac\eta N\pmod1,
        \qquad
        |\eta|\le\frac N{dD_0}.
\]
This convention includes the zero frequency as
\((d,\ell,\eta)=(1,1,0)\).  Throughout the paper, unless otherwise stated, a
sum over primitive numerators means \(1\le\ell\le d\) and
\((\ell,d)=1\).

If \(\alpha\in\mathcal G_N\), then \(\eta\) need not be an integer.  For fixed
\((d,\ell)\), it lies in the shifted lattice
\[
        \eta\in\mathbb Z-\frac{N\ell}{d}.
\]
In particular, \(\eta\in\mathbb Z\) when \(d\mid N\).  All estimates below
are uniform in the selected Dirichlet triple.

Let
\[
        L:=(\log X)^A,
\]
where the cutoff exponent \(A>0\) will be chosen sufficiently large for the
required final logarithmic saving.  We define the major arcs by existence:
\begin{align}\label{def:major}
\mathfrak M_A^{(N)}
:=\left\{\alpha\in\mathcal G_N:
\alpha=\frac\ell d+\frac\eta N\pmod1
\text{ for some }d\le L,\ |\eta|\le L,
\ 1\le\ell\le d,\ (\ell,d)=1\right\}.
\end{align}
The minor arcs are
\begin{align}\label{def:minor}
        \mathfrak m_A^{(N)}:=\mathcal G_N\setminus\mathfrak M_A^{(N)}.
\end{align}
For all sufficiently large \(X\), it is clear that a major-arc representation in the displayed
range is unique.
Also, if \(\alpha\in\mathfrak m_A^{(N)}\), every selected Dirichlet triple has
either \(d>L\) or \(|\eta|>L\).

\subsection{Prime exponential sums on the minor arcs}

We next record the minor-arc estimate for the prime exponential sum
\(\widehat\Lambda_X\).  
We use Maynard's nonzero-shift estimate
\cite[Lemma~4.2]{MaynardRestrictedDigits}, together with the exact-rational
endpoint used in the proof of \cite[Lemma~6.1]{MaynardRestrictedDigits}.

\begin{lemma}[Prime minor-arc decay]\label{lem:minor_Mangoldt}
Assume that
\[
        \theta=\frac{\ell}{d}+\beta,
        \qquad
        (\ell,d)=1,
        \qquad
        0<|\beta|<\frac1{d^2}.
\]
Then
\[
        \widehat\Lambda_X(\theta)
        \ll
        \left(
        X^{4/5}
        +
        \frac{X^{1/2}}{|d\beta|^{1/2}}
        +
        X|d\beta|^{1/2}
        \right)(\log X)^4.
\]
In particular, if \(d\sim D\), \(|\eta|\sim B\), and
\(
        D^2B\ll N,
\)
then
\[
        \left|
        \widehat\Lambda_X\!\left(\frac{\ell}{d}+\frac{\eta}{N}\right)
        \right|
        \ll
        \left(
        X^{4/5}
        +
        \frac{X}{(DB)^{1/2}}
        +
        X^{1/2}(DB)^{1/2}
        \right)(\log X)^4.
\]

At the exact rational point \(\beta=0\), one has
\begin{equation}\label{eq:maynard-prime-sum-bound-rational-two-term}
        \widehat\Lambda_X\!\left(\frac{\ell}{d}\right)
        \ll
        \left(
        X^{4/5}
        +
        \frac{X}{d^{1/2}}
        +
        X^{1/2}d^{1/2}
        \right)(\log X)^4.
\end{equation}
Consequently, by partial summation, if \(d\sim D\) and \(|\eta|\sim0\), then
\begin{equation}\label{eq:maynard-prime-sum-bound-perturbed-two-term}
        \left|
        \widehat\Lambda_X\!\left(\frac{\ell}{d}+\frac{\eta}{N}\right)
        \right|
        \ll
        \left(
        X^{4/5}
        +
        \frac{X}{D^{1/2}}
        +
        X^{1/2}D^{1/2}
        \right)(\log X)^4.
\end{equation}
\end{lemma}

\subsection{Prime exponential sums on the major arcs}

We next record the major-arc approximation for the von Mangoldt exponential
sum.  Since the Siegel--Walfisz input yields an arbitrary power saving in
\(\log X\), we record a stronger form than will ultimately be needed.

\begin{lemma}[Prime major-arc approximation]
\label{lem:prime-major-general}
Uniformly for
\[
        (\ell,d)=1,
        \qquad
        d\le L^2,
        \qquad
        |\eta|\le 3L,
\]
one has
\[
        \widehat\Lambda_X\!\left(\frac{\ell}{d}+\frac{\eta}{N}\right)
        =
        \frac{\mu(d)}{\phi(d)}
        U_X\!\left(\frac{\eta}{N}\right)
        +
        O_A\!\left(X(\log X)^{-4A}\right),
\]
where
\begin{equation}\label{def:UX}
        U_X(t):=\sum_{0\le m<X}e(mt),  
\end{equation}
and
\begin{align}\label{eq:UX_estimate}
|U_X(t)|\leq \min\left(X,\frac{1}{2\|t\|}\right).
\end{align}
\end{lemma}

\begin{proof}
This is the standard Siegel--Walfisz major-arc approximation for
\(\widehat\Lambda_X\); see \cite[Ch.~5]{IwaniecKowalski}.  It holds uniformly
for \(d\le (\log X)^C\) and \(|\eta|\le (\log X)^C\) for any fixed \(C\).  The
stated form follows by taking \(C\) sufficiently large in terms of \(A\).
\end{proof}

\subsection{Decay away from \texorpdfstring{$b$}{b}-smooth denominators}

We conclude the digit-side input with a Maynard-type decay estimate away from
denominators supported on primes dividing \(b\).  There is one small resonance
issue which should be separated from the estimate itself.

{The following remark explains why the digit
non-resonance condition is necessary for decay at denominators containing a
prime factor outside the base.}
\begin{remark}[Digit resonances]
\label{rem:digit-resonances}
Let
\[
        g_{\mathcal D}
        :=
        \gcd\{a-a':a,a'\in\mathcal D\}.
\]
Since \(|\mathcal D|\ge2\), this is a positive integer.  Let
\[
P_{\mathcal D}(x)
=
\sum_{a\in\mathcal D}e(ax).
\]
The points at which
the one-digit polynomial $P_{\mathcal D}(x)$ has maximal modulus are exactly
\[
        \left\{x\in\mathbb T: |P_{\mathcal D}(x)|=|\mathcal D|\right\}
        =
        \left\{x\in\mathbb T: g_{\mathcal D}x\in\mathbb Z\right\}.
\]
Indeed, equality in the triangle inequality occurs precisely when all phases
\(e(ax)\), \(a\in\mathcal D\), are equal, which is equivalent to
\[
        (a-a')x\in\mathbb Z
        \qquad
        \text{for all }a,a'\in\mathcal D.
\]

Thus, if \(g_{\mathcal D}>1\), the digit polynomial can have full modulus at
non-integral rational points.  For example, take \(b=10\) and
\[
        \mathcal D=\{0,3,6,9\}.
\]
Then all digits in \(\mathcal D\) are congruent modulo \(3\), and therefore
\[
        |P_{\mathcal D}(1/3)|=|\mathcal D|=4.
\]
Since \(10\equiv1\pmod 3\), we also have
\[
        |P_{\mathcal D}(10^j/3)|=4
        \qquad(j\ge0).
\]
Consequently,
\[
        |\widehat C_k(1/3)|
        =
        \prod_{j=0}^{k-1}|P_{\mathcal D}(10^j/3)|
        =
        4^k
        =
        |\mathcal C_k|.
\]
Thus there is no decay at the denominator \(3\), even though \(3\nmid b\).

For this reason, the bad-denominator decay below is valid, in the stated
\(b\)-smooth form, under the following non-resonance condition: every prime
divisor of \(g_{\mathcal D}\) already divides the base \(b\).  Equivalently,
\(g_{\mathcal D}\) has no prime factor outside the prime factors of \(b\).
Under this condition, every prime \(p\nmid b\) is genuinely non-resonant for
the digit polynomial, and denominators involving such a prime give decay.
Without this condition, one should instead regard primes dividing
\(g_{\mathcal D}\) as additional good primes.  Closely related congruence
obstructions occur in the residue-class distribution theory of restricted-digit
sets; see \cite{ErdosMauduitSarkozyI,KonyaginMissingDigits,
SaavedraAraya,BurginEtAlIntegerCantor}.
\end{remark}

The following is a slightly adapted form of
\cite[Lemma~5.4]{MaynardRestrictedDigits}; we give the proof in Appendix \ref{app:digit-bad-major-general}.

\begin{lemma}[Decay away from \(b\)-smooth denominators]
\label{lem:digit-bad-major-general}
{Assume that \(\mathcal D\) is non-resonant with \(b\) in the sense of \cref{def:digit-nonresonance}.}  Let
\(A>0\) be fixed.  There exists \(c_{b,\mathcal D,A}>0\) such that the
following holds.  Suppose {\(3X\le N\le4X\)} and
\[
        \theta=\frac{\ell}{d}+\frac{\eta}{N},
        \qquad
        (\ell,d)=1,
        \qquad
        1\le d\le L,
        \qquad
        |\eta|\le L,
        \qquad
        L=(\log X)^A.
\]
If \(d\) has a prime factor not dividing \(b\), then
\[
        |\widehat C_k(\theta)|
        \ll_{b,\mathcal D,A}
        |\mathcal C_k|
        \exp\!\left(
        -c_{b,\mathcal D,A}\frac{\log X}{\log\log X}
        \right).
\]
Equivalently, since \(X=b^k\),
\[
        |\widehat C_k(\theta)|
        \ll_{b,\mathcal D,A}
        |\mathcal C_k|
        \exp\!\left(
        -c_{b,\mathcal D,A}\frac{k}{\log k}
        \right).
\]
\end{lemma}

\section{Prime pairs with restricted-digit difference}
\label{sec:prime-pairs}

For any subset \(\Omega\subset \mathcal G_N\), define
\[
T_k^{(2)}(\Omega)
:=
\frac1N
\sum_{\alpha\in\Omega}
\widehat C_k(-\alpha)
\left|
\widehat\Lambda_X(\alpha)
\right|^2.
\]
Recall the definitions of major and minor arcs in \eqref{def:major} and \eqref{def:minor} and consider 
$$D_0:=X^{1/2}.$$
We decompose the prime pair sum \eqref{eq:prime-pair-sum} into major arc and minor arc sums:
\[
T_k^{(2)}
=
T_k^{(2)}\!\left(\mathfrak M_A^{(N)}\right)
+
T_k^{(2)}\!\left(\mathfrak m_A^{(N)}\right).
\]

\subsection{Minor arcs}
\label{subsec:minor-arcs-pairs}

As in \cite{MaynardRestrictedDigits}, we decompose the minor arcs into dyadic pieces.
For \(B\ge1\), let
\[
\mathcal E^{(2)}(D,B)
:=
\frac1N
\sum_{d\sim D}
\sum_{\substack{1\le\ell\le d\\(\ell,d)=1}}
\sum_{\substack{|\eta|\sim B\\
|\eta|<N/d^2\\
N\ell/d+\eta\in\mathbb Z}}
\left|\widehat C_k\!\left(\frac\ell d+\frac\eta N\right)\right|
\left|\widehat\Lambda_X\!\left(\frac\ell d+\frac\eta N\right)\right|^2,
\]
and let
\[
\mathcal E^{(2)}(D,0)
:=
\frac1N
\sum_{d\sim D}
\sum_{\substack{1\le\ell\le d\\(\ell,d)=1}}
\sum_{\substack{|\eta|\sim0\\
N\ell/d+\eta\in\mathbb Z}}
\left|\widehat C_k\!\left(\frac\ell d+\frac\eta N\right)\right|
\left|\widehat\Lambda_X\!\left(\frac\ell d+\frac\eta N\right)\right|^2.
\]

\begin{proposition}[Dyadic minor-arc estimate]
\label{prop:dyadic-minor-general}
Let \(1\le D\le D_0\), and let \(B\ge1\) satisfy 
$B\leq \frac{N}{DD_0}$.
Assume that \(\widehat C_k\) is
\(\alpha_{b,\mathcal D}\)-admissible.
Then
\begin{equation}
\label{eq:dyadic-minor-nonzero-eta}
\mathcal E^{(2)}(D,B)
\ll
|\mathcal C_k|\,X(\log X)^8
\left(
(D^2B)^{\alpha_{b,\mathcal D}-\frac25}
+
\frac{X^{\alpha_{b,\mathcal D}}}{D_0}
\right),
\end{equation}
and the endpoint range satisfies:
\begin{equation}
\label{eq:dyadic-minor-small-eta}
\mathcal E^{(2)}(D,0)
\ll
|\mathcal C_k|\,X(\log X)^8
\left(
D^{2\alpha_{b,\mathcal D}-\frac45}
+
\frac{D_0^{2\alpha_{b,\mathcal D}+1}}{X}
\right).
\end{equation}
\end{proposition}

\begin{proof}
We first treat the range {\(|\eta|\sim B\)}, with \(B\ge 1\).
\cref{lem:minor_Mangoldt} yields, recall $N=3X$,
\[
\left|
\widehat\Lambda_X\!\left(\frac{\ell}{d}+\frac{\eta}{N}\right)
\right|
\ll
\left(
X^{4/5}
+
\frac{X}{(DB)^{1/2}}
+
X^{1/2}(DB)^{1/2}
\right)(\log X)^4 .
\]
The shell is contained in the cumulative packet \(|\eta|<2B\), so
\(\alpha_{b,\mathcal D}\)-admissibility gives
\begin{align*}
&\mathcal E^{(2)}(D,B)\\
&\ll
\frac{1}{X}
\left(
X^{4/5}
+
\frac{X}{(DB)^{1/2}}
+
X^{1/2}(DB)^{1/2}
\right)^2
(\log X)^8 \cdot\left(\sum_{d\sim D}
\sum_{\substack{1\le\ell\le d\\(\ell,d)=1}}
\sum_{\substack{|\eta|<2B\\N\ell/d+\eta\in\mathbb Z}}
\left|\widehat C_k\!\left(\frac\ell d+\frac\eta N\right)\right|\right)
\\
&\ll
\frac{|\mathcal{C}_k|}{X} (D^2B)^{\alpha_{b,\mathcal{D}}}
\left(
X^{4/5}
+
\frac{X}{(DB)^{1/2}}
+
X^{1/2}(DB)^{1/2}
\right)^2
(\log X)^8 \\
&\ll
|\mathcal C_k|(\log X)^8
\left(
X^{3/5}(D^2B)^{\alpha_{b,\mathcal D}}
+
X\frac{(D^2B)^{\alpha_{b,\mathcal D}}}{DB}
+
(D^2B)^{\alpha_{b,\mathcal D}}DB
\right).
\end{align*}
Since \(D^2B\lesssim X\), the first term is bounded by
\[
X^{3/5}(D^2B)^{\alpha_{b,\mathcal D}}
=
X\cdot X^{-2/5}(D^2B)^{\alpha_{b,\mathcal D}}
\lesssim
X(D^2B)^{\alpha_{b,\mathcal D}-\frac25}.
\]
also, since $B\geq 1,$
\[
X\frac{(D^2B)^{\alpha_{b,\mathcal D}}}{DB}
\le
X(D^2B)^{\alpha_{b,\mathcal D}-\frac12}
\le
X(D^2B)^{\alpha_{b,\mathcal D}-\frac25}.
\]
Finally, from \(B\leq N/(DD_0)\) and \(N=3X\), we have $
DB\lesssim X/D_0$, and hence
\[
(D^2B)^{\alpha_{b,\mathcal D}}DB
\lesssim
\frac{X}{D_0}(D^2B)^{\alpha_{b,\mathcal D}}
\lesssim
\frac{X^{1+\alpha_{b,\mathcal D}}}{D_0}.
\]
This proves \eqref{eq:dyadic-minor-nonzero-eta}.

It remains to treat \(|\eta|\sim0\). In this endpoint range, \cref{lem:minor_Mangoldt} gives
\[
\left|
\widehat\Lambda_X\!\left(\frac{\ell}{d}+\frac{\eta}{N}\right)
\right|
\ll
\left(
X^{4/5}
+
\frac{X}{D^{1/2}}
+
X^{1/2}D^{1/2}
\right)(\log X)^4 ,
\]
uniformly for \(d\sim D\) and \(|\eta|\sim0\).  Applying
\(\alpha_{b,\mathcal D}\)-admissibility, we get
\begin{align*}
\mathcal E^{(2)}(D,0)
&\ll
\frac{|\mathcal C_k|}{X}
D^{2\alpha_{b,\mathcal D}}
\left(
X^{4/5}
+
\frac{X}{D^{1/2}}
+
X^{1/2}D^{1/2}
\right)^2
(\log X)^8
\\
&\ll
|\mathcal C_k|(\log X)^8
\left(
X^{3/5}D^{2\alpha_{b,\mathcal D}}
+
XD^{2\alpha_{b,\mathcal D}-1}
+
D^{2\alpha_{b,\mathcal D}+1}
\right).
\end{align*}
Since \(D\le D_0\le X^{1/2}\), we have 
\[
X^{3/5}D^{2\alpha_{b,\mathcal D}}
=
X\cdot X^{-2/5}D^{2\alpha_{b,\mathcal D}}
\le
X D^{2\alpha_{b,\mathcal D}-\frac45}.
\]
Also, since $D\geq 1$,
\[
XD^{2\alpha_{b,\mathcal D}-1}
\le
X D^{2\alpha_{b,\mathcal D}-\frac45}.
\]
Finally,
\[
D^{2\alpha_{b,\mathcal D}+1}
\le
D_0^{2\alpha_{b,\mathcal D}+1}
=
X\frac{D_0^{2\alpha_{b,\mathcal D}+1}}{X}.
\]
This proves \eqref{eq:dyadic-minor-small-eta}.
\end{proof}

\begin{proposition}[Total minor-arc contribution]
\label{prop:minor-total-general}
Assume that
\(
\alpha_{b,\mathcal D}<\frac25.
\)
Then, after choosing \(A\) sufficiently large,
\[
T_k^{(2)}\!\left(\mathfrak m_A^{(N)}\right)
\ll_{b,\mathcal D,A}
|\mathcal C_k|\,X(\log X)^{-A}.
\]
\end{proposition}

\begin{proof}
For each \(\theta\in\mathfrak m_A^{(N)}\), consider the Dirichlet approximation
\[
\theta=\frac{\ell}{d}+\frac{\eta}{N},
\qquad
(\ell,d)=1,
\qquad
1\le d\le D_0,
\qquad
|\eta|\leq \frac{N}{dD_0}.
\]
We decompose these chosen approximations into dyadic ranges \(d\sim D\) and
either \(|\eta|\sim B\), with \(B\ge1\), or \(|\eta|\sim0\).
The number of dyadic parameters is at most \(O((\log X)^2)\).

Since \(\theta\) belongs to the minor arcs, it has no representation with
\(d\le L\) and \(|\eta|\le L\).  Hence, for the chosen approximation, either
\(D\gg L\) or \(B\gg L\) in the range \(|\eta|\sim B\).  Therefore
\[
D^2B\gg L.
\]

Therefore 
\begin{align}
    T_k^{(2)}\!\left(\mathfrak m_A^{(N)}\right)
\le
\sum_{D^2B\gg L}\mathcal E^{(2)}(D,B)
+
\sum_{D\gg L}\mathcal E^{(2)}(D,0).
\end{align}

For each $\mathcal{E}^{(2)}(D,B)$ with $D^2B\gg L$: since \(\alpha_{b,\mathcal D}<2/5\), the first term in
\eqref{eq:dyadic-minor-nonzero-eta} satisfies
\[
(D^2B)^{\alpha_{b,\mathcal D}-\frac25}
\ll
(\log X)^{-A(2/5-\alpha_{b,\mathcal D})}.
\]
The second term in \eqref{eq:dyadic-minor-nonzero-eta} gives
\[
\frac{X^{\alpha_{b,\mathcal D}}}{D_0}
=
X^{\alpha_{b,\mathcal D}-\frac12}.
\]
Since \(\alpha_{b,\mathcal D}<2/5\), this is a fixed negative power of \(X\),
and hence is \(O_A((\log X)^{-A})\) for every fixed \(A>0\).

In the endpoint range \(|\eta|\sim0\), the minor-arc condition forces
\(D\gg L\).  Therefore
\[
D^{2\alpha_{b,\mathcal D}-\frac45}
\ll
(\log X)^{-A(4/5-2\alpha_{b,\mathcal D})}.
\]
Also,
\[
\frac{D_0^{2\alpha_{b,\mathcal D}+1}}{X}
=
X^{\alpha_{b,\mathcal D}-\frac12}
=
O_A((\log X)^{-A}).
\]
Summing over \(O((\log X)^2)\) dyadic blocks gives
\[
T_k^{(2)}\!\left(\mathfrak m_A^{(N)}\right)
\ll
|\mathcal C_k|\,X(\log X)^{10}
\left(
(\log X)^{-A(2/5-\alpha_{b,\mathcal D})}
+
O_A((\log X)^{-A-10})
\right).
\]
Increasing the logarithmic parameter \(A\) if necessary completes the proof.
\end{proof}

\subsection{Major arcs and evaluation of the main term}
\label{subsec:major-arcs-pairs}

We now evaluate the major-arc contribution.  The main point is that the
restricted-digit Fourier transform is large only near rationals whose
denominators are supported on the prime factors of the base \(b\).  
We split the major arcs according to \cref{lem:digit-bad-major-general}.  Define the good major arcs by
\[
\mathfrak M_{A,\mathrm{good}}^{(N)}
:=
\left\{
\theta=\frac{\ell}{d}+\frac{\eta}{N}\in\mathfrak M_A^{(N)}:
p\mid d \Longrightarrow p\mid b
\right\},
\]
where $p$ represents a prime number,
and set
\[
\mathfrak M_{A,\mathrm{bad}}^{(N)}
:=
\mathfrak M_A^{(N)}\setminus \mathfrak M_{A,\mathrm{good}}^{(N)}.
\]
For \(X\) sufficiently large, the major arcs are disjoint, since \(L^3\ll N\).
Hence each major-arc point has a unique representation
\[
        \theta=\frac{\ell}{d}+\frac{\eta}{N},
        \qquad
        (\ell,d)=1,
        \qquad
        d\le L,
        \qquad
        |\eta|\le L, 
\]
with $N\ell/d+\eta\in\mathbb Z$.

\begin{lemma}[Bad major arcs]
\label{lem:bad-major-general}
There exists \(c_{b,\mathcal D,A}>0\) such that
\[
T_k^{(2)}\!\left(\mathfrak M_{A,\mathrm{bad}}^{(N)}\right)
\ll_{b,\mathcal D,A}
|\mathcal C_k|\,X\,\exp\!\left(
        -c_{b,\mathcal D,A}\frac{\log X}{\log\log X}
        \right).
\]
\end{lemma}

\begin{proof}
On \(\mathfrak M_{A,\mathrm{bad}}^{(N)}\), \cref{lem:digit-bad-major-general}
gives
\[
        |\widehat C_k(-\theta)|
        \ll_{b,\mathcal D,A}
        |\mathcal C_k|\exp\!\left(
        -c_{b,\mathcal D,A}\frac{\log X}{\log\log X}
        \right).
\]
Using the trivial estimate
\[
        |\widehat\Lambda_X(\theta)|\ll X
\]
and the fact that there are \(O(L^3)\) admissible triples \((d,\ell,\eta)\), we
obtain
\[
T_k^{(2)}\!\left(\mathfrak M_{A,\mathrm{bad}}^{(N)}\right)
\ll
\frac1N
L^3
|\mathcal C_k|\exp\!\left(
        -c_{b,\mathcal D,A}\frac{\log X}{\log\log X}
        \right)
X^2.
\]
Since \(N\asymp X\) and \(L=(\log X)^A\), the factor \(L^3\) is absorbed into
the exponential saving.
\end{proof}

\begin{proposition}[Major-arc reduction]
\label{prop:major-reduction-general}
One has
\begin{align}
T_k^{(2)}\!\left(\mathfrak M_A^{(N)}\right)
&=
\frac1N
\sum_{d\mid b}
\frac{\mu(d)^2}{\phi(d)^2}
\sum_{\substack{0\le \ell<d\\(\ell,d)=1}}
\sum_{\substack{|\eta|\le L \\ \eta \in \mathbb Z}}
\widehat C_k\!\left(-\frac{\ell}{d}-\frac{\eta}{N}\right)
\left|
U_X\!\left(\frac{\eta}{N}\right)
\right|^2
\notag\\
&\qquad
+
O_{b,\mathcal D,A}\!\left(|\mathcal C_k|X(\log X)^{-A}\right).
\label{eq:major-reduction-general}
\end{align}
\end{proposition}

\begin{proof}
The contribution of \(\mathfrak M_{A,\mathrm{bad}}^{(N)}\) is handled by
\cref{lem:bad-major-general}, and its sub-exponential saving is stronger than
any fixed power of \(\log X\).

On \(\mathfrak M_{A,\mathrm{good}}^{(N)}\), apply
\cref{lem:prime-major-general} to get
\[
        |\widehat\Lambda_X(\frac{\ell}{d}+\frac{\eta}{N})|^2
        =
        \frac{\mu(d)^2}{\phi(d)^2}
        \left|U_X\!\left(\frac{\eta}{N}\right)\right|^2
        +
        O_A\!\left(X^2(\log X)^{-4A}\right).
\]
Consequently, the total contribution of the error term
is then
\[
\ll
\frac1N
(\sum_{d\le L}
\sum_{\substack{(\ell,d)=1}}
\sum_{\substack{|\eta|\le L \\ N\ell/d+\eta\in\mathbb Z}}
|\widehat C_k(-\ell/d-\eta/N)|)\cdot 
X^2(\log X)^{-4A}.
\]
Using the trivial bound,
\[
        |\widehat C_k(\theta)|\le |\mathcal C_k|
\]
and counting \(O(L^3)\) triples \((d,\ell,\eta)\), this is
\[
        \ll
        |\mathcal C_k|X(\log X)^{-4A}L^3
        =
        |\mathcal C_k|X(\log X)^{-A},
\]
since \(L=(\log X)^A\).

It remains to simplify the denominator condition in the main term.  On the good
major arcs, every prime divisor of \(d\) divides \(b\).  Moreover, the main
coefficient contains the factor \(\mu(d)^2\), so only squarefree \(d\) contribute.
Hence the nonzero contribution has
\[
        d\mid b.
\]
For \(X\) sufficiently large, all such divisors satisfy \(d\le L\).
Finally, since \(d\mid b\), \(X=b^k\), and \(N=3X\), one has
\(d\mid N\).  Hence the condition \(N\ell/d+\eta\in\mathbb Z\) implies that
\(\eta\in\mathbb Z\).  This gives the displayed main term.
\end{proof}

Define the projected two-prime local factor by
\begin{equation}\label{def:local-factor-Sb}
        \mathfrak S_b^{(2)}(n)
        :=
        \sum_{d\mid b}
        \frac{\mu(d)^2}{\phi(d)^2}c_d(n),
\end{equation}
where $c_d(n)$
is the Ramanujan sum defined in
\eqref{eq:ramanujan-sum}.  The factor
\(\mathfrak S_b^{(2)}(n)\) is the \(b\)-local part of the usual
Hardy--Littlewood prime-pair singular series, whose Ramanujan-sum expansion, {for \(n\ne0\),} is
\[
        \mathfrak S^{(2)}(n)
        =
        \sum_{q\ge1}
        \frac{\mu(q)^2}{\phi(q)^2}c_q(n).
\]
See, for example, \cite[Ch.~17]{Davenport} and
\cite[Ch.~17]{MontgomeryVaughan}.

Indeed,  \(\mathfrak S_b^{(2)}(n)\) retains only the squarefree moduli \(d\) whose
prime factors divide \(b\).  Equivalently, since \(\mu(d)^2\) restricts the
sum to squarefree \(d\), one may view \(\mathfrak S_b^{(2)}(n)\) as the projection
of the singular series onto the squarefree divisors of \(b\). Additionally, \(\mathfrak S_b^{(2)}(n)\) is only dependent on the least significant digit of $n$ in base $b.$

\begin{proposition}[Evaluation of the major arcs]
\label{prop:major-main-general}
One has
\[
T_k^{(2)}\!\left(\mathfrak M_A^{(N)}\right)
=
\sum_{0\le n<X}
\mathbf 1_{\mathcal C_k}(n)(X-n)\mathfrak S_b^{(2)}(n)
+
O_{b,\mathcal D,A}\!\left(|\mathcal C_k|X(\log X)^{-A}\right).
\]
\end{proposition}

\begin{proof}
Starting from \cref{prop:major-reduction-general}, we extend the
\(\eta\)-sum in the main term from \(|\eta|\le L\) to all residue classes modulo
\(N\).  Since the main term now involves only the finite set \(d\mid b\), the
omitted tail is easy to control.  

For each fixed \(d\mid b\) and \(\ell\),

\[
\begin{aligned}
\frac1N
\sum_{L<|\eta|\le N/2}
\left|
\widehat C_k\!\left(-\frac{\ell}{d}-\frac{\eta}{N}\right)
\right|
\left|U_X\!\left(\frac{\eta}{N}\right)\right|^2       
\ll
\frac{|\mathcal C_k|}{N}
\sum_{|\eta|>L}\frac{N^2}{\eta^2}
\ll
|\mathcal C_k|\frac{N}{L}
\ll
|\mathcal C_k|\frac{X}{L},
\end{aligned}
\]
where we used \eqref{eq:UX_estimate} to estimate $U_X$.
After summing over the \(O_b(1)\) possible \(d,\ell\), this contributes
\[
        O_b(|\mathcal C_k|X/L)
        =
        O_{b,\mathcal D,A}\!\left(|\mathcal C_k|X(\log X)^{-A}\right).
\]

Thus
\begin{align*}
T_k^{(2)}\!\left(\mathfrak M_A^{(N)}\right)
&=
\frac1N
\sum_{d\mid b}
\frac{\mu(d)^2}{\phi(d)^2}
\sum_{\substack{0\le \ell<d\\(\ell,d)=1}}
\sum_{\eta\;(\mathrm{mod}\,N)}
\widehat C_k\!\left(-\frac{\ell}{d}-\frac{\eta}{N}\right)
\left|U_X\!\left(\frac{\eta}{N}\right)\right|^2
\\
&\qquad
+
O_{b,\mathcal D,A}\!\left(|\mathcal C_k|X(\log X)^{-A}\right).
\end{align*}

Now
\[
        |U_X(\eta/N)|^2
        =
        \sum_{h=-(X-1)}^{X-1}
        (X-|h|)
        e\!\left(\frac{h\eta}{N}\right),
\]
and
\[
        \widehat C_k\!\left(-\frac{\ell}{d}-\frac{\eta}{N}\right)
        =
        \sum_{0\le n<X}
        \mathbf 1_{\mathcal C_k}(n)
        e\!\left(-\frac{\ell n}{d}\right)
        e\!\left(-\frac{\eta n}{N}\right).
\]
Orthogonality in the \(\eta\)-variable gives
\[
\frac1N
\sum_{\eta\;(\mathrm{mod}\,N)}
\widehat C_k\!\left(-\frac{\ell}{d}-\frac{\eta}{N}\right)
|U_X(\eta/N)|^2
=
\sum_{0\le n<X}
\mathbf 1_{\mathcal C_k}(n)
e\!\left(-\frac{\ell n}{d}\right)(X-n).
\]
Indeed, the congruence \(h\equiv n\pmod N\), with
\[
        |h|<X,\qquad 0\le n<X,\qquad N>2X,
\]
forces \(h=n\).  Summing over \(d\mid b\) and \(\ell\) gives the desired formula.
\end{proof}

Combining \cref{prop:major-main-general} with the minor-arc estimate
from \cref{prop:minor-total-general}, we obtain the following global
formula.

\begin{theorem}[Major and minor arcs]
\label{thm:major-plus-minor-general}
Assume
\[
        0<\alpha_{b,\mathcal D}<\frac25.
\]
Then, after choosing \(A\) sufficiently large,
\[
T_k^{(2)}
=
\sum_{0\le n<X}
\mathbf 1_{\mathcal C_k}(n)(X-n)\mathfrak S^{(2)}_b(n)
+
O_{b,\mathcal D,A}\!\left(
|\mathcal C_k|X(\log X)^{-A}
\right).
\]
\end{theorem}

\subsection{Averaging the local factor}
\label{subsec:averaging-local-factor-general}

{It remains to average the arithmetic local factor over \(\mathcal C_k\).}  Since
\(d\mid b\), the residue class of \(n\) modulo \(d\) is determined by the
final base-\(b\) digit of \(n\).  Thus this computation is finite and
explicit.

Define
\[
        \overline{\mathfrak S}^{(2)}_{b,\mathcal D}
        :=
        \frac{1}{|\mathcal D|}
        \sum_{a\in\mathcal D}\mathfrak S^{(2)}_b(a)
        =
        \sum_{d\mid b}
        \frac{\mu(d)^2}{\phi(d)^2}
        \frac{1}{|\mathcal D|}
        \sum_{a\in\mathcal D}c_d(a).
\]
{Thus \(\overline{\mathfrak S}^{(2)}_{b,\mathcal D}\) is the
averaged arithmetic factor over the final digit.}  We also write
\[
        S_{\mathcal D}:=\sum_{a\in\mathcal D}a.
\]

\begin{proposition}[Weighted average of the local factor]
\label{prop:weighted-main-term-general}
One has
\[
\sum_{0\le n<X}
\mathbf 1_{\mathcal C_k}(n)(X-n)\mathfrak S^{(2)}_b(n)
=
\left(
1-\frac{S_{\mathcal D}}{(b-1)|\mathcal D|}
\right)
\overline{\mathfrak S}^{(2)}_{b,\mathcal D}\,
X|\mathcal C_k|
+
O_{b,\mathcal D}(|\mathcal C_k|).
\]
Moreover, the leading coefficient
\[
        \left(
        1-\frac{S_{\mathcal D}}{(b-1)|\mathcal D|}
        \right)
        \overline{\mathfrak S}^{(2)}_{b,\mathcal D}
\]
is positive if and only if either \(b\) is odd or
\(\mathcal D\) contains an even digit.
\end{proposition}

\begin{proof}
Write
\[
        n=a+bm,
        \qquad
        a\in\mathcal D,
        \qquad
        m\in\mathcal C_{k-1}.
\]
Since \(d\mid b\), the value of \(c_d(n)\), and hence of
\(\mathfrak S_b^{(2)}(n)\), depends only on the final digit \(a\).  Thus
\[
        \mathfrak S_b^{(2)}(n)=\mathfrak S_b^{(2)}(a).
\]
It follows that
\[
\begin{aligned}
&\sum_{0\le n<X}
\mathbf 1_{\mathcal C_k}(n)(X-n)\mathfrak S^{(2)}_b(n)       \\
&\qquad
=
\sum_{a\in\mathcal D}\mathfrak S_b^{(2)}(a)
\sum_{m\in\mathcal C_{k-1}}(X-a-bm)                          \\
&\qquad
=
|\mathcal C_{k-1}|
\sum_{a\in\mathcal D}(X-a)\mathfrak S^{(2)}_b(a)
-
b\left(
\sum_{a\in\mathcal D}\mathfrak S^{(2)}_b(a)
\right)
\left(
\sum_{m\in\mathcal C_{k-1}}m
\right).
\end{aligned}
\]
The average digit is \(S_{\mathcal D}/|\mathcal D|\), and hence
\[
        \sum_{m\in\mathcal C_{k-1}}m
        =
        |\mathcal C_{k-1}|
        \frac{S_{\mathcal D}}{|\mathcal D|}
        \frac{b^{k-1}-1}{b-1}.
\]
Using
\[
        b(b^{k-1}-1)=X-b,
\]
we obtain the exact identity
\begin{align}
&\sum_{0\le n<X}
\mathbf 1_{\mathcal C_k}(n)(X-n)\mathfrak S^{(2)}_b(n)
\notag\\
&\qquad
=
|\mathcal C_{k-1}|
\sum_{a\in\mathcal D}
\left(
X-a-
\frac{S_{\mathcal D}}{|\mathcal D|}
\frac{X-b}{b-1}
\right)
\mathfrak S_b^{(2)}(a).
\label{eq:weighted-main-term-exact-general}
\end{align}

Since
\[
        \sum_{a\in\mathcal D}\mathfrak S^{(2)}_b(a)
        =
        |\mathcal D|\,
        \overline{\mathfrak S}^{(2)}_{b,\mathcal D},
        \qquad
        |\mathcal C_k|
        =
        |\mathcal D|\,|\mathcal C_{k-1}|,
\]
the leading term in
\eqref{eq:weighted-main-term-exact-general} is
\[
        \left(
        1-\frac{S_{\mathcal D}}{(b-1)|\mathcal D|}
        \right)
        \overline{\mathfrak S}^{(2)}_{b,\mathcal D}\,
        X|\mathcal C_k|.
\]
All remaining terms depend only on \(b\) and \(\mathcal D\), and hence
contribute
\[
        O_{b,\mathcal D}(|\mathcal C_k|).
\]
This proves the asymptotic formula.

It remains to determine when the leading coefficient is positive.  Since every
digit lies in \(\{0,1,\dots,b-1\}\),
\[
        1-\frac{S_{\mathcal D}}{(b-1)|\mathcal D|}
        \ge0.
\]
Equality could hold only if every digit in \(\mathcal D\) were equal to
\(b-1\).  This is excluded by the standing assumption
\(|\mathcal D|\ge2\).  Hence the {normalized archimedean factor} is strictly positive,
and it remains only to determine when
\(\overline{\mathfrak S}^{(2)}_{b,\mathcal D}\) is positive.

By multiplicativity,
\[
        \mathfrak S^{(2)}_b(n)
        =
        \prod_{p\mid b}
        \left(
        1+\frac{c_p(n)}{(p-1)^2}
        \right).
\]
For \(p\mid b\),
\[
        c_p(n)
        =
        \begin{cases}
        p-1, & p\mid n,\\
        -1, & p\nmid n.
        \end{cases}
\]
Therefore
\[
        1+\frac{c_p(n)}{(p-1)^2}
        =
        \begin{cases}
        \dfrac{p}{p-1}, & p\mid n,\\[2ex]
        \dfrac{p(p-2)}{(p-1)^2}, & p\nmid n.
        \end{cases}
\]
For \(p>2\), both values are strictly positive.  For \(p=2\), the
corresponding factor is
\[
        \begin{cases}
        2, & 2\mid n,\\
        0, & 2\nmid n.
        \end{cases}
\]
Consequently,
\[
        \mathfrak S^{(2)}_b(n)\ge0
        \qquad(n\in\mathbb Z),
\]
and it can vanish only when \(2\mid b\) and \(n\) is odd.  Hence
\[
        \overline{\mathfrak S}^{(2)}_{b,\mathcal D}
        =
        \frac{1}{|\mathcal D|}
        \sum_{a\in\mathcal D}\mathfrak S^{(2)}_b(a)
        \ge0.
\]
Moreover,
\[
        \overline{\mathfrak S}^{(2)}_{b,\mathcal D}>0
\]
if and only if at least one digit \(a\in\mathcal D\) satisfies
\[
        \mathfrak S^{(2)}_b(a)>0.
\]
By the preceding discussion, this is equivalent to saying that either \(b\)
is odd, or \(b\) is even and \(\mathcal D\) contains an even digit.  Thus
\[
        \overline{\mathfrak S}^{(2)}_{b,\mathcal D}>0
        \qquad\Longleftrightarrow\qquad
        b\text{ is odd or }\mathcal D\text{ contains an even digit}.
\]
This proves the proposition.
\end{proof}

We are now in a position to complete the proof of
\cref{thm:intro-two-prime}.

\begin{proof}[Proof of Theorem~\ref{thm:intro-two-prime}]
By \cref{thm:major-plus-minor-general},
\[
        T_k^{(2)}
        =
        \sum_{0\le n<X}
        \mathbf 1_{\mathcal C_k}(n)
        (X-n)\mathfrak S_b^{(2)}(n)
        +
        O_{b,\mathcal D,A}\!\left(
        |\mathcal C_k|X(\log X)^{-A}
        \right).
\]
Applying \cref{prop:weighted-main-term-general} to the main term gives
\[
\begin{aligned}
        T_k^{(2)}
        &=
        \left(
        1-\frac{S_{\mathcal D}}{(b-1)|\mathcal D|}
        \right)
        \overline{\mathfrak S}^{(2)}_{b,\mathcal D}\,
        X|\mathcal C_k| 
        +
        O_{b,\mathcal D,A}\!\left(
        |\mathcal C_k|X(\log X)^{-A}
        \right),
\end{aligned}
\]
where the error \(O_{b,\mathcal D}(|\mathcal C_k|)\) from
\cref{prop:weighted-main-term-general} has been absorbed into the displayed
error term.

The positivity criterion follows from
\cref{prop:weighted-main-term-general}.  Under this criterion the leading
coefficient is a fixed positive constant depending only on
\(b\) and \(\mathcal D\), while the error is
\(o_{b,\mathcal D}(X|\mathcal C_k|)\).  Therefore
\[
        T_k^{(2)}
        \asymp_{b,\mathcal D}
        X|\mathcal C_k|
\]
for all sufficiently large \(k\).
\end{proof}

\begin{remark}[The missing last digit case]
\label{rem:missing-last-digit-constant}
Suppose that
\[
        \mathcal D=\{0,1,\dots,b-2\},
\]
so that the only missing digit is \(a_0=b-1\).  Then
\(|\mathcal D|=b-1\), and the averaged local factor is
\[
        \overline{\mathfrak S}^{(2)}_{b,\mathcal D}
        =
        1-\frac{1}{b-1}
        \sum_{\substack{d\mid b\\ d>1}}
        \frac{\mu(d)}{\phi(d)^2}.
\]
Indeed, for \(d>1\) with \(d\mid b\), the complete average of \(c_d(a)\)
over \(a\in\{0,1,\dots,b-1\}\) vanishes.  Removing the digit \(b-1\)
therefore leaves
\[
        \sum_{a=0}^{b-2}c_d(a)
        =
        -c_d(b-1)
        =
        -c_d(-1)
        =
        -\mu(d),
\]
and hence
\[
        \frac{1}{b-1}
        \sum_{a=0}^{b-2}c_d(a)
        =
        -\frac{\mu(d)}{b-1}.
\]

The {normalized archimedean factor} is
\[
\begin{aligned}
        1-\frac{S_{\mathcal D}}{(b-1)|\mathcal D|}
        &=
        1-\frac{0+1+\cdots+(b-2)}{(b-1)^2}  \\
        &=
        1-\frac{b-2}{2(b-1)}
        =
        \frac{b}{2(b-1)}.
\end{aligned}
\]
Therefore the leading coefficient in the weighted pair count is
\[
        \left(
        1-\frac{S_{\mathcal D}}{(b-1)|\mathcal D|}
        \right)
        \overline{\mathfrak S}^{(2)}_{b,\mathcal D}
        =
        \frac{b}{2(b-1)}
        \left(
        1-\frac{1}{b-1}
        \sum_{\substack{d\mid b\\ d>1}}
        \frac{\mu(d)}{\phi(d)^2}
        \right).
\]

For example, when \(b=10\) and
\(\mathcal D=\{0,1,\dots,8\}\), the divisors \(d>1\) of \(10\) are
\(2,5,10\), and
\[
        \sum_{\substack{d\mid10\\d>1}}
        \frac{\mu(d)}{\phi(d)^2}
        =
        -1-\frac1{16}+\frac1{16}
        =
        -1.
\]
Thus the leading coefficient is
\[
        \frac{10}{18}\left(1+\frac19\right)
        =
        \frac{50}{81}
        \approx0.62.
\]
\end{remark}

\section{Three-term progressions in primes with restricted-digit common difference}
\label{sec:three-term}

\subsection{Minor arcs for the three-term problem}
\label{subsec:three-term-minor-arcs}

We use the same major/minor arcs as in {\eqref{def:major} and \eqref{def:minor}}. Take
\[
        D_0:=X^{1/2}.
\]
For every \(\alpha\in\mathcal G_N\), fix one Dirichlet approximation
\[
        \alpha=\frac{\ell}{d}+\frac{\eta}{N},
        \qquad
        (\ell,d)=1,
        \qquad
        1\le d\le D_0,
        \qquad
        |\eta|\leq \frac{N}{dD_0}.
\]
The major arcs \(\mathfrak M_A^{(N)}\) are defined by the existence of such a
representation with
\[
        d\le L,
        \qquad
        |\eta|\le L,
        \qquad
        L:=(\log X)^A,
\]
and the minor arcs are
\[
        \mathfrak m_A^{(N)}
        :=
        \mathcal G_N\setminus\mathfrak M_A^{(N)}.
\]
For dyadic analysis, we use the fixed approximations above and decompose
\(\mathfrak m_A^{(N)}\) into packets
\(\mathfrak m_A^{(N)}(D,B)\), where
\[
        d\sim D,
\qquad
\max(1,|\eta|)\sim B.
\]
On every nonempty packet,
\[
        1\le D\le D_0,
        \qquad
        B\ge1,
        \qquad
        \max(D,B)\gg L.
\]
Moreover,
\begin{equation}\label{eq:DB_D2B<<}
        DB\ll X^{1/2},
        \qquad
        D^2(2B)\le 8X.
\end{equation}
Indeed, if \(|\eta|\sim0\), then \(B=1\), and hence
\[
        DB\le D_0=X^{1/2},
        \qquad
        D^2(2B)\le 2D_0^2=2X.
\]
If \(|\eta|\ge1\), then
\[
        B\le |\eta|\le \frac{N}{dD_0},
\]
so
\begin{align}
        DB\le \frac{N}{D_0}\ll X^{1/2},
        \qquad
        D^2(2B)
        \le \frac{2D^2N}{dD_0}
        \le 2N=8X.
\end{align}

For \(\Omega\subset\mathcal G_N\), write
\begin{align}\label{def:Tk3_Omega}
        T_k^{(3)}(\Omega)
        :=
        \frac1N
        \sum_{\alpha\in\Omega}
        \widehat H_\Lambda(\alpha)\widehat C_k(\alpha),
\end{align}
where \(\widehat H_\Lambda(\alpha)\) is defined by
\eqref{eq:def-HLambda}. Then
\[
        T_k^{(3)}
        =
        T_k^{(3)}(\mathfrak M_A^{(N)})
        +
        T_k^{(3)}(\mathfrak m_A^{(N)}).
\]

\subsubsection{A three-point minor-arc lemma}

The key new ingredient is a pointwise estimate for the prime-side kernel
\(\widehat H_\Lambda\).  Fix
\[
        \alpha\in\mathfrak m_A^{(N)}(D,B),
        \qquad
        \alpha=\frac{\ell}{d}+\frac{\eta}{N},
        \qquad
        (\ell,d)=1,
        \qquad
        d\sim D,
        \qquad
        \max(1,|\eta|)\sim B.
\]

\begin{lemma}[Three-point minor-arc lemma]
\label{lem:optimal-three-point}
{For every \(\epsilon>0\), uniformly for
\(\alpha\in\mathfrak m_A^{(N)}(D,B)\) and
\(\beta\in\mathcal G_N\), one has}
\begin{equation}\label{eq:optimal-three-point-small}
\min\!\Bigl(
|\widehat\Lambda_X(\beta)|,
|\widehat\Lambda_X(\alpha-2\beta)|,
|\widehat\Lambda_X(\beta-\alpha)|
\Bigr)
\ll_\epsilon
\Bigl(
X^{-1/5}
+
(DB)^{-1/3+\epsilon}
\Bigr)X(\log X)^4.
\end{equation}
\end{lemma}

\begin{proof} 
{It is enough to consider \(0<\epsilon<1/3\), since for \(\epsilon\ge1/3\) the claim follows from the trivial estimate.}
If \(DB\) is bounded, the claim
also follows from the trivial estimate.  We may therefore assume that
\(DB\) is sufficiently large.

Consider the Dirichlet approximations
\[
\beta=\frac{\ell_1}{d_1}+\frac{\eta_1}{N},
\qquad
\alpha-2\beta=\frac{\ell_2}{d_2}+\frac{\eta_2}{N},
\qquad
\beta-\alpha=\frac{\ell_3}{d_3}+\frac{\eta_3}{N},
\]
with
\[
        (\ell_i,d_i)=1,
        \qquad
        1\le d_i\le D_0,
        \qquad
        \frac{|\eta_i|}{N}
        \leq \frac{1}{d_iD_0}
        \leq\frac{1}{d_i^2}.
\]
{The boundary case \(|\eta_i|/N=1/d_i^2\) follows by continuity.}
Let \(D_i,B_i\) be the dyadic parameters such that
\[
        d_i\sim D_i,
        \qquad
        \max(1,|\eta_i|)\sim B_i.
\]

Applying \cref{lem:minor_Mangoldt}, we have, for \(j=1,2,3\),
\begin{align}\label{eq:LambdaX_etaj}
\left|
\widehat\Lambda_X\!\left(
\frac{\ell_j}{d_j}+\frac{\eta_j}{N}
\right)
\right|
\ll
\left(
X^{-1/5}
+
(D_jB_j)^{-1/2}
+
\left(\frac{D_jB_j}{N}\right)^{1/2}
\right)X(\log X)^4.
\end{align}
Using \(B_j\ll N/(D_jD_0)\) and \(D_0=X^{1/2}\), we have
\[
        D_jB_j\ll X^{1/2},
        \qquad
        \left(\frac{D_jB_j}{N}\right)^{1/2}
        \ll X^{-1/4}.
\]
Thus the last term in \eqref{eq:LambdaX_etaj} is absorbed by \(X^{-1/5}\).  Therefore
\begin{align}\label{eq:Lambda_etaj_2}
\left|
\widehat\Lambda_X\!\left(
\frac{\ell_j}{d_j}+\frac{\eta_j}{N}
\right)
\right|
\ll
\left(
X^{-1/5}
+
(D_jB_j)^{-1/2}
\right)X(\log X)^4.
\end{align}
Suppose, for contradiction, that all three quantities in
\eqref{eq:Lambda_etaj_2} are larger than the asserted
right-hand side of \eqref{eq:optimal-three-point-small}. Then, for each \(j=1,2,3\),
\begin{equation}\label{eq:DjBj<DB}
        D_jB_j\ll (DB)^{2/3-2\epsilon}.
\end{equation}
We will show that this leads to a contradiction.

{\subsubsection[The identity involving three inner frequencies]{The identity involving three inner frequencies: $\{\beta, \alpha-2\beta, \beta-\alpha\}$.}}
We first use the identity
\begin{align}\label{eq:three_inner=0}
        \beta+(\alpha-2\beta)+(\beta-\alpha)=0
        \qquad\text{in }\mathbb R/\mathbb Z.
\end{align}
We will show that this implies:
\begin{proposition}\label{prop:123}
\begin{align}\label{eq:goal_three_inner}
    \begin{cases}
        \frac{\ell_1}{d_1}
        +
        \frac{\ell_2}{d_2}
        +
        \frac{\ell_3}{d_3}
        \in\mathbb Z\\
        \eta_1+\eta_2+\eta_3=0\\
        \mathrm{lcm}(d_1,d_2,d_3)^2\mid d_1d_2d_3.
    \end{cases}
\end{align}    
\end{proposition}
\begin{proof}
Note that \eqref{eq:three_inner=0} is the same as:
\begin{align}
    \frac{\ell_1}{d_1}+\frac{\ell_2}{d_2}+\frac{\ell_3}{d_3}=-\frac{\eta_1+\eta_2+\eta_3}{N} \qquad\text{in }\mathbb R/\mathbb Z.
\end{align}
By \eqref{eq:DjBj<DB} and \eqref{eq:DB_D2B<<},
\begin{align}\label{eq:Bi<N1/3}
        B_i\le D_iB_i
        \ll (DB)^{2/3-2\epsilon}
        \ll N^{1/3-\epsilon},
        \qquad 1\le i\le3.
\end{align}
Thus \(B_1+B_2+B_3=o(N)\).  If
\[
        \frac{\ell_1}{d_1}
        +
        \frac{\ell_2}{d_2}
        +
        \frac{\ell_3}{d_3}
        \notin\mathbb Z,
\]
then rational separation gives
\begin{align}
    \frac{1}{d_1d_2d_3}\leq \frac{|\eta_1|+|\eta_2|+|\eta_3|}{N},
\end{align}
which yields
\[
\begin{aligned}
1
&\ll
\frac{
(D_1B_1)D_2D_3
+
D_1(D_2B_2)D_3
+
D_1D_2(D_3B_3)
}{N} \\
&\ll
\frac{(DB)^{2-6\epsilon}}{N}
=o(1),
\end{aligned}
\]
where we used \(D_i\le D_iB_i\ll (DB)^{2/3-2\epsilon}\) as in \eqref{eq:DjBj<DB}, and $DB\ll N^{1/2}$ as in
\eqref{eq:DB_D2B<<}.  This is impossible.  Therefore
\begin{align}\label{eq:goal_three_inner_ld}
        \frac{\ell_1}{d_1}
        +
        \frac{\ell_2}{d_2}
        +
        \frac{\ell_3}{d_3}
        \in\mathbb Z.
\end{align}
The same identity gives
\[
        \eta_1+\eta_2+\eta_3\in N\mathbb Z.
\]
By \eqref{eq:Bi<N1/3},
\[
        |\eta_1+\eta_2+\eta_3|
        \ll B_1+B_2+B_3=o(N),
\]
and hence
\begin{align}\label{eq:goal_three_inner_eta}
        \eta_1+\eta_2+\eta_3=0.
\end{align}

Fix a prime \(p\) dividing
\(\operatorname{lcm}(d_1,d_2,d_3)\), and let \(p^v\) be the largest
power of \(p\) dividing one of \(d_1,d_2,d_3\).  We claim that \(p^v\)
divides at least two of the three denominators.  Otherwise, suppose for
example that the full power \(p^v\) occurs only in \(d_1\).  After
multiplying \eqref{eq:goal_three_inner_ld}
by \(\operatorname{lcm}(d_1,d_2,d_3)\) and reducing modulo \(p\), the
term involving \(\ell_1/d_1\) is nonzero, since
\((\ell_1,d_1)=1\), while the other two terms and the right-hand side
vanish modulo \(p\).  This is impossible.  Thus \(p^v\) divides at
least two of \(d_1,d_2,d_3\), and hence
\[
        {p^{2v}}\mid d_1d_2d_3.
\]
Applying this for every prime \(p\) gives
\begin{align}\label{eq:goal_three_inner_lcm}
        \operatorname{lcm}(d_1,d_2,d_3)^2
        \mid d_1d_2d_3.
\end{align}
Combining \eqref{eq:goal_three_inner_ld}, \eqref{eq:goal_three_inner_eta}, \eqref{eq:goal_three_inner_lcm} gives \eqref{eq:goal_three_inner}.
\end{proof}

{\subsubsection[The identity involving two inner frequencies and the outer frequency]{The identity involving two inner frequencies and the outer frequency $\alpha$.}}
We now use one of the three identities
\[
\alpha=\beta-(\beta-\alpha),
\qquad
\alpha=(\alpha-2\beta)+2\beta,
\qquad
\alpha=-(\alpha-2\beta)-2(\beta-\alpha)
\]
that involves the two largest of \(d_1,d_2,d_3\).  Suppose, for
definiteness and for the purpose of the following Proposition \ref{prop:023}, that \(d_1\le d_2\le d_3\).
We use the last identity, which is:
\begin{align}
    \frac{\ell}{d}+\frac{\ell_2}{d_2}+2\frac{\ell_3}{d_3}=-\frac{\eta+\eta_2+2\eta_3}{N}, \qquad\text{in }\mathbb R/\mathbb Z.
\end{align}
We will show:
\begin{proposition}\label{prop:023}
    \begin{align}\label{eq:goal_two_inner}
        \begin{cases}
            \frac{\ell}{d}+\frac{\ell_2}{d_2}+2\frac{\ell_3}{d_3}\in \Z\\
            \eta+\eta_2+2\eta_3=0.
        \end{cases}
    \end{align}
\end{proposition}
\begin{proof}
Suppose otherwise that:
\[
        \frac{\ell}{d}
        +
        \frac{\ell_2}{d_2}
        +
        2\frac{\ell_3}{d_3}
        \notin\mathbb Z,
\]
then rational separation gives
\begin{align}\label{eq:1<<eta023}
1
\ll
\frac{
d\,\operatorname{lcm}(d_2,d_3)
\bigl(|\eta|+|\eta_2|+|\eta_3|\bigr)
}{N}.
\end{align}
We estimate the three terms on the right-hand side of \eqref{eq:1<<eta023} separately.

First, note that by \eqref{eq:goal_three_inner}, 
\[
\operatorname{lcm}(d_2,d_3)\leq \operatorname{lcm}(d_1,d_2,d_3)
\ll
(D_1D_2D_3)^{1/2}
\ll
(DB)^{1-3\epsilon}.
\]
Hence, the term involving \(\eta\) is bounded by
\begin{align}\label{eq:eta023_1}
        \frac{d\,\operatorname{lcm}(d_2,d_3)|\eta|}{N}
        \ll
        \frac{DB \operatorname{lcm}(d_2,d_3)}{N}\ll \frac{(DB)^{2-3\epsilon}}{N}\ll N^{-\epsilon},
\end{align}
where we used {$DB\ll N^{1/2}$} as in \eqref{eq:DB_D2B<<}. 

Next, we consider for $j=2,3$ that,
\begin{align}\label{eq:eta023_j}
    \frac{d\,\operatorname{lcm}(d_2,d_3)|\eta_j|}{N}\ll D\frac{\operatorname{lcm}(d_2,d_3)}{d_j} \frac{D_jB_j}{N}\ll \frac{\operatorname{lcm}(d_2,d_3)}{d_j} \frac{(DB)^{5/3-2\epsilon}}{N}.
\end{align}
We have 
\begin{align}
    \frac{\operatorname{lcm}(d_2,d_3)}{d_j}\leq \frac{\operatorname{lcm}(d_1,d_2,d_3)}{d_2}\leq \frac{(d_1d_2d_3)^{1/2}}{d_2}\leq d_3^{1/2}\leq (D_3B_3)^{1/2}\ll (DB)^{1/3-\epsilon}.
\end{align}
Plugging this into \eqref{eq:eta023_j}, we obtain
\begin{align}\label{eq:eta023_23}
    \frac{d\,\operatorname{lcm}(d_2,d_3)|\eta_j|}{N}\ll \frac{(DB)^{2-3\epsilon}}{N}\ll N^{-\epsilon},
\end{align}
where we used {$DB\ll N^{1/2}$} again.

Plugging \eqref{eq:eta023_1} and \eqref{eq:eta023_23} back into \eqref{eq:1<<eta023}, we get a contradiction.
Thus 
\begin{align}\label{eq:goal_023_ld}
    \frac{\ell}{d}+\frac{\ell_2}{d_2}+2\frac{\ell_3}{d_3}\in \Z.
\end{align}
Clearly this also implies
\begin{align}
    \eta+\eta_2+2\eta_3 \in N\Z.
\end{align}
Note that
\begin{align}
    |\eta|+|\eta_2|+2|\eta_3|\ll B+B_2+B_3\ll DB+D_2B_2+D_3B_3\ll DB\ll N^{1/2}=o(N),
\end{align}
where we used $D_jB_j\ll (DB)^{2/3-2\epsilon}$ as in \eqref{eq:DjBj<DB} and $DB\ll N^{1/2}$ as in \eqref{eq:DB_D2B<<}.
This forces 
\[\eta+\eta_2+2\eta_3=0\]
as claimed.
\end{proof}

Now we combine Proposition \ref{prop:123} with Proposition \ref{prop:023}, this gives
\begin{align}\label{eq:ld====}
        \frac{\ell}{d}
        \equiv
        \frac{\ell_1}{d_1}-\frac{\ell_3}{d_3}
        \equiv
        \frac{\ell_2}{d_2}+2\frac{\ell_1}{d_1}
        \equiv
        -\frac{\ell_2}{d_2}-2\frac{\ell_3}{d_3}
        \pmod1,
\end{align}
and
\begin{align}\label{eq:eta====}
        \eta=\eta_1-\eta_3,
        \qquad
        \eta=\eta_2+2\eta_1,
        \qquad
        \eta=-\eta_2-2\eta_3.
\end{align}
The rational identities \eqref{eq:ld====} imply
\[
        d\mid\operatorname{lcm}(d_1,d_3),
        \qquad
        d\mid\operatorname{lcm}(d_1,d_2),
        \qquad
        d\mid\operatorname{lcm}(d_2,d_3).
\]
By \eqref{eq:goal_three_inner}, we have
\begin{align}\label{eq:D<D1D2D3}
    D\leq d\leq \operatorname{lcm}(d_1,d_2,d_3)\leq (d_1d_2d_3)^{1/2}\ll (D_1D_2D_3)^{1/2}.
\end{align}

Among the three identities involving $\eta$'s in \eqref{eq:eta====}. 
Suppose, without loss of generality, that
\(D_1\le D_2\le D_3\).  Then the corresponding  identity $\eta=-\eta_2-2\eta_3$ gives
\begin{align}\label{eq:B<B2+B3}
        B\ll B_2+B_3.
\end{align}
Combining \eqref{eq:D<D1D2D3} and
\eqref{eq:B<B2+B3}, we have
\begin{align}\label{eq:DB<<<}
DB
\ll
(D_1D_2D_3)^{1/2}B_2
+
(D_1D_2D_3)^{1/2}B_3.
\end{align}
Using $D_1\leq D_2$ and \eqref{eq:DjBj<DB}, the first term on the right-hand side of \eqref{eq:DB<<<} satisfies
\[
(D_1D_2D_3)^{1/2}B_2
\le
(D_2B_2)D_3^{1/2}
\ll
(DB)^{1-3\epsilon}.
\]
For the second term on the right-hand side of \eqref{eq:DB<<<}, similarly, we have
\[
\begin{aligned}
(D_1D_2D_3)^{1/2}B_3
=
(D_3B_3)
\left(\frac{D_1D_2}{D_3}\right)^{1/2} 
\le
(D_3B_3)D_2^{1/2}
\ll
(DB)^{1-3\epsilon}.
\end{aligned}
\]
Therefore \eqref{eq:DB<<<} implies
\[
        DB\ll (DB)^{1-3\epsilon}.
\]
For \(DB\) sufficiently large this is impossible.  This contradiction
proves the lemma.
\end{proof}

\begin{corollary}[Pointwise minor-arc bound for \(\widehat H_\Lambda\)]
\label{cor:HLambda-minor-pointwise}
For every \(\epsilon>0\), uniformly for
\(\alpha\in\mathfrak m_A^{(N)}(D,B)\), one has
\[
|\widehat H_\Lambda(\alpha)|
\ll_\epsilon
\Bigl(
X^{-1/5}
+
(DB)^{-1/3+\epsilon}
\Bigr)
X^2(\log X)^5.
\]
\end{corollary}

\begin{proof}
For simplicity, let
\[
\Delta_0
:=
\Bigl(
X^{-1/5}
+
(DB)^{-1/3+\epsilon}
\Bigr)X(\log X)^4.
\]
By \cref{lem:optimal-three-point}, we can bound
\begin{align*}
|\widehat H_\Lambda(\alpha)|
&\ll
\frac{\Delta_0}{N}
\sum_{\beta\in\mathcal G_N}
\Bigl(
|\widehat\Lambda_X(\beta)\widehat\Lambda_X(\alpha-2\beta)|
+
|\widehat\Lambda_X(\beta)\widehat\Lambda_X(\beta-\alpha)| \\
&\qquad\qquad\qquad
+
|\widehat\Lambda_X(\alpha-2\beta)
\widehat\Lambda_X(\beta-\alpha)|
\Bigr).
\end{align*}
Each bilinear average is bounded by \(O(X\log X)\) by
Cauchy--Schwarz and Parseval.  For instance,
\[
\frac1N
\sum_{\beta\in\mathcal G_N}
|\widehat\Lambda_X(\beta)\widehat\Lambda_X(\alpha-2\beta)|
\ll X\log X.
\]
Here the map \(\beta\mapsto\alpha-2\beta\) has fibers of cardinality
\(2\), since \(N=4X\), and this contributes only an absolute constant.
Thus
\[
        |\widehat H_\Lambda(\alpha)|
        \ll
        \Delta_0\cdot X\log X,
\]
which is the claimed bound.
\end{proof}

\subsubsection{Minor arc contributions for the three-term problem}

Define the dyadic minor-arc contribution
\[
\mathcal E^{(3)}_{\mathrm{minor}}(D,B)
:=
\frac1N
\sum_{\alpha\in\mathfrak m_A^{(N)}(D,B)}
|\widehat C_k(\alpha)\widehat H_\Lambda(\alpha)|.
\]
Inserting the pointwise estimate for \(\widehat H_\Lambda\) in
\cref{cor:HLambda-minor-pointwise} into the outer sum against the
\(\alpha_{b,\mathcal D}\)-admissible \(\widehat C_k\) yields the
following.

\begin{proposition}[Dyadic three-term minor-arc contribution]
\label{prop:dyadic-minor-HLambda}
Assume that \(\widehat C_k\) is
\(\alpha_{b,\mathcal D}\)-admissible.  For every \(\epsilon>0\),
\[
\mathcal E^{(3)}_{\mathrm{minor}}(D,B)
\ll_\epsilon
|\mathcal C_k|X(\log X)^5
\left(
X^{-1/5}(D^2B)^{\alpha_{b,\mathcal D}}
+
(DB)^{-1/3+\epsilon}(D^2B)^{\alpha_{b,\mathcal D}}
\right).
\]
\end{proposition}

\begin{proof}
The packet is contained in the cumulative range \(|\eta|<2B\).
By {\eqref{eq:DB_D2B<<}}, one has
\[
        D^2(2B)\le8X,
\]
so admissibility with cumulative parameter \(2B\) gives
\[
\sum_{\alpha\in\mathfrak m_A^{(N)}(D,B)}
|\widehat C_k(\alpha)|
\ll_{b,\mathcal D}
|\mathcal C_k|(D^2B)^{\alpha_{b,\mathcal D}}.
\]
Multiplying by \cref{cor:HLambda-minor-pointwise} and using
\(N\asymp X\) proves the proposition.
\end{proof}

\begin{theorem}[Three-term minor-arc bound]
\label{thm:minor-arc-final}
Assume that \(\widehat C_k\) is
\(\alpha_{b,\mathcal D}\)-admissible and that
\[
        \alpha_{b,\mathcal D}<\frac16.
\]
Then, for every \(\epsilon>0\),
\[
T_k^{(3)}(\mathfrak m_A^{(N)})
\ll_\epsilon
|\mathcal C_k|X
(\log X)^{
7
-
A\left(\frac16-\alpha_{b,\mathcal D}-\epsilon\right)
}.
\]
In particular, after choosing \(\epsilon>0\) sufficiently small and then
taking \(A\) sufficiently large, the minor-arc contribution is smaller
than \(|\mathcal C_k|X\) by an arbitrarily large power of \(\log X\).
\end{theorem}

\begin{proof}
It is enough to consider
\(0<\epsilon<1/6-\alpha_{b,\mathcal D}\); the estimate for larger
\(\epsilon\) follows from any smaller choice.  We sum
\cref{prop:dyadic-minor-HLambda} over all dyadic packets
\(\mathfrak m_A^{(N)}(D,B)\).  There are \(O((\log X)^2)\) such packets.
On the minor arcs,
\[
        \max(D,B)\gg L.
\]
Therefore
\[
        D^2B\gg L,
        \qquad
        DB\gg L.
\]
Also \(D^2B\ll X\).  Hence
\[
X^{-1/5}(D^2B)^{\alpha_{b,\mathcal D}}
\ll
(D^2B)^{\alpha_{b,\mathcal D}-1/5}
\ll
L^{-(1/5-\alpha_{b,\mathcal D})}
\le
L^{-(1/6-\alpha_{b,\mathcal D}-\epsilon)}.
\]
For the second term, since the admissible exponent is nonnegative and
\(D^2B\le(DB)^2\),
\[
\begin{aligned}
(DB)^{-1/3+\epsilon}(D^2B)^{\alpha_{b,\mathcal D}}
&\leq
(DB)^{2\alpha_{b,\mathcal D}-1/3+\epsilon} \\
&\ll
L^{-(1/3-2\alpha_{b,\mathcal D}-\epsilon)} \\
&\le
L^{-(1/6-\alpha_{b,\mathcal D}-\epsilon)}.
\end{aligned}
\]
Thus
\[
T_k^{(3)}(\mathfrak m_A^{(N)})
\ll_\epsilon
|\mathcal C_k|X(\log X)^7
L^{-(1/6-\alpha_{b,\mathcal D}-\epsilon)}.
\]
Using \(L=(\log X)^A\) proves the theorem.
\end{proof}

\subsection{Major arcs for the three-term problem}
\label{subsec:three-term-major-arcs}

We now evaluate {the major-arc contribution}.
Recall that the major arcs are
\[
        \alpha=\frac{\ell}{d}+\frac{\eta}{N},
        \qquad
        (\ell,d)=1,
        \qquad
        d\le L,
        \qquad
        |\eta|\le L, 
        \qquad
        L=(\log X)^A,
\]
with $N\ell/d+\eta\in\mathbb Z$.
As before, we separate the major arcs according to the arithmetic of the
denominator.  Let
\[
\mathfrak M_{A,\mathrm{good}}^{(N)}
:=
\left\{
\alpha=\frac{\ell}{d}+\frac{\eta}{N}\in\mathfrak M_A^{(N)}:
        p\mid d \Longrightarrow p\mid b
\right\},
\]
and
\[
        \mathfrak M_{A,\mathrm{bad}}^{(N)}
        :=
        \mathfrak M_A^{(N)}\setminus \mathfrak M_{A,\mathrm{good}}^{(N)}.
\]

First, we show the bad major arcs are negligible because \(\widehat C_k\) is exponentially small there. 
On the bad major arcs, we simply invoke the trivial upper bound for $\widehat H_{\Lambda}$, see \cref{prop:H-Lambda-combinatorial}, that 
\begin{align}\label{eq:HLambda-zero-trivial}
        |\widehat H_\Lambda(\alpha)|
        \le
        \widehat H_\Lambda(0)\ll X^2(\log X)^3,
\end{align}
where one uses trivial estimate \(\Lambda(m)\ll\log X\).

\begin{lemma}[Bad major arcs]
\label{lem:bad-major-three-term}
There exists \(c_{b,\mathcal D,A}>0\) such that
\[
        T_k^{(3)}(\mathfrak M_{A,\mathrm{bad}}^{(N)})
        \ll_{b,\mathcal D,A}
        |\mathcal C_k|X \exp\!\left(
        -c_{b,\mathcal D,A}\frac{\log X}{\log\log X}
        \right).
\]
\end{lemma}

\begin{proof}
If the reduced denominator of \(\alpha\) has a prime factor not dividing \(b\),
then by \cref{lem:digit-bad-major-general} the restricted-digit product structure gives
\[
        |\widehat C_k(\alpha)|
        \ll_{b,\mathcal D,A}
        |\mathcal C_k|\exp\!\left(
        -c_{b,\mathcal D,A}\frac{\log X}{\log\log X}
        \right),
\]
uniformly for \(d\le L\) and \(|\eta|\le L\).  Therefore, by
\eqref{eq:HLambda-zero-trivial},
\[
\left|
T_k^{(3)}(\mathfrak M_{A,\mathrm{bad}}^{(N)})
\right|
\ll
\frac1N
\#\mathfrak M_A^{(N)}\cdot 
X^2(\log X)^3
|\mathcal C_k|\exp\!\left(
        -c_{b,\mathcal D,A}\frac{\log X}{\log\log X}
        \right).
\]
Since \(\#\mathfrak M_A^{(N)}\ll L^3\) and \(N\asymp X\), the powers of
\(\log X\) are absorbed into the exponential saving.
\end{proof}

\subsubsection{Reduction on the good major arcs}
Set
\[
        A':=10(A+1),
        \qquad
        L':=(\log X)^{A'}.
\]
Fix $\alpha \in\mathfrak M_{A,\mathrm{good}}^{(N)}$.
We carry out another major/minor arc decomposition in the $\beta$ variable, by splitting the \(\beta\)-sum defining \(\widehat H_\Lambda(\alpha)\), see \eqref{eq:def-HLambda}, according to
\[
        \beta\in\mathfrak M_{A'}^{(N)}
        \qquad\text{and}\qquad
        \beta\in\mathfrak m_{A'}^{(N)}.
\]
If \(\beta\in\mathfrak m_{A'}^{(N)}\), its selected Dirichlet
approximation has \(d'\max(1,|\eta'|)\gg L'\).  Hence
\cref{lem:minor_Mangoldt}, including the rational endpoint, gives
\[
        |\widehat\Lambda_X(\beta)|
        \ll_A X(\log X)^{-A'/2+4}.
\]
By Cauchy--Schwarz and Parseval, allowing for the multiplicity-two map
\(\beta\mapsto\alpha-2\beta\),
\[
\frac1N
\sum_{\beta\in\mathcal G_N}
|\widehat\Lambda_X(\alpha-2\beta)
  \widehat\Lambda_X(\beta-\alpha)|
\ll X\log X.
\]
Thus the contribution of the terms with \(\beta\in\mathfrak m_{A'}^{(N)}\) is
\begin{align}\label{eq:H-beta-minor}
        \ll_A X^2(\log X)^{-A'/2+5}
        \ll_A X^2(\log X)^{-5A}.
\end{align}

It remains to consider \(\beta\in\mathfrak M_{A'}^{(N)}\).  Write
\[
        \beta=\frac{\ell'}{d'}+\frac{\eta'}{N},
        \qquad
        (\ell',d')=1,
        \qquad
        d'\le L',
        \qquad
        |\eta'|\le L'.
\]
Then
\[
        \alpha-2\beta
        =
        \frac{\ell d'-2\ell'd}{dd'}+\frac{\eta-2\eta'}{N},
\]
and
\[
        \beta-\alpha
        =
        \frac{\ell'd-\ell d'}{dd'}+\frac{\eta'-\eta}{N}.
\]
After reducing rational parts, write
\begin{align}\label{def:ell1d1_ell2d2}
        \frac{\ell d'-2\ell'd}{dd'}=\frac{\ell_1}{d_1},
        \qquad
        (\ell_1,d_1)=1, \text{ and, } 
        \frac{\ell'd-\ell d'}{dd'}=\frac{\ell_2}{d_2},
        \qquad
        (\ell_2,d_2)=1.
\end{align}
Clearly,
\[
        d_1\mid dd',
        \qquad
        d_2\mid dd'.
\]
Hence
\[
        d_1,d_2\le LL',
\]
and the major-arc approximation \cref{lem:prime-major-general}
applies to the three frequencies
\[
        \beta,
        \qquad
        \alpha-2\beta,
        \qquad
        \beta-\alpha,
\]
with a logarithmic parameter depending only on \(A'\).  This yields the following main term in major arc:
\begin{align}\label{def:H_good}
\widehat H_{\mathrm{good}}(\alpha)
:=
\frac1N
\sum_{\substack{\beta\in\mathfrak M_{A'}^{(N)}\\
\beta=\frac{\ell'}{d'}+\frac{\eta'}{N}}}
&\frac{\mu(d')}{\phi(d')}
\frac{\mu(d_1)}{\phi(d_1)}
\frac{\mu(d_2)}{\phi(d_2)}
U_X\!\left(\frac{\eta'}{N}\right)
U_X\!\left(\frac{\eta-2\eta'}{N}\right)
U_X\!\left(\frac{\eta'-\eta}{N}\right),
\end{align}
where \(d_1,d_2\) are the reduced denominators, see \eqref{def:ell1d1_ell2d2}.

\begin{proposition}[Reduction to the good-major model]
\label{prop:H-good-direct-reduction}
For every \(A>0\),
\[
\sup_{\alpha\in\mathfrak M_{A,\mathrm{good}}^{(N)}}
\left|
\widehat H_\Lambda(\alpha)-\widehat H_{\mathrm{good}}(\alpha)
\right|
\ll_A X^2(\log X)^{-A'/2+5}.
\]
\end{proposition}

\begin{proof}
The contribution of \(\beta\in\mathfrak m_{A'}^{(N)}\) was estimated above.
We therefore restrict to \(\beta\in\mathfrak M_{A'}^{(N)}\), and write
\[
        \theta_0:=\beta,
        \qquad
        \theta_1:=\alpha-2\beta,
        \qquad
        \theta_2:=\beta-\alpha.
\]
For such a \(\beta\), define the corresponding major-arc main terms by
\[
\begin{aligned}
        \mathcal M_0
        &:=
        \frac{\mu(d')}{\phi(d')}
        U_X\!\left(\frac{\eta'}{N}\right),\\
        \mathcal M_1
        &:=
        \frac{\mu(d_1)}{\phi(d_1)}
        U_X\!\left(\frac{\eta-2\eta'}{N}\right),\\
        \mathcal M_2
        &:=
        \frac{\mu(d_2)}{\phi(d_2)}
        U_X\!\left(\frac{\eta'-\eta}{N}\right).
\end{aligned}
\]

We first verify the uniformity needed to apply
\cref{lem:prime-major-general} with the logarithmic parameter \(A'\).
Since \(A'\ge A\), one has \(L\le L'\), and hence
\[
        d'\le L'\le (L')^2,
        \qquad
        d_1,d_2\le LL'\le (L')^2.
\]
Moreover,
\[
        |\eta'|\le L',
        \qquad
        |\eta-2\eta'|\le L+2L'\le 3L',
        \qquad
        |\eta'-\eta|\le L'+L\le 2L'.
\]
Thus, uniformly in
\(\alpha\in\mathfrak M_{A,\mathrm{good}}^{(N)}\) and
\(\beta\in\mathfrak M_{A'}^{(N)}\),
\[
        \widehat\Lambda_X(\theta_j)
        =
        \mathcal M_j+\mathcal R_j,
        \qquad
        \mathcal R_j
        \ll_{A'}X(\log X)^{-4A'}
        \qquad (j=0,1,2).
\]
Since
\[
        \left|\frac{\mu(q)}{\phi(q)}\right|\le 1
        \qquad\text{and}\qquad
        |U_X(t)|\le X,
\]
we also have
\[
        |\mathcal M_j|\le X,
        \qquad
        |\widehat\Lambda_X(\theta_j)|
        \ll_{A'}X
        \qquad (j=0,1,2).
\]

We now compare the two triple products using the exact telescoping identity
\[
\begin{aligned}
&\widehat\Lambda_X(\theta_0)
 \widehat\Lambda_X(\theta_1)
 \widehat\Lambda_X(\theta_2)
 -
 \mathcal M_0\mathcal M_1\mathcal M_2\\
&\quad=
 \bigl(\widehat\Lambda_X(\theta_0)-\mathcal M_0\bigr)
 \widehat\Lambda_X(\theta_1)
 \widehat\Lambda_X(\theta_2)\\
&\qquad\quad+
 \mathcal M_0
 \bigl(\widehat\Lambda_X(\theta_1)-\mathcal M_1\bigr)
 \widehat\Lambda_X(\theta_2)\\
&\qquad\quad+
 \mathcal M_0\mathcal M_1
 \bigl(\widehat\Lambda_X(\theta_2)-\mathcal M_2\bigr).
\end{aligned}
\]
Consequently, for every
\(\beta\in\mathfrak M_{A'}^{(N)}\),
\[
\begin{aligned}
&\left|
 \widehat\Lambda_X(\beta)
 \widehat\Lambda_X(\alpha-2\beta)
 \widehat\Lambda_X(\beta-\alpha)
 -
 \mathcal M_0\mathcal M_1\mathcal M_2
\right|\\
&\hspace{4cm}
\ll_{A'}X^3(\log X)^{-4A'}.
\end{aligned}
\]
The point is that this is a telescoping expansion: each displayed
summand contains one major-arc remainder, while the two remaining
factors are \(O_{A'}(X)\).  In particular, no product of the three
remainders is being used.

It follows that the total error from the \(\beta\)-major arcs is
\[
\begin{aligned}
&\frac1N
\sum_{\beta\in\mathfrak M_{A'}^{(N)}}
\left|
 \widehat\Lambda_X(\beta)
 \widehat\Lambda_X(\alpha-2\beta)
 \widehat\Lambda_X(\beta-\alpha)
 -
 \mathcal M_0\mathcal M_1\mathcal M_2
\right|\\
&\qquad\ll_{A'}
\frac{\#\mathfrak M_{A'}^{(N)}}{N}
X^3(\log X)^{-4A'}\\
&\qquad\ll_A
X^2(\log X)^{-A'},
\end{aligned}
\]
because
\[
        \#\mathfrak M_{A'}^{(N)}
        \ll (L')^3
        =
        (\log X)^{3A'},
        \qquad
        N\asymp X.
\]
Since \(A'=10(A+1)\), the \(\beta\)-major-arc error is smaller than
\(X^2(\log X)^{-A'/2+5}\).  Combining it with the previously
established contribution \eqref{eq:H-beta-minor} of \(\beta\in\mathfrak m_{A'}^{(N)}\) gives
\[
\sup_{\alpha\in\mathfrak M_{A,\mathrm{good}}^{(N)}}
\left|
\widehat H_\Lambda(\alpha)-\widehat H_{\mathrm{good}}(\alpha)
\right|
\ll_A X^2(\log X)^{-A'/2+5},
\]
which proves the proposition.
\end{proof}

Since
\(\#\mathfrak M_A^{(N)}\ll L^3\), \(|\widehat C_k(\alpha)|\le
|\mathcal C_k|\), and \(N\asymp X\), the error in
\cref{prop:H-good-direct-reduction} contributes at most
\[
\frac{L^3}{N}|\mathcal C_k|X^2(\log X)^{-A'/2+5}
\ll_A |\mathcal C_k|X(\log X)^{-2A}.
\]
Consequently,
\begin{equation}\label{eq:good-major-reduction}
T_k^{(3)}(\mathfrak M_{A,\mathrm{good}}^{(N)})
=
\frac1N
\sum_{\alpha\in\mathfrak M_{A,\mathrm{good}}^{(N)}}
\widehat H_{\mathrm{good}}(\alpha)\widehat C_k(\alpha)
+
O_{b,\mathcal D,A}\!\left(|\mathcal C_k|X(\log X)^{-A}\right).
\end{equation}

\subsubsection[Reduction to smooth alpha]{Reduction to smooth \(\alpha\)}
We can actually reduce from all $\alpha=\frac{\ell}{d}+\frac{\eta}{N} \in \mathfrak{M}_{A,\text{good}}^{(N)}$ to a smaller collection of $\alpha$ where $d$ is squarefree and $d\mid b$.

Due to the appearance of
\(\mu(d'),\mu(d_1),\mu(d_2)\), every nonzero summand has
\(d',d_1,d_2\) squarefree.  Also, every prime factor of the outer
denominator \(d\) divides \(b\).

We claim that \(d\) is squarefree.  Suppose that \(v:=v_p(d)\ge2\) for
some prime \(p\), where \(v_p(n)\) denotes the exponent of \(p\) in the
prime factorization of \(n\). If \(p\nmid d'\), then
\[
\frac{\ell'}{d'}-\frac{\ell}{d}
=\frac{\ell'd-\ell d'}{dd'},
\qquad
\ell'd-\ell d'\equiv-\ell d'\not\equiv0\pmod p.
\]
Thus the reduced denominator \(d_2\) contains \(p^v\), contradicting
\(\mu(d_2)\ne0\).  If \(p\mid d'\), write \(d'=pr'\), with \(p\nmid r'\).
After cancelling the evident common factor \(p\), one has
\[
\frac{\ell'}{pr'}-\frac{\ell}{p^vr}
=
\frac{p^{v-1}\ell'r-\ell r'}{p^vrr'},
\qquad
p^{v-1}\ell'r-\ell r'\equiv-\ell r'\not\equiv0\pmod p.
\]
Again \(p^v\mid d_2\), a contradiction.  Thus \(d\) is squarefree.  Since
all its prime factors divide \(b\), it follows that \(d\mid b\). Hence
\eqref{eq:good-major-reduction} yields the following.
\begin{proposition}[Reduction to squarefree \(b\)-smooth denominators]\label{prop:Tk3=good_smooth_alpha}
    \begin{align}\label{eq:good-major-reduction_smooth}
    T_k^{(3)}(\mathfrak M_{A,\mathrm{good}}^{(N)})
=
\frac1N \sum_{d\mid b}\sum_{\substack{1\le\ell\le d\\(\ell,d)=1}}\sum_{\substack{|\eta|\leq L\\ N\frac{\ell}{d}+\eta\in \Z}} 
\widehat H_{\mathrm{good}}(\frac{\ell}{d}+\frac{\eta}{N})\widehat C_k(\frac{\ell}{d}+\frac{\eta}{N})
+
O_{b,\mathcal D,A}\!\left(|\mathcal C_k|X(\log X)^{-A}\right)
\end{align}
\end{proposition}

We next introduce the projected three-prime local factor that will arise from
the good major arcs.  Throughout the decoupling argument below, we write
\[
        L':=(\log X)^{A'}
\]
for the auxiliary \(\beta\)-major arc cutoff.  For \(d\ge1\) and
\((\ell,d)=1\), set
\begin{equation}\label{def:S1full-three-term}
S_{1,\mathrm{full}}\!\left(\frac{\ell}{d}\right)
:=
\sum_{d'=1}^{\infty}
\sum_{\substack{1\le \ell'\le d'\\(\ell',d')=1}}
\frac{\mu(d')}{\phi(d')}
\frac{\mu(d_1)}{\phi(d_1)}
\frac{\mu(d_2)}{\phi(d_2)},
\end{equation}
where \(d_1,d_2\) are the reduced denominators determined by \eqref{def:ell1d1_ell2d2}.
We then define
\begin{equation}\label{def:projected-three-prime-local-factor}
        \mathfrak S_b^{(3)}(n)
        :=
        \sum_{d\mid b}
        \sum_{\substack{1\le \ell\le d\\(\ell,d)=1}}
        S_{1,\mathrm{full}}\!\left(\frac{\ell}{d}\right)
        e\!\left(\frac{n\ell}{d}\right).
\end{equation}

\medskip

The goal of this subsection is to prove the following:
\begin{proposition}[Three-term major-arc formula]
\label{prop:three-term-major-projected-local-factor}
For every sufficiently large \(A>0\),  one has
\begin{equation}\label{eq:three-term-major-projected-local-factor}
T_k^{(3)}\!\left(\mathfrak M_A^{(N)}\right)
=
\sum_{0\le n<X/2}
\mathbf 1_{\mathcal C_k}(n)(X-2n)\mathfrak S_b^{(3)}(n)
+
O_{b,\mathcal D,A}\!\left(
|\mathcal C_k|X(\log X)^{-A}
\right).
\end{equation}
\end{proposition}

For the proof of this proposition, we first need a decoupling lemma as follows.

\begin{lemma}[Approximate decoupling of the good-major kernel]
\label{lem:Hgood-approximate-decoupling}
Let
\[
        \alpha=\frac{\ell}{d}+\frac{\eta}{N}
        \in \mathfrak M_{A,\mathrm{good}}^{(N)},
        \qquad
        d\mid b,
        \qquad
        |\eta|\le L.
\]
Then we can almost decouple the contributions from $\ell/d$ and $\eta$:
\begin{equation}\label{eq:Hgood-S1S2full}
        \widehat H_{\mathrm{good}}\!\left(\frac{\ell}{d}+\frac{\eta}{N}\right)
        =
        S_{1,\mathrm{full}}\!\left(\frac{\ell}{d}\right)
        S_{2,\mathrm{full}}(\eta)
        +
        O_b(X^2(L')^{-1}).
\end{equation}
where
\begin{equation}\label{eq:S2full-compact}
        S_{2,\mathrm{full}}(\eta)
        :=
        \sum_{|h|<X/2}
        (X-2|h|)
        e\!\left(-\frac{h\eta}{N}\right).
\end{equation}
{Thus \(S_{2,\mathrm{full}}\) is the archimedean kernel: it is
built from the physical weight \(X-2|h|\) imposed by the interval
restrictions.}  Moreover,
\begin{equation}\label{eq:S2full-decay}
        |S_{2,\mathrm{full}}(\eta)|
        \ll
        \frac{X^2}{(1+|\eta|)^2}.
\end{equation}
\end{lemma}

\begin{proof}
{We first rewrite the good-major model in a form that separates
the arithmetic coefficient from the archimedean kernel.}  For \(d'\ge1\) and
\((\ell',d')=1\), define
\[
        \mathcal A_{\ell,d}(\ell',d')
        :=
        \frac{\mu(d')}{\phi(d')}
        \frac{\mu(d_1)}{\phi(d_1)}
        \frac{\mu(d_2)}{\phi(d_2)},
\]
where \(d_1,d_2\) are the reduced denominators appearing in
\eqref{def:S1full-three-term}.  Then the definition of
\(\widehat H_{\mathrm{good}}\) gives
\begin{align}\label{eq:Hgood-pre-decoupling}
\widehat H_{\mathrm{good}}\!\left(\frac{\ell}{d}+\frac{\eta}{N}\right)
&=
\sum_{d'\le L'}
\sum_{\substack{1\le \ell'\le d'\\(\ell',d')=1}}
\mathcal A_{\ell,d}(\ell',d')\,
S_{2,L'}\!\left(\frac{\ell'}{d'},\eta\right),
\end{align}
where
\[
        S_{2,L'}\!\left(\frac{\ell'}{d'},\eta\right)
        :=
        \frac1N
        \sum_{\substack{|\eta'|\le L'\\N\ell'/d'+\eta'\in\mathbb Z}}
        U_X\!\left(\frac{\eta'}{N}\right)
        U_X\!\left(\frac{\eta-2\eta'}{N}\right)
        U_X\!\left(\frac{\eta'-\eta}{N}\right).
\]

\medskip

\noindent\textbf{Step 1: Completion of the inner \(\eta'\)-sum.}
We claim that, uniformly for \(d'\le L'\), \((\ell',d')=1\), and
\(|\eta|\le L\),
\begin{equation}\label{eq:S2L-S2full}
        S_{2,L'}\!\left(\frac{\ell'}{d'},\eta\right)
        =
        S_{2,\mathrm{full}}(\eta)
        +
        O(X^2(L')^{-1}).
\end{equation}
To prove this, complete the \(\eta'\)-sum to the full residue class
\[
        \eta'\in\mathbb Z/N\mathbb Z-N\ell'/d'.
\]
Expanding the three \(U_X\)-factors gives
\begin{align}\label{eq:complete_UXeta'}
&\frac1N
\sum_{\eta'\in\mathbb Z/N\mathbb Z-N\ell'/d'}
U_X\!\left(\frac{\eta'}{N}\right)
U_X\!\left(\frac{\eta-2\eta'}{N}\right)
U_X\!\left(\frac{\eta'-\eta}{N}\right)      \\
&\qquad =
\sum_{0\le m_1,m_2,m_3<X}
e\!\left(\frac{(m_2-m_3)\eta}{N}\right)
\frac1N
\sum_{\eta'\in\mathbb Z/N\mathbb Z-N\ell'/d'}
e\!\left(\frac{(m_1-2m_2+m_3)\eta'}{N}\right).\notag
\end{align}
Writing
\(\eta'=t-N\ell'/d'\), with \(t\in\mathbb Z/N\mathbb Z\), gives
\[
\begin{aligned}
&\frac1N
\sum_{\eta'\in\mathbb Z/N\mathbb Z-N\ell'/d'}
e\!\left(\frac{(m_1-2m_2+m_3)\eta'}{N}\right)\\
&\qquad=
e\!\left(-\frac{(m_1-2m_2+m_3)\ell'}{d'}\right)
\frac1N\sum_{t\in\mathbb Z/N\mathbb Z}
e\!\left(\frac{(m_1-2m_2+m_3)t}{N}\right).
\end{aligned}
\]
This equals zero unless
\(m_1-2m_2+m_3\equiv0\pmod N\), and equals one when that congruence
holds.  Since \(0\le m_i<X\) and \(N=4X\), one has
\(|m_1-2m_2+m_3|<2X<N\); hence the congruence is equivalent to the
ordinary equality \(m_1-2m_2+m_3=0\).
Writing the resulting three-term progression as
\[
        m_1=m,\qquad m_2=m+h,\qquad m_3=m+2h,
\]
the common difference \(h\) may be positive or negative.  The constraints
\(0\le m_i<X\) imply
\[
        |h|<\frac X2,
\]
and for each such \(h\) there are exactly \(X-2|h|\) admissible choices of
\(m\).  Since \(m_2-m_3=-h\), the completed sum in \eqref{eq:complete_UXeta'} becomes
\[
        \sum_{|h|<X/2}
        (X-2|h|)
        e\!\left(-\frac{h\eta}{N}\right)
        =
        S_{2,\mathrm{full}}(\eta).
\]

It remains to bound the omitted tail \(|\eta'|>L'\).  Choose
representatives of the shifted residue class in \((-N/2,N/2]\).  Since
\(|\eta|\le L\) and \(L'\gg L\), for \(L'<|\eta'|\le N/4\) one has
\[
\left|U_X\!\left(\frac{\eta'}N\right)
U_X\!\left(\frac{\eta-2\eta'}N\right)
U_X\!\left(\frac{\eta'-\eta}N\right)\right|
\ll \frac{X^3}{|\eta'|^3}.
\]
The shifted lattice has unit spacing, so this portion is
\[
\ll\frac{X^3}{N}\sum_{r>L'}r^{-3}
\ll X^2(L')^{-2}.
\]
If \(N/4<|\eta'|\le N/2\), the first and third factors are \(O(1)\), while
\[
\left|U_X\!\left(\frac{\eta-2\eta'}N\right)\right|
\ll
\frac{X}{1+\operatorname{dist}(2\eta'-\eta,N\mathbb Z)}.
\]
The doubling map has bounded multiplicity on the unit-spaced shifted lattice;
after division by \(N\asymp X\), this range contributes \(O(\log X)\).
Consequently the omitted tail is
\[
O\!\left(X^2(L')^{-2}+\log X\right)
=O\!\left(X^2(L')^{-1}\right),
\]
uniformly in \(d',\ell'\), and \(|\eta|\le L\).  This proves
\eqref{eq:S2L-S2full}.

The factor $S_{2,\mathrm{full}}(\eta)$, arising from the interval constraints, can be expressed in terms of Fejér kernels. If $X$ is even, write \(X=2M\), then
\[
S_{2,\mathrm{full}}(\eta)
=2\left|U_M\!\left(\frac\eta N\right)\right|^2,
\]
whereas if $X$ is odd, write \(X=2M+1\), then
\[
S_{2,\mathrm{full}}(\eta)
=\left|U_M\!\left(\frac\eta N\right)\right|^2
+\left|U_{M+1}\!\left(\frac\eta N\right)\right|^2.
\]
The usual geometric-sum estimate now proves
\eqref{eq:S2full-decay}.

\medskip

\noindent\textbf{Step 2: Completion of the \(d'\) sum.}

Define the truncated arithmetic coefficient
\[
        S_{1,L'}\!\left(\frac{\ell}{d}\right)
        :=
        \sum_{d'\le L'}
        \sum_{\substack{1\le \ell'\le d'\\(\ell',d')=1}}
        \mathcal A_{\ell,d}(\ell',d').
\]
We claim that
\begin{equation}\label{eq:S1-S1full}
        S_{1,L'}\!\left(\frac{\ell}{d}\right)
        =
        S_{1,\mathrm{full}}\!\left(\frac{\ell}{d}\right)
        +
        O_b((L')^{-1}),
\end{equation}
where $S_{1,\mathrm{full}}$ is defined in \eqref{def:S1full-three-term},
and the following upper bound
\begin{equation}\label{eq:S1full-bound}
        \left|
        S_{1,\mathrm{full}}\!\left(\frac{\ell}{d}\right)
        \right|
        \ll_b 1.
\end{equation}

We first note that any prime divisor $p$ of $d'$ that is coprime to \(2d\) must divide both $d_1$ and $d_2$.  
In fact, let $p$ be such a prime divisor. 
Since 
\((\ell',d')=1\), we have \(p\nmid\ell'\), and since \(p\nmid d\),
\[
        \ell'd-\ell d'\equiv \ell'd\not\equiv0\pmod p,
\]
while
\[
        \ell d'-2\ell'd\equiv -2\ell'd\not\equiv0\pmod p.
\]
Hence \(p\) does not cancel from the denominator \(dd'\) in either reduced
fraction defining \(d_1\) and \(d_2\), see \eqref{def:ell1d1_ell2d2}.  Thus every prime \(p\nmid 2d\)
dividing \(d'\) also divides both \(d_1\) and \(d_2\).

Set
\[
        r(d'):=\prod_{\substack{p\mid d'\\p\nmid 2d}}p.
\]
If \(\mu(d')\neq0\), then \(d'\) is squarefree, and the preceding paragraph
implies
\[
        r(d')\mid d_1,
        \qquad
        r(d')\mid d_2.
\]
Therefore
\[
\begin{aligned}
\sum_{\substack{1\le \ell'\le d'\\(\ell',d')=1}}
\left|
\mathcal A_{\ell,d}(\ell',d')
\right|
&\le
\sum_{\substack{1\le \ell'\le d'\\(\ell',d')=1}}
\frac{\mu^2(d')}{\phi(d')\phi(r(d'))^2}  \\
&=
\frac{\mu^2(d')}{\phi(r(d'))^2}.
\end{aligned}
\]
Write \(d'=qr\), where every prime factor of \(q\) divides \(2d\), and
\((r,2d)=1\).  If \(\mu^2(d')\neq0\), then \(q\) is a squarefree divisor of
\[
        P_{2d}:=\prod_{p\mid 2d}p,
\]
and \(r(d')=r\).  Hence
\[
\sum_{d'\ge1}
\sum_{\substack{1\le \ell'\le d'\\(\ell',d')=1}}
\left|
\mathcal A_{\ell,d}(\ell',d')
\right|
\ll
\sum_{q\mid P_{2d}}
\sum_{\substack{r\ge1\\(r,2d)=1}}
\frac{\mu^2(r)}{\phi(r)^2}
\ll_b 1,
\]
because \(d\mid b\) and \(b\) is fixed.
This proves
\eqref{eq:S1full-bound}.

The tail is estimated in the same way:
\begin{align}\label{eq:S1full-S1L'}
\left|
S_{1,\mathrm{full}}\!\left(\frac{\ell}{d}\right)
-
S_{1,L'}\!\left(\frac{\ell}{d}\right)
\right|
\le
\sum_{d'>L'}
\frac{\mu^2(d')}{\phi(r(d'))^2}.
\end{align}
Writing again \(d'=qr\), with \(q\mid P_{2d}\), gives
\begin{align}\label{eq:S1full-S1L'_1}
\sum_{d'>L'}
\frac{\mu^2(d')}{\phi(r(d'))^2}
\le
\sum_{q\mid P_{2d}}
\sum_{r>L'/q}
\frac{\mu^2(r)}{\phi(r)^2}
\ll
\sum_{q\mid P_{2d}}\frac{q}{L'}
\ll_b (L')^{-1},
\end{align}
Here we used a standard estimate:
\begin{equation*}
        \sum_{r>Y}\frac{\mu^2(r)}{\phi(r)^2}\ll \frac1Y
        \qquad(Y\ge1).
\end{equation*}
Combining \eqref{eq:S1full-S1L'} with \eqref{eq:S1full-S1L'_1} proves \eqref{eq:S1-S1full}.
\medskip

Finally, we conclude  the decoupling.
Substituting \eqref{eq:S2L-S2full} into
\eqref{eq:Hgood-pre-decoupling} gives
\[
\widehat H_{\mathrm{good}}\!\left(\frac{\ell}{d}+\frac{\eta}{N}\right)
=
S_{2,\mathrm{full}}(\eta)
S_{1,L'}\!\left(\frac{\ell}{d}\right)
+
O(X^2(L')^{-1})
\sum_{d'\le L'}
\sum_{\substack{1\le \ell'\le d'\\(\ell',d')=1}}
\left|\mathcal A_{\ell,d}(\ell',d')\right|.
\]
The absolute-convergence estimate preceding
\eqref{eq:S1-S1full} applies equally to its partial sums, so the last
double sum is \(O_b(1)\).  Hence
\[
\widehat H_{\mathrm{good}}\!\left(\frac{\ell}{d}+\frac{\eta}{N}\right)
=
S_{2,\mathrm{full}}(\eta)
S_{1,L'}\!\left(\frac{\ell}{d}\right)
+
O_b(X^2(L')^{-1}).
\]
Using \eqref{eq:S1-S1full} and the trivial bound
\[
        |S_{2,\mathrm{full}}(\eta)|\ll X^2
\]
from \eqref{eq:S2full-decay}, we obtain
\[
        S_{2,\mathrm{full}}(\eta)
        \left(
        S_{1,L'}\!\left(\frac{\ell}{d}\right)
        -
        S_{1,\mathrm{full}}\!\left(\frac{\ell}{d}\right)
        \right)
        \ll_b
        X^2(L')^{-1}.
\]
Therefore
\[
        \widehat H_{\mathrm{good}}\!\left(\frac{\ell}{d}+\frac{\eta}{N}\right)
        =
        S_{1,\mathrm{full}}\!\left(\frac{\ell}{d}\right)
        S_{2,\mathrm{full}}(\eta)
        +
        O_b(X^2(L')^{-1}),
\]
as claimed.
\end{proof}

\bigskip

We now use the approximate decoupling to evaluate the good-major contribution.

\begin{proof}[Proof of \cref{prop:three-term-major-projected-local-factor}]
By the previous reduction to good squarefree \(b\)-smooth denominators and by
\cref{lem:Hgood-approximate-decoupling}, we have
\begin{align}\label{eq:Tk3-good-before-eta-completion}
T_k^{(3)}\!\left(\mathfrak M_{A,\mathrm{good}}^{(N)}\right)
&=
\frac1N
\sum_{d\mid b}
\sum_{\substack{1\le \ell\le d\\(\ell,d)=1}}
\sum_{\substack{|\eta|\le L\\N\ell/d+\eta\in\mathbb Z}}
\widehat C_k\!\left(\frac{\ell}{d}+\frac{\eta}{N}\right)
S_{1,\mathrm{full}}\!\left(\frac{\ell}{d}\right)
S_{2,\mathrm{full}}(\eta)                                    \notag\\
&\qquad
+
O_{b,\mathcal D}\!\left(
|\mathcal C_k|X L^{-(1-\alpha_{b,\mathcal D})}
\right).
\end{align}
Indeed, the error term in \eqref{eq:Hgood-S1S2full} contributes
\[
        \ll_b
        \frac{X^2(L')^{-1}}{N}
        \sum_{d\mid b}
        \sum_{\substack{1\le \ell\le d\\(\ell,d)=1}}
        \sum_{\substack{|\eta|\le L\\N\ell/d+\eta\in\mathbb Z}}
        \left|
        \widehat C_k\!\left(\frac{\ell}{d}+\frac{\eta}{N}\right)
        \right|.
\]
Since \(d\mid b\), the localized hybrid estimate gives
\[
        \sum_{d\mid b}
        \sum_{\substack{1\le \ell\le d\\(\ell,d)=1}}
        \sum_{\substack{|\eta|\le L\\N\ell/d+\eta\in\mathbb Z}}
        \left|
        \widehat C_k\!\left(\frac{\ell}{d}+\frac{\eta}{N}\right)
        \right|
        \ll_{b,\mathcal D}
        |\mathcal C_k|L^{\alpha_{b,\mathcal D}}.
\]
Since \(N\asymp X\), this contribution is
\[
        \ll_{b,\mathcal D}
        |\mathcal C_k|X(L')^{-1}L^{\alpha_{b,\mathcal D}}
        \ll_{b,\mathcal D}
        |\mathcal C_k|X L^{-(1-\alpha_{b,\mathcal D})},
\]
because \(A'\ge A\).  This proves the error term in
\eqref{eq:Tk3-good-before-eta-completion}.

\noindent\textbf{Step 3: Completion of the remaining \(\eta\) sum.}

We now complete the remaining \(\eta\) sum in the main term of
\eqref{eq:Tk3-good-before-eta-completion}.  Fix \(d\mid b\) and
\((\ell,d)=1\).  By \eqref{eq:S2full-compact}, the completed \(\eta\) sum is

\[
\begin{aligned}
&\sum_{\substack{\eta\;(\mathrm{mod}\,N)\\N\ell/d+\eta\in\mathbb Z}}
S_{2,\mathrm{full}}(\eta)
\widehat C_k\!\left(\frac{\ell}{d}+\frac{\eta}{N}\right) \\
&\quad =
\sum_{|h|<X/2}
(X-2|h|)
\sum_{0\le n<X}
\mathbf 1_{\mathcal C_k}(n)e\!\left(\frac{n\ell}{d}\right)
\sum_{\substack{\eta\;(\mathrm{mod}\,N)\\N\ell/d+\eta\in\mathbb Z}}
e\!\left(\frac{(n-h)\eta}{N}\right).
\end{aligned}
\]
To evaluate the inner sum, write
\[
        \eta=t-\frac{N\ell}{d},
        \qquad
        t\in\mathbb Z/N\mathbb Z.
\]
Then
\[
\begin{aligned}
\sum_{\substack{\eta\;(\mathrm{mod}\,N)\\N\ell/d+\eta\in\mathbb Z}}
e\!\left(\frac{(n-h)\eta}{N}\right)
&=
e\!\left(-\frac{(n-h)\ell}{d}\right)
\sum_{t\in\mathbb Z/N\mathbb Z}
e\!\left(\frac{(n-h)t}{N}\right).
\end{aligned}
\]
Thus the inner sum vanishes unless
\[
        n-h\equiv0\pmod N.
\]
Since \(0\le n<X\), \(|h|<X/2\), and \(N>3X\), there is no wraparound, so the
congruence is equivalent to \(n=h\).  In that case the additional phase is
\(1\), and the inner sum is \(N\).  Hence only the terms with \(h=n\ge0\)
survive, and the completed sum equals
\begin{equation}\label{eq:completed-outer-eta-main}
        N
        \sum_{0\le n<X/2}
        \mathbf 1_{\mathcal C_k}(n)(X-2n)
        e\!\left(\frac{n\ell}{d}\right).
\end{equation}

It remains to estimate the tail introduced by replacing the
truncated sum \(|\eta|\le L\) with the completed \(\eta\)-sum.  By
\eqref{eq:S1full-bound} and the quadratic decay
\eqref{eq:S2full-decay} of $|S_{2,\mathrm{full}}(\eta)|$, its contribution to
\eqref{eq:Tk3-good-before-eta-completion} is at most
\[
\frac{X^2}{N}
\sum_{d\mid b}
\sum_{\substack{1\le \ell\le d\\(\ell,d)=1}}
\sum_{\substack{|\eta|>L\\N\ell/d+\eta\in\mathbb Z}}
\frac{
\left|
\widehat C_k\!\left(\frac{\ell}{d}+\frac{\eta}{N}\right)
\right|
}{(1+|\eta|)^2}.
\]
A dyadic decomposition and admissibility give
\[
\begin{aligned}
&\sum_{d\mid b}
\sum_{\substack{1\le \ell\le d\\(\ell,d)=1}}
\sum_{\substack{|\eta|>L\\N\ell/d+\eta\in\mathbb Z}}
\frac{
\left|
\widehat C_k\!\left(\frac{\ell}{d}+\frac{\eta}{N}\right)
\right|
}{(1+|\eta|)^2}\\
&\qquad\ll_{b,\mathcal D}
\sum_{\substack{B>L\\B\ \mathrm{dyadic}}}
B^{-2}|\mathcal C_k|B^{\alpha_{b,\mathcal D}}
\ll_{b,\mathcal D}
|\mathcal C_k|L^{-(2-\alpha_{b,\mathcal D})}.
\end{aligned}
\]
Since \(N\asymp X\), this is
\(O_{b,\mathcal D}(|\mathcal C_k|X
L^{-(2-\alpha_{b,\mathcal D})})\), which is smaller than the error
already present in \eqref{eq:Tk3-good-before-eta-completion}.

Substituting \eqref{eq:completed-outer-eta-main} into
\eqref{eq:Tk3-good-before-eta-completion}, and using the preceding tail
estimate, gives
\[
\begin{aligned}
T_k^{(3)}\!\left(\mathfrak M_{A,\mathrm{good}}^{(N)}\right)
&=
\sum_{0\le n<X/2}
\mathbf 1_{\mathcal C_k}(n)(X-2n)
\sum_{d\mid b}
\sum_{\substack{1\le \ell\le d\\(\ell,d)=1}}
S_{1,\mathrm{full}}\!\left(\frac{\ell}{d}\right)
e\!\left(\frac{n\ell}{d}\right)                       \\
&\qquad
+
O_{b,\mathcal D}\!\left(
|\mathcal C_k|X L^{-(1-\alpha_{b,\mathcal D})}
\right).
\end{aligned}
\]
By the definition of \(\mathfrak S_b^{(3)}(n)\), the inner double sum is
precisely
\[
        \mathfrak S_b^{(3)}(n).
\]
Since every modulus \(d\) in its definition divides \(b\),
\(\mathfrak S_b^{(3)}(n)\) depends only on \(n\bmod b\), and hence only
on the least significant base-\(b\) digit of \(n\).

Therefore
\begin{equation}\label{eq:three-term-good-major-local-factor}
T_k^{(3)}\!\left(\mathfrak M_{A,\mathrm{good}}^{(N)}\right)
=
\sum_{0\le n<X/2}
\mathbf 1_{\mathcal C_k}(n)(X-2n)
\mathfrak S_b^{(3)}(n)
+
O_{b,\mathcal D}\!\left(
|\mathcal C_k|X L^{-(1-\alpha_{b,\mathcal D})}
\right).
\end{equation}

The bad-major-arc estimate gives
\[
T_k^{(3)}\!\left(\mathfrak M_{A,\mathrm{bad}}^{(N)}\right)
\ll_{b,\mathcal D,A}
|\mathcal C_k|X(\log X)^{-A}.
\]
Combining this with \eqref{eq:three-term-good-major-local-factor}, and
choosing the logarithmic parameter sufficiently large in terms of the
prescribed saving \(A\), proves
\eqref{eq:three-term-major-projected-local-factor}.
\end{proof}

\subsubsection{Averaging the three-term local factor}
\label{subsec:averaging-three-term-local-factor}

{It remains to average the arithmetic local factor against the
archimedean weight arising from the requirement
\(m,m+n,m+2n\in[0,X)\):}
\[
        \sum_{0\le n<X/2}
        \mathbf 1_{\mathcal C_k}(n)(X-2n)\mathfrak S_b^{(3)}(n).
\]
As in the prime-pair case, the {arithmetic local factor} depends
only on the final base-\(b\) digit.  The new feature in the three-term problem is that the
cutoff \(n<X/2\) is governed by the leading base-\(b\) digits.
Define
\begin{equation}\label{def:three-term-average-local-factor}
        \overline{\mathfrak S}^{(3)}_{b,\mathcal D}
        :=
        \frac{1}{|\mathcal D|}
        \sum_{a\in\mathcal D}\mathfrak S_b^{(3)}(a).
\end{equation}
{This is the averaged arithmetic factor over the final digit.}
We also define the finite-scale averages
\[
        \kappa_j
        :=
        \frac{1}{b^j|\mathcal C_j|}
        \sum_{\substack{m\in\mathcal C_j\\0\le m<b^j/2}}
        (b^j-2m).
\]
The limit
\begin{equation}\label{def:kappa-bD-three-term}
        \kappa_{b,\mathcal D}
        :=
        \lim_{j\to\infty}\kappa_j
\end{equation}
exists, as proved below.  {We call
\(\kappa_{b,\mathcal D}\) the archimedean factor, since it is the limiting
normalized average of the physical weight \((1-2m/b^j)_+\) over
\(\mathcal C_j\).}

\begin{proposition}\label{prop:weighted-three-term-local-factor} 
One has 
\begin{equation}\label{eq:weighted-three-term-local-factor} 
\sum_{0\le n<X/2} \mathbf 1_{\mathcal C_k}(n)(X-2n) \mathfrak S_b^{(3)}(n) = \kappa_{b,\mathcal D}\, \overline{\mathfrak S}^{(3)}_{b,\mathcal D}\,X|\mathcal C_k| + O_{b,\mathcal D}(|\mathcal C_k|). 
\end{equation} 
Moreover, the leading coefficient 
\[ \kappa_{b,\mathcal D}\, \overline{\mathfrak S}^{(3)}_{b,\mathcal D}>0 \] 
if and only if 
\[ \min\mathcal D<\frac{b-1}{2} \qquad\text{and}\qquad \mathcal D\cap \gcd(b,6)\mathbb Z\neq\emptyset. \] 
\end{proposition}
\begin{proof}
\noindent \textbf{Existence of the archimedean factor.}
Let
\[
        f(x):=
        \begin{cases}
        1-2x, & 0\le x<1/2,\\
        0, & 1/2\le x\le 1.
        \end{cases}
\]
Then
\[
        \kappa_j
        =
        \frac1{|\mathcal C_j|}
        \sum_{m\in\mathcal C_j}
        f\!\left(\frac{m}{b^j}\right).
\]
Every element of \(\mathcal C_{j+1}\) has a unique representation
\[
        n=a+bm,
        \qquad
        a\in\mathcal D,
        \qquad
        m\in\mathcal C_j.
\]
Hence
\[
\kappa_{j+1}
=
\frac{1}{|\mathcal D|\,|\mathcal C_j|}
\sum_{a\in\mathcal D}
\sum_{m\in\mathcal C_j}
f\!\left(
        \frac{m}{b^j}+\frac{a}{b^{j+1}}
\right).
\]
Since \(f\) is Lipschitz and \(0\le a\le b-1\), it follows that
\[
        |\kappa_{j+1}-\kappa_j|
        \ll_b b^{-j}.
\]
Thus \((\kappa_j)_{j\ge1}\) is Cauchy, and its limit
\(\kappa_{b,\mathcal D}\) exists.  

Clearly, we also have
\[
        |\kappa_j-\kappa_{b,\mathcal D}|
        \ll_b b^{-j}.
\]

Define
\[
        H_j
        :=
        \sum_{\substack{m\in\mathcal C_j\\0\le m<b^j/2}}
        (b^j-2m).
\]
By the definition of \(\kappa_j\),
\[
        H_j=b^j|\mathcal C_j|\kappa_j.
\]
Therefore
\begin{equation}\label{eq:Hj-asymptotic}
        H_j
        =
        \kappa_{b,\mathcal D}\,b^j|\mathcal C_j|
        +
        O_{b,\mathcal D}(|\mathcal C_j|).
\end{equation}

\medskip
\noindent\textbf{Weighted average.}
Write
\[
        n=a+bm,
        \qquad
        a\in\mathcal D,
        \qquad
        m\in\mathcal C_{k-1}.
\]
Since every modulus \(d\) appearing in the definition of
\(\mathfrak S_b^{(3)}\) divides \(b\),
\(\mathfrak S_b^{(3)}(n)\) depends only on the least significant
base-\(b\) digit \(a\).  Thus
\[
        \mathfrak S_b^{(3)}(n)
        =
        \mathfrak S_b^{(3)}(a).
\]
Consequently,
\[
\begin{aligned}
&\sum_{0\le n<X/2}
\mathbf 1_{\mathcal C_k}(n)(X-2n)
\mathfrak S_b^{(3)}(n)                                      \\
&\qquad =
\sum_{a\in\mathcal D}
\mathfrak S_b^{(3)}(a)
\sum_{\substack{m\in\mathcal C_{k-1}\\a+bm<X/2}}
(X-2a-2bm).
\end{aligned}
\]

For fixed \(a\in\mathcal D\),
\[
        a+bm<\frac X2
        \quad\Longleftrightarrow\quad
        m<\frac{b^{k-1}}2-\frac ab.
\]
Since \(0\le a/b<1\), this range differs from
\(m<b^{k-1}/2\) by at most one integer.  Moreover,
\[
        X-2a-2bm
        =
        b(b^{k-1}-2m)-2a.
\]
The possible boundary term is \(O_b(1)\), while the contribution of
\(-2a\) is \(O_b(|\mathcal C_{k-1}|)\).  Hence, uniformly in
\(a\in\mathcal D\),
\[
        \sum_{\substack{m\in\mathcal C_{k-1}\\a+bm<X/2}}
        (X-2a-2bm)
        =
        bH_{k-1}
        +
        O_b(|\mathcal C_{k-1}|).
\]
It follows that
\[
\begin{aligned}
&\sum_{0\le n<X/2}
\mathbf 1_{\mathcal C_k}(n)(X-2n)
\mathfrak S_b^{(3)}(n)                                      \\
&\qquad =
bH_{k-1}
\sum_{a\in\mathcal D}\mathfrak S_b^{(3)}(a)
+
O_{b,\mathcal D}(|\mathcal C_k|).
\end{aligned}
\]

By \eqref{eq:Hj-asymptotic},
\[
        bH_{k-1}
        =
        \kappa_{b,\mathcal D}X|\mathcal C_{k-1}|
        +
        O_{b,\mathcal D}(|\mathcal C_{k-1}|).
\]
Also, by the definition of
\(\overline{\mathfrak S}^{(3)}_{b,\mathcal D}\),
\[
        \sum_{a\in\mathcal D}\mathfrak S_b^{(3)}(a)
        =
        |\mathcal D|\,
        \overline{\mathfrak S}^{(3)}_{b,\mathcal D},
\]
while
\[
        |\mathcal C_k|
        =
        |\mathcal D|\,|\mathcal C_{k-1}|.
\]
Combining these identities gives
\[
\begin{aligned}
&\sum_{0\le n<X/2}
\mathbf 1_{\mathcal C_k}(n)(X-2n)
\mathfrak S_b^{(3)}(n)                                      \\
&\qquad =
\kappa_{b,\mathcal D}\,
\overline{\mathfrak S}^{(3)}_{b,\mathcal D}\,
X|\mathcal C_k|
+
O_{b,\mathcal D}(|\mathcal C_k|).
\end{aligned}
\]
This proves \eqref{eq:weighted-three-term-local-factor}.

\medskip
\noindent\textbf{Positivity of the leading coefficient.}
{It remains to determine when the product of the arithmetic factor
$\overline{\mathfrak S}^{(3)}_{b,\mathcal D}$ and the archimedean factor
$\kappa_{b,\mathcal D}$ is positive.  We treat the two factors separately.}

\medskip
\noindent\underline{Arithmetic factor.}
Concerning this factor, we have:
\begin{lemma}[Evaluation of the three-prime local factor]
\label{lem:three-term-local-factor-evaluation}
Let
\[
        S_0:=S_{1,\mathrm{full}}(1).
\]
Then
\begin{equation}\label{eq:S0-three-term}
        S_0
        =
        2\prod_{p>2}
        \left(1-\frac{1}{(p-1)^2}\right)
        >0.
\end{equation}
Moreover,
\[
        S_{1,\mathrm{full}}\!\left(\frac{\ell}{d}\right)=0
\]
unless \(d\) is squarefree, while for squarefree \(d\),
\begin{equation}\label{eq:S1full-local-factorized}
        S_{1,\mathrm{full}}\!\left(\frac{\ell}{d}\right)
        =
        S_0
        \prod_{\substack{p\mid d\\p>2}}
        \frac{2}{(p-1)(p-2)}.
\end{equation}
Consequently,
\begin{equation}\label{eq:S3-Ramanujan-product}
        \mathfrak S_b^{(3)}(n)
        =
        S_0
        \prod_{p\mid b}
        \left(1+\beta_p c_p(n)\right),
\end{equation}
where
\[
        \beta_2:=1,
        \qquad
        \beta_p:=\frac{2}{(p-1)(p-2)}
        \quad (p>2).
\]
In particular,
\[
        \mathfrak S_b^{(3)}(n)\ge0
        \qquad(n\in\mathbb Z),
\]
and
\begin{equation}\label{eq:S3-positivity}
        \mathfrak S_b^{(3)}(n)>0
        \quad\Longleftrightarrow\quad
        \gcd(b,6)\mid n.
\end{equation}
Hence the average 
\[
        \overline{\mathfrak S}_{b,\mathcal D}^{(3)}>0
        \quad\Longleftrightarrow\quad
        \mathcal D\cap\gcd(b,6)\mathbb Z\neq\emptyset.
\]
\end{lemma}
{The proof is deferred to \cref{sec:three-term-local-factor}.}

\medskip \noindent\underline{Archimedean factor.}
Let \[ a_*:=\min\mathcal D. \] Suppose first that \[ a_*\ge\frac{b-1}{2}. \] Every \(m\in\mathcal C_j\) satisfies \[ m\ge a_*\frac{b^j-1}{b-1}. \] 
Thus, if \(m<b^j/2\), then \(m\) can only lie within \(O_b(1)\) of the boundary \(b^j/2\), and hence \[ b^j-2m\ll_b1. \] Consequently, \[ H_j\ll_b|\mathcal C_j|, \qquad \kappa_j\ll_b b^{-j}, \] and therefore \[ \kappa_{b,\mathcal D}=0. \] 
Conversely, suppose \[ a_*<\frac{b-1}{2}. \] Choose \(J\) sufficiently large that \[ \frac{a_*}{b-1}(1-b^{-J})+b^{-J}<\frac12. \] 
For \(j\ge J\), consider those elements of \(\mathcal C_j\) whose \(J\) most significant base-\(b\) digits are all equal to \(a_*\). 
There are \(r^{j-J}\) such elements. For every such \(m\), \[ \frac{m}{b^j} \le \frac{a_*}{b-1}(1-b^{-J})+b^{-J} <\frac12. \] 
Hence there exists \(\sigma=\sigma_{b,\mathcal D}>0\) such that \[ b^j-2m\ge 2\sigma b^j \] for all these \(m\). 
It follows that \[ \kappa_j = \frac{H_j}{b^j|\mathcal C_j|} \ge 2\sigma\,r^{-J}>0 \] uniformly for \(j\ge J\). 
Passing to the limit gives \[ \kappa_{b,\mathcal D}>0. \] Thus \[ \kappa_{b,\mathcal D}>0 \quad\Longleftrightarrow\quad \min\mathcal D<\frac{b-1}{2}. \] 

Combining the arithmetic and archimedean criteria gives \[ \kappa_{b,\mathcal D}\,
\overline{\mathfrak S}^{(3)}_{b,\mathcal D}>0 \] if and only if \[ \min\mathcal D<\frac{b-1}{2} \qquad\text{and}\qquad \mathcal D\cap\gcd(b,6)\mathbb Z\neq\emptyset, \]
as claimed.
\end{proof}

We are now {in a position} to complete the proof of Theorem \ref{thm:intro-three-prime}. 
\begin{proof}[Proof of Theorem \ref{thm:intro-three-prime}]
Recall the following decomposition:
\[
        T_k^{(3)}
        =
        T_k^{(3)}(\mathfrak M_A^{(N)})
        +
        T_k^{(3)}(\mathfrak m_A^{(N)}).
\]
Strictly speaking, for a prescribed logarithmic saving \(A>0\), the major/minor
arc parameter is chosen sufficiently large in terms of this \(A\).

By the three-term minor-arc estimate, \cref{thm:minor-arc-final},
\[
        T_k^{(3)}(\mathfrak m_A^{(N)})
        \ll_{b,\mathcal D,A}
        |\mathcal C_k|X(\log X)^{-A}.
\]
On the major arcs, \cref{prop:three-term-major-projected-local-factor}
gives
\[
        T_k^{(3)}(\mathfrak M_A^{(N)})
        =
        \sum_{0\le n<X/2}
        \mathbf 1_{\mathcal C_k}(n)(X-2n)\mathfrak S_b^{(3)}(n)
        +
        O_{b,\mathcal D,A}\!\left(
        |\mathcal C_k|X(\log X)^{-A}
        \right).
\]
Finally, \cref{prop:weighted-three-term-local-factor} evaluates the
{arithmetic local factor against the archimedean weight}:
\[
        \sum_{0\le n<X/2}
        \mathbf 1_{\mathcal C_k}(n)(X-2n)\mathfrak S_b^{(3)}(n)
        =
        \kappa_{b,\mathcal D}\, \overline{\mathfrak S}^{(3)}_{b,\mathcal D}\,X|\mathcal C_k|
        +
        O_{b,\mathcal D}(|\mathcal C_k|).
\]
The final error term \(O_{b,\mathcal D}(|\mathcal C_k|)\) is absorbed into
\[
        O_{b,\mathcal D,A}\!\left(
        |\mathcal C_k|X(\log X)^{-A}
        \right)
\]
for \(X\) sufficiently large.  Combining the preceding three estimates gives
the stated asymptotic formula.  The non-vanishing criterion for
\(\kappa_{b,\mathcal D}\, \overline{\mathfrak S}^{(3)}_{b,\mathcal D}\) is exactly the one proved in
\cref{prop:weighted-three-term-local-factor}.
\end{proof}

\section{Alternative-range variants}
\label{sec:alternative-range}

For \(j\in\{1,2,3\}\), put
\[
        \Lambda_{jX}(m):=\Lambda(m)\mathbf 1_{[0,jX)}(m),
        \qquad
        \widehat\Lambda_{jX}(\theta)
        :=
        \sum_{0\le m<jX}\Lambda(m)e(m\theta).
\]
All prime exponential-sum estimates used above remain valid uniformly for
\(j\in\{1,2,3\}\).  In particular, if \(N\asymp X\), \(d\sim D\),
\(|\eta|\sim B\), and \(D^2B\ll N\), then
\begin{equation}\label{eq:fixed-dilation-minor}
\left|
\widehat\Lambda_{jX}
\left(\frac{\ell}{d}+\frac{\eta}{N}\right)
\right|
\ll
\left(
X^{4/5}
+\frac{X}{(DB)^{1/2}}
+X^{1/2}(DB)^{1/2}
\right)(\log X)^4.
\end{equation}
The corresponding rational-endpoint estimate, major-arc approximation, and
Parseval bound also hold uniformly in \(j\).  Thus all minor-arc estimates and
all arithmetic major-arc reductions from the preceding sections carry over
without change.

\begin{proof}[Proofs of
\cref{thm:intro-two-prime-alternative-range,%
thm:intro-three-prime-alternative-range}]
For the two-term problem, take \(N=3X\).  Fourier inversion gives
\[
        T_{k,\mathrm{alt}}^{(2)}
        =
        \frac1N
        \sum_{\alpha\in\mathcal G_N}
        \widehat\Lambda_X(-\alpha)
        \widehat\Lambda_{2X}(\alpha)
        \widehat C_k(-\alpha).
\]
There is no wraparound since \(m,n<X\) implies \(m+n<2X<N\).
By the fixed-dilation observation above, the proof of the two-term
major- and minor-arc estimates applies verbatim.  The arithmetic local factor
is therefore still \(\mathfrak S_b^{(2)}(n)\).  The only change is that the
completed archimedean weight is
\[
        \#\{0\le m<X:m+n<2X\}=X
        \qquad(0\le n<X).
\]
{Thus the normalized archimedean factor is \(1\).}
Consequently,
\[
\begin{aligned}
        T_{k,\mathrm{alt}}^{(2)}
        &=
        X\sum_{0\le n<X}
        \mathbf 1_{\mathcal C_k}(n)\mathfrak S_b^{(2)}(n)
        +
        O_{b,\mathcal D,A}\!\left(
        |\mathcal C_k|X(\log X)^{-A}
        \right) \\
        &=
        \overline{\mathfrak S}^{(2)}_{b,\mathcal D}
        X|\mathcal C_k|
        +
        O_{b,\mathcal D,A}\!\left(
        |\mathcal C_k|X(\log X)^{-A}
        \right).
\end{aligned}
\]

For the three-term problem, take \(N=4X\) and define
\[
        \widehat H_{\Lambda,\mathrm{alt}}(\alpha)
        :=
        \frac1N\sum_{\beta\in\mathcal G_N}
        \widehat\Lambda_X(\beta)
        \widehat\Lambda_{2X}(\alpha-2\beta)
        \widehat\Lambda_{3X}(\beta-\alpha).
\]
Then
\[
        T_{k,\mathrm{alt}}^{(3)}
        =
        \frac1N\sum_{\alpha\in\mathcal G_N}
        \widehat H_{\Lambda,\mathrm{alt}}(\alpha)
        \widehat C_k(\alpha).
\]
The proof of the three-point minor-arc lemma depends only on the three
frequencies and the corresponding prime-sum estimates, so it applies
unchanged.  The rational major-arc coefficients are also unchanged, and hence
the arithmetic local factor is again \(\mathfrak S_b^{(3)}(n)\).

The completed archimedean weight is
\[
        W_X(h)
        :=
        \#\left\{
        0\le m<X:
        0\le m+h<2X,\;
        0\le m+2h<3X
        \right\}.
\]
This is a compactly supported piecewise-affine function with
\[
        W_X(n)=X
        \qquad(0\le n<X).
\]
{Thus the normalized archimedean factor is again \(1\).}
Its second discrete difference has bounded \(\ell^1\)-norm, giving the same
quadratic decay needed to complete the major-arc sums.  Therefore
\[
\begin{aligned}
        T_{k,\mathrm{alt}}^{(3)}
        &=
        X\sum_{0\le n<X}
        \mathbf 1_{\mathcal C_k}(n)\mathfrak S_b^{(3)}(n)
        +
        O_{b,\mathcal D,A}\!\left(
        |\mathcal C_k|X(\log X)^{-A}
        \right) \\
        &=
        \overline{\mathfrak S}^{(3)}_{b,\mathcal D}
        X|\mathcal C_k|
        +
        O_{b,\mathcal D,A}\!\left(
        |\mathcal C_k|X(\log X)^{-A}
        \right),
\end{aligned}
\]
as claimed.
\end{proof}

{{
\section{From von Mangoldt weights to genuine prime configurations}
\label{sec:true-prime}

The four asymptotic formulae established above are stated with von Mangoldt
weights.  Since \(\Lambda\) is supported on prime powers, we must remove the
terms in which at least one prime variable is a proper prime power.  To obtain
nontrivial configurations, we must also exclude the diagonal \(n=0\), but only
when it is present.  If \(0\notin\mathcal C_k\), then \(n=0\) does not occur in
any of the four counting functions, and every restricted-digit difference is
already positive.

\begin{lemma}[Removing proper prime powers and the diagonal]
\label{lem:remove-prime-powers}
In either the original or the alternative pair range, the total contribution
from configurations in which at least one of \(m\) and \(m+n\) is a proper
prime power is
\[
        O\!\left(
        |\mathcal C_k|X^{1/2}(\log X)^2
        \right).
\]
In either the original or the alternative three-term range, the total
contribution from configurations in which at least one of
\[
        m,\qquad m+n,\qquad m+2n
\]
is a proper prime power is
\[
        O\!\left(
        |\mathcal C_k|X^{1/2}(\log X)^3
        \right).
\]

If \(0\in\mathcal C_k\), then the diagonal \(n=0\) contributes
\[
        O\!\left(X(\log X)^2\right)
\]
to either pair count and
\[
        O\!\left(X(\log X)^3\right)
\]
to either three-term count.  If \(0\notin\mathcal C_k\), no diagonal term
occurs.  All of these quantities are \(o(X|\mathcal C_k|)\).
\end{lemma}

\begin{proof}
All von Mangoldt arguments occurring in the four counting functions are less
than \(3X\).  Let
\[
        \mathcal Q_X
        :=
        \left\{
        p^\nu\le 3X:
        p\ \text{prime},\ \nu\ge2
        \right\}
\]
be the set of proper prime powers in this range.  We have
\[
\begin{aligned}
        |\mathcal Q_X|
        \le
        \sum_{2\le \nu\le \log_2(3X)}
        (3X)^{1/\nu}                      \ll
        X^{1/2}+X^{1/3}\log X
        \ll
        X^{1/2}.
\end{aligned}
\]

Consider first either of the pair counts.  After choosing
\(n\in\mathcal C_k\), a proper prime power \(q\in\mathcal Q_X\), and which
one of the two linear forms
\[
        m,\qquad m+n
\]
is equal to \(q\), the starting point \(m\) is determined.  Ignoring the
range restrictions can only enlarge the count.  Since each von Mangoldt
factor is \(O(\log X)\), the total contribution of such configurations is
\[
        \ll
        |\mathcal C_k|\,|\mathcal Q_X|\,(\log X)^2
        \ll
        |\mathcal C_k|X^{1/2}(\log X)^2.
\]

The three-term case is identical.  After choosing \(n\), a proper prime power,
and which one of
\[
        m,\qquad m+n,\qquad m+2n
\]
takes that value, \(m\) is determined.  Since there are now three von Mangoldt
factors, the total contribution is
\[
        \ll
        |\mathcal C_k|\,|\mathcal Q_X|\,(\log X)^3
        \ll
        |\mathcal C_k|X^{1/2}(\log X)^3.
\]

It remains to consider the diagonal.  If \(0\notin\mathcal C_k\), then the
factor \(\mathbf 1_{\mathcal C_k}(n)\) excludes \(n=0\), so there is nothing
to remove.  If \(0\in\mathcal C_k\), then the diagonal contributions to the
pair and three-term counts are respectively
\[
        \sum_{0\le m<X}\Lambda(m)^2
        \ll
        X(\log X)^2
\]
and
\[
        \sum_{0\le m<X}\Lambda(m)^3
        \ll
        X(\log X)^3.
\]

Finally,
\[
        |\mathcal C_k|
        =
        r^k
        =
        X^{\log_b r},
        \qquad r\ge2.
\]
Hence
\[
        X^{1/2}(\log X)^3=o(X)
\]
and
\[
        X(\log X)^3
        =
        o\!\left(X|\mathcal C_k|\right).
\]
The same conclusions hold with the smaller logarithmic powers appearing
above, and the proof is complete.
\end{proof}

\begin{corollary}[Nontrivial genuine-prime configurations]
\label{cor:genuine-prime-configurations}
Consider either the original or the alternative pair range.  If the leading
constant in the corresponding asymptotic formula is positive, then, for all
sufficiently large \(k\), the number of choices of \((m,n)\) in that range
such that
\[
        n\in\mathcal C_k\setminus\{0\},
        \qquad
        m,\ m+n\ \text{are prime},
\]
is
\[
        \gg_{b,\mathcal D}
        \frac{X|\mathcal C_k|}{(\log X)^2}.
\]

Similarly, consider either the original or the alternative three-term range.
If the leading constant in the corresponding asymptotic formula is positive,
then the number of choices of \((m,n)\) in that range such that
\[
        n\in\mathcal C_k\setminus\{0\},
        \qquad
        m,\ m+n,\ m+2n\ \text{are prime},
\]
is
\[
        \gg_{b,\mathcal D}
        \frac{X|\mathcal C_k|}{(\log X)^3}.
\]
In particular, the three-term results produce nontrivial arithmetic
progressions of primes with positive restricted-digit common difference.
\end{corollary}

\begin{proof}
Suppose first that the relevant pair constant is positive.  The corresponding
asymptotic formula, together with
\cref{lem:remove-prime-powers}, shows that the von Mangoldt-weighted
contribution from configurations with \(n>0\) and both prime variables genuine
primes is
\[
        \gg_{b,\mathcal D}X|\mathcal C_k|.
\]
Each such configuration has weight at most
\[
        (\log(2X))^2\ll(\log X)^2.
\]
Dividing by this upper bound gives the asserted unweighted pair count.

The same argument applies to the three-term counts.  After removing proper
prime powers and the diagonal, the remaining weighted contribution is
\[
        \gg_{b,\mathcal D}X|\mathcal C_k|,
\]
and each genuine-prime progression has weight at most
\[
        (\log(3X))^3\ll(\log X)^3.
\]
This proves the three-term lower bound.
\end{proof}
}}

\section{A transfer-operator bound for the Fourier exponent}
\label{sec:transfer-model}

The main results of the paper are stated in terms of the admissibility exponent
\(\alpha_{b,\mathcal D}\) appearing in the localized hybrid Fourier estimate
\eqref{eq:hybrid-digit-general}.  The purpose of this section is to explain one
standard way of understanding and estimating such exponents.  The material here
should be viewed as a useful supplement to the main argument: it gives intuition
for the size of \(\alpha_{b,\mathcal D}\) and provides a practical way to
approximate admissible exponents in concrete digit examples.

The point of view is to encode the product structure of \(\widehat C_k\) by a
positive transfer operator.  The spectral radius of this operator gives an
upper bound for the growth of the relevant localized \(\ell^1\)-norms.  This
is closely related to the Riesz-product and Fourier analysis of restricted-digit
sets in the work of Maynard
\cite{MaynardRestrictedDigitsInv,MaynardRestrictedDigits}.

Transfer-operator methods are also standard in the study of dynamically defined
Cantor-type sets and fractal dimension estimates
\cite{Baladi,JenkinsonPollicott}.  We include the details here because they give
a convenient way to test the admissibility hypothesis for explicit digit sets.

Let
\[
        \mathcal D\subset\{0,1,\dots,b-1\},
        \qquad
        r:=|\mathcal D|.
\]
Define the one-step digit polynomial
\[
        P_{\mathcal D}(x):=\sum_{a\in\mathcal D}e(ax).
\]
Then the Fourier transform of the fixed-length restricted-digit set satisfies
\[
        |\widehat C_k(\theta)|
        =
        \prod_{j=0}^{k-1}|P_{\mathcal D}(b^j\theta)|.
\]
We associate to \((b,\mathcal D)\) the normalized positive transfer operator $\mathcal{L}_{b,\mathcal{D}}$:
\[
        (\mathcal L_{b,\mathcal D}F)(x)
        :=
        \frac1r
        \sum_{m=0}^{b-1}
        \left|
        P_{\mathcal D}\!\left(\frac{x+m}{b}\right)
        \right|
        F\!\left(\frac{x+m}{b}\right),
        \qquad x\in\mathbb T.
\]
Let $\rho(\mathcal L_{b,\mathcal D})$
denote its spectral radius on \(C(\mathbb T)\).

The goal of this section is to show that, for every \(\varepsilon>0\), the
exponent
\[
        \log_b\!\bigl(\rho(\mathcal L_{b,\mathcal D})\bigr)+\varepsilon
\]
is admissible in the sense of \eqref{eq:hybrid-digit-general}.  We first
record a standard sampling consequence of a uniform grid bound.

\begin{lemma}[A sampling consequence of a uniform grid bound]
\label{lem:sampling-badic-grid}
Let \(M\ge1\) be an integer and let \(0<c\le1\). Let \(F\) be a
trigonometric polynomial whose frequencies have modulus \(<M\).
Suppose that a finite multiset \(\Xi\subset\mathbb T\) has at most
\(K\) points, counted with multiplicity, in every interval of length
\(c/M\). Then
\[
        \sum_{\xi\in\Xi}|F(\xi)|
        \ll_c KM\int_{\mathbb T}|F(u)|\,du
        \ll_c K
        \sup_{x\in\mathbb T}
        \sum_{a=0}^{M-1}
        \left|F\!\left(\frac{x+a}{M}\right)\right|.
\]
\end{lemma}

\begin{proof}
Let \(V_M\) be a de la Vall\'ee Poussin kernel which reproduces all
frequencies of modulus \(<M\) and satisfies
\[
        |V_M(u)|
        \ll
        \frac{M}{1+M^2\|u\|^2}.
\]
Then \(F=F*V_M\). For fixed \(u\in\mathbb T\), decompose \(\mathbb T\) into the sets
\[
        \left\{v\in\mathbb T:
        \frac{jc}{M}\le \|v-u\|<
        \frac{(j+1)c}{M}\right\},
        \qquad j\ge0.
\]
Each such set is the union of at most two intervals of length \(c/M\),
and hence contains at most \(2K\) points of \(\Xi\), counted with
multiplicity. Therefore
\[
\begin{aligned}
        \sum_{\xi\in\Xi}|V_M(\xi-u)|
        &\ll
        KM\sum_{j\ge0}\frac{1}{1+c^2j^2}\\
        &\ll_c KM.
\end{aligned}
\]
Using the convolution representation and interchanging the finite sum
with the integral, we obtain
\[
\begin{aligned}
        \sum_{\xi\in\Xi}|F(\xi)|
        &\le
        \int_{\mathbb T}|F(u)|
        \sum_{\xi\in\Xi}|V_M(\xi-u)|\,du \\
        &\ll_c
        KM\int_{\mathbb T}|F(u)|\,du.
\end{aligned}
\]
Finally, by the change of variables \(u=(x+a)/M\),
\[
\begin{aligned}
        \int_0^1
        \sum_{a=0}^{M-1}
        \left|F\!\left(\frac{x+a}{M}\right)\right|\,dx
        &=
        \sum_{a=0}^{M-1}
        M\int_{a/M}^{(a+1)/M}|F(u)|\,du \\
        &=
        M\int_{\mathbb T}|F(u)|\,du.
\end{aligned}
\]
The integral on the left is bounded above by the corresponding
supremum over \(x\), which proves the second estimate.
\end{proof}

\begin{proposition}[A transfer-operator sufficient condition for admissibility]
\label{prop:transfer-controls-hybrid}
Let \(\rho(\mathcal L_{b,\mathcal D})\) be the spectral radius.
Then, for every \(\varepsilon>0\), the localized hybrid estimate
\begin{equation}\label{eq:hybrid-digit-transfer}
\sum_{\substack{d\sim D}}
\sum_{\substack{1\le\ell\le d\\(\ell,d)=1}}
\sum_{\substack{|\eta|<B\\ N\ell/d+\eta\in\mathbb Z}}
\left|
\widehat C_k\!\left(\frac{\ell}{d}+\frac{\eta}{N}\right)
\right|
\ll_{\varepsilon,b,\mathcal D}
|\mathcal C_k|\,(D^2B)^{\log_b(\rho(\mathcal L_{b,\mathcal D}))+\varepsilon}
\end{equation}
holds uniformly for \(D,B\ge1\) in the range
\[
        3X\le N\le4X,\qquad D^2B\le8X.
\]
Consequently, in the notation of \eqref{eq:hybrid-digit-general}, one may take
\[
        \alpha_{b,\mathcal D}
        =
        \log_b(\rho(\mathcal L_{b,\mathcal D}))+\varepsilon
\]
for any fixed \(\varepsilon>0\).
\end{proposition}

\begin{proof}
The product formula
\[
        \widehat C_{k_0}(\theta)
        =
        \prod_{j=0}^{k_0-1}P_{\mathcal D}(b^j\theta)
\]
implies that the normalized \(b\)-adic \(\ell^1\)-sums of \(\widehat C_{k_0}\)
are generated by the iterates of \(\mathcal L_{b,\mathcal D}\).  More
precisely, for every \(k_0\ge1\) and every \(x\in\mathbb T\), one has
\begin{equation}
\label{eq:transfer-l1-identity-general}
        \frac1{r^{k_0}}
        \sum_{a=0}^{b^{k_0}-1}
        \left|
        \widehat C_{k_0}\!\left(\frac{x+a}{b^{k_0}}\right)
        \right|
        =
        ((\mathcal L_{b,\mathcal D})^{k_0}\mathbf 1)(x).
\end{equation}
{Taking \(x=0\) in
\eqref{eq:transfer-l1-identity-general} and retaining the term \(a=0\) shows
that \(\|(\mathcal L_{b,\mathcal D})^{k_0}\|_{L^\infty\to L^\infty}\ge1\)
for every \(k_0\), and hence \(\rho(\mathcal L_{b,\mathcal D})\ge1\).}
By the spectral-radius formula, for every \(\varepsilon>0\) there is a
constant \(C_{\varepsilon,b,\mathcal D}\) such that
\[
        \|(\mathcal L_{b,\mathcal D})^{k_0}\|_{L^\infty\to L^\infty}
        \le
        C_{\varepsilon,b,\mathcal D}
        b^{{k_0}(\log_b(\rho(\mathcal L_{b,\mathcal D}))+\varepsilon)}
        \qquad({k_0}\ge1).
\]
Indeed, Gelfand's formula gives
\[
        \lim_{k_0\to\infty}
        \|(\mathcal L_{b,\mathcal D})^{k_0}\|_{L^\infty\to L^\infty}^{1/k_0}
        =\rho(\mathcal L_{b,\mathcal D});
\]
the finitely many remaining values of \(k_0\) are absorbed into
\(C_{\varepsilon,b,\mathcal D}\).  Combining this with
\eqref{eq:transfer-l1-identity-general} yields
\begin{equation}
\label{eq:badic-l1-bound-general}
        \sup_{x\in\mathbb T}
        \sum_{a=0}^{b^{k_0}-1}
        \left|
        \widehat C_{k_0}\!\left(\frac{x+a}{b^{k_0}}\right)
        \right|
        \ll_{\varepsilon,b,\mathcal D}
        r^{k_0} b^{{k_0}(\log_b(\rho(\mathcal L_{b,\mathcal D}))+\varepsilon)}.
\end{equation}

At this point we follow Maynard's argument
\cite[Lemmas 5.1--5.3]{MaynardRestrictedDigits}.  More precisely, the
passage from the \(b\)-adic \(\ell^1\)-estimate
\eqref{eq:badic-l1-bound-general} to the hybrid estimate is precisely the
two-block argument used in Maynard's proof of his hybrid lemma.  The only
difference is that here the \(b\)-adic estimate has been obtained from the
transfer operator.  The sampling lemma above gives the same argument
uniformly for the shifted grids \(3X\le N\le4X\).

Since \(D^2B\le8X=8b^k\), we may choose integers \(k_1,k_2\ge0\), after
changing either of them by at most \(O_b(1)\), such that
\[
        b^{k_1}\asymp_b D^2,
        \qquad
        b^{k_2}\asymp_b B,
        \qquad
        k_1+k_2\le k.
\]
We use the convention \(\widehat C_0\equiv1\).  Splitting off the first
\(k_1\) and last \(k_2\) digit blocks gives
\[
        |\widehat C_k(\theta)|
        \le
        r^{k-k_1-k_2}
        |\widehat C_{k_1}(\theta)|
        |\widehat C_{k_2}(b^{k-k_2}\theta)|.
\]

For fixed \(d,\ell\), the frequencies
\[
        b^{k-k_2}
        \left(\frac{\ell}{d}+\frac{\eta}{N}\right),
        \qquad
        |\eta|<B,\quad \frac{N\ell}{d}+\eta\in\mathbb Z,
\]
have spacing
\[
        \frac{b^{k-k_2}}{N}
        =
        \frac{X}{Nb^{k_2}}
        \asymp b^{-k_2}
\]
and consist of \(O(B)=O_b(b^{k_2})\) points.  Their total unwrapped span
is therefore \(O_b(1)\), so reduction modulo one preserves bounded local
multiplicity at scale \(b^{-k_2}\).  Hence
\cref{lem:sampling-badic-grid} and
\eqref{eq:badic-l1-bound-general} give
\[
\sup_{\substack{d\sim D\\1\le\ell\le d\\(\ell,d)=1}}
\sum_{\substack{|\eta|<B\\N\ell/d+\eta\in\mathbb Z}}
\left|
\widehat C_{k_2}\!\left(
b^{k-k_2}
\left(\frac{\ell}{d}+\frac{\eta}{N}\right)
\right)
\right|
\ll_{\varepsilon,b,\mathcal D}
r^{k_2}
b^{k_2(\log_b(\rho(\mathcal L_{b,\mathcal D}))+\varepsilon)}.
\]

For the first block, for each reduced fraction \(\ell/d\), \(d\sim D\),
choose a point in the closed interval
\[
        \left[\frac{\ell}{d}-\frac{B}{N},
        \frac{\ell}{d}+\frac{B}{N}\right]
\]
at which \(|\widehat C_{k_1}|\) is maximal.
Distinct such fractions are \(\gg D^{-2}\)-separated, while
\[
        \frac{B}{N}\ll D^{-2}.
\]
Thus the selected points have bounded local multiplicity at scale
\(D^{-2}\asymp_b b^{-k_1}\).  A second application of
\cref{lem:sampling-badic-grid} and
\eqref{eq:badic-l1-bound-general} yields
\[
\sum_{d\sim D}
\sum_{\substack{1\le\ell\le d\\(\ell,d)=1}}
\sup_{|\eta|<B}
\left|
\widehat C_{k_1}\!\left(
\frac{\ell}{d}+\frac{\eta}{N}
\right)
\right|
\ll_{\varepsilon,b,\mathcal D}
r^{k_1}
b^{k_1(\log_b(\rho(\mathcal L_{b,\mathcal D}))+\varepsilon)}.
\]

Combining the two block estimates with the trivial bound on the middle
digit factors, we obtain
\[
\begin{aligned}
\sum_{d\sim D}
\sum_{\substack{1\le\ell\le d\\(\ell,d)=1}}
\sum_{\substack{|\eta|<B\\N\ell/d+\eta\in\mathbb Z}}
\left|
\widehat C_k\!\left(
\frac{\ell}{d}+\frac{\eta}{N}
\right)
\right|
&{\ll_{\varepsilon,b,\mathcal D}
r^k
b^{(k_1+k_2)
(\log_b(\rho(\mathcal L_{b,\mathcal D}))+\varepsilon)}}\\
&{\ll_{\varepsilon,b,\mathcal D}
|\mathcal C_k|
(D^2B)^{\log_b(\rho(\mathcal L_{b,\mathcal D}))+\varepsilon}.}
\end{aligned}
\]
This proves \eqref{eq:hybrid-digit-transfer}.
\end{proof}

Thus the transfer-operator model gives a sufficient condition for
admissibility.  To obtain explicit consequences, we do not attempt to compute
\(\rho(\mathcal L_{b,\mathcal D})\) exactly.  Instead, we bound it by
the one-step \(L^\infty\)-operator norm, which positivity turns into an
explicit row sum.

\begin{lemma}[Row-sum bound for the transfer operator]
\label{lem:transfer-row-sum}
Let \(\mathcal D\subset\{0,\dots,b-1\}\), with \(r=|\mathcal D|\). Then
\[
        \rho(\mathcal L_{b,\mathcal D})
        \le
        \frac1r
        \sup_{x\in\mathbb T}
        \sum_{m=0}^{b-1}
        \left|
        P_{\mathcal D}\!\left(\frac{x+m}{b}\right)
        \right|.
\]
\end{lemma}

\begin{proof}
For every bounded operator,
\[
        \rho(\mathcal L_{b,\mathcal D})
        \le
        \|\mathcal L_{b,\mathcal D}\|_{L^\infty\to L^\infty}.
\]
Since \(\mathcal L_{b,\mathcal D}\) is positive,
\[
        |\mathcal L_{b,\mathcal D}F|
        \le
        \mathcal L_{b,\mathcal D}|F|
        \le
        \|F\|_\infty\,\mathcal L_{b,\mathcal D}\mathbf 1.
\]
Testing the operator on \(\mathbf 1\) gives the reverse inequality, and
therefore
\[
        \|\mathcal L_{b,\mathcal D}\|_{L^\infty\to L^\infty}
        =
        \|\mathcal L_{b,\mathcal D}\mathbf 1\|_\infty.
\]
By definition,
\[
        (\mathcal L_{b,\mathcal D}\mathbf 1)(x)
        =
        \frac1r
        \sum_{m=0}^{b-1}
        \left|
        P_{\mathcal D}\!\left(\frac{x+m}{b}\right)
        \right|,
\]
which proves the claim.
\end{proof}

The remaining subsections apply this lemma to several natural families of
digit sets.  The resulting estimates are intended as transparent, easily
checkable sufficient bounds rather than sharp evaluations of the spectral
radius.

\subsection{Consecutive digit intervals}
\label{subsec:transfer-consecutive-digit-intervals}

Let
\[
        \mathcal D=\{a,a+1,\dots,a+r-1\}
        \subset\{0,1,\dots,b-1\},
        \qquad 1\le r\le b.
\]
Then
\[
        P_{\mathcal D}(x)
        =e(ax)\sum_{u=0}^{r-1}e(ux),
\]
so the corresponding row-sum bound depends only on the length \(r\), not on
the translate \(a\).  The geometric-sum formula and
\(|\sin(\pi x)|\ge2\|x\|\) give
\[
        |P_{\mathcal D}(x)|
        \le
        \min\!\left(r,\frac1{2\|x\|}\right).
\]

\begin{proposition}[Row-sum bound for consecutive digit intervals]
\label{prop:transfer-consecutive-row-sum}
For every \(b\ge2\) and \(1\le r\le b\),
\[
        \rho(\mathcal L_{b,\mathcal D})
        \le
        \frac{b}{r}\bigl(\log(2r)+2\bigr).
\]
\end{proposition}

\begin{proof}
By \cref{lem:transfer-row-sum}, it suffices to bound the row sum.  Fix
\(x\in\mathbb T\) and put \(y_m=(x+m)/b\).  The points \(y_m\) are
\(1/b\)-separated modulo one, so for \(1\le \tau\le r\),
\[
        \#\left\{m:\|y_m\|<\frac1{2\tau}\right\}
        \le \frac b\tau+1.
\]
The layer-cake representation therefore gives
\[
\begin{aligned}
        \sum_{m=0}^{b-1}|P_{\mathcal D}(y_m)|
        &\le
        \int_0^r
        \#\left\{m:
        \min\!\left(r,\frac1{2\|y_m\|}\right)>\tau
        \right\}\,d\tau \\
        &\le
        b+\int_1^r\left(\frac b\tau+1\right)d\tau \\
        &=b+b\log r+r-1
        \le b\bigl(\log(2r)+2\bigr).
\end{aligned}
\]
Dividing by \(r\) proves the claim.
\end{proof}

\begin{corollary}[Consecutive intervals in the required exponent ranges]
\label{cor:transfer-consecutive-exponent-ranges}
Let
\[
        \mathcal D=\{a,a+1,\dots,a+r-1\}.
\]
If
\[
        r>b^{3/5}\bigl(\log(2r)+2\bigr),
\]
then one may choose an admissible exponent
$\alpha_{b,\mathcal D}<\frac25$.
If
\[
        r>b^{5/6}\bigl(\log(2r)+2\bigr),
\]
then one may choose an admissible exponent
$\alpha_{b,\mathcal D}<\frac16$.
\end{corollary}

\subsection{One missing digit}
\label{subsec:transfer-one-missing-digit}

Let
\[
        \mathcal D_{a_0}
        :=\{0,1,\dots,b-1\}\setminus\{a_0\},
        \qquad
        P_{\mathrm{full}}(y):=\sum_{a=0}^{b-1}e(ay).
\]
The preceding layer-cake argument, applied with \(r=b\), gives the elementary
full-grid estimate
\begin{equation}\label{eq:full-discrete-lebesgue}
        \sup_{x\in\mathbb T}
        \sum_{m=0}^{b-1}
        \left|P_{\mathrm{full}}\!\left(\frac{x+m}{b}\right)\right|
        \le b(\log b+2).
\end{equation}

\begin{proposition}[One-step estimate for one missing digit]
\label{prop:transfer-one-missing-digit-estimate}
For every \(b\ge2\) and every \(a_0\in\{0,1,\dots,b-1\}\),
\[
        \rho(\mathcal L_{b,\mathcal D_{a_0}})
        \le
        \frac b{b-1}(\log b+3).
\]
Consequently,
\[
        \log_b\!\bigl(\rho(\mathcal L_{b,\mathcal D_{a_0}})\bigr)
        \le
        \log_b\!\left[\frac b{b-1}(\log b+3)\right].
\]
\end{proposition}

\begin{proof}
Since
\(P_{\mathcal D_{a_0}}(y)=P_{\mathrm{full}}(y)-e(a_0y)\),
\eqref{eq:full-discrete-lebesgue} and the triangle inequality give
\[
        \sup_x\sum_{m=0}^{b-1}
        \left|P_{\mathcal D_{a_0}}\!\left(\frac{x+m}{b}\right)\right|
        \le b(\log b+3).
\]
The claim follows from \cref{lem:transfer-row-sum}.
\end{proof}

\begin{corollary}[One missing digit in the required exponent ranges]
\label{cor:transfer-one-missing-digit-thresholds}
Let
\[
        \mathcal D_{a_0}
        =
        \{0,1,\dots,b-1\}\setminus\{a_0\}.
\]
If \(b\ge202\), then one may choose an admissible exponent satisfying
\[
        \alpha_{b,\mathcal D_{a_0}}<\frac25.
\]
If \(b\ge95{,}313{,}426\), then one may choose an admissible exponent
satisfying
\[
        \alpha_{b,\mathcal D_{a_0}}<\frac16.
\]
\end{corollary}

\begin{proof}
By
\cref{prop:transfer-one-missing-digit-estimate,prop:transfer-controls-hybrid},
it suffices to verify
\[
        \frac b{b-1}(\log b+3)<b^{2/5}
\]
and
\[
        \frac b{b-1}(\log b+3)<b^{1/6},
\]
respectively.  The ratio
\[
        \frac{b^\gamma}
        {(b/(b-1))(\log b+3)}
\]
is increasing in \(b\) for \(\gamma\ge1/6\), and direct evaluation at the
two stated endpoints proves the claims.
\end{proof}

\begin{remark}
If \(a_0=0\) or \(a_0=b-1\), then \(\mathcal D_{a_0}\) is a consecutive
interval of length \(b-1\), so
\cref{prop:transfer-consecutive-row-sum} may give a sharper bound.
\end{remark}

\subsection{Several missing digits}
\label{subsec:transfer-several-missing-digits}

Let
\[
        \mathcal S\subset\{0,1,\dots,b-1\},
        \qquad
        |\mathcal S|=s,
        \qquad
        1\le s\le b-2,
\]
and set
\[
        \mathcal D:=\{0,1,\dots,b-1\}\setminus\mathcal S.
\]

\begin{proposition}[One-step estimate for \(s\) missing digits]
\label{prop:transfer-s-missing-digits-estimate}
One has
\[
        \rho(\mathcal L_{b,\mathcal D})
        \le
        \frac{b}{b-s}\left(\log b+2+\sqrt{s}\right).
\]
Consequently,
\[
        \log_b\!\bigl(\rho(\mathcal L_{b,\mathcal D})\bigr)
        \le
        \log_b\!\left[
        \frac{b}{b-s}\left(\log b+2+\sqrt{s}\right)
        \right].
\]
\end{proposition}

\begin{proof}
Write
\[
        P_{\mathcal S}(y):=\sum_{a\in\mathcal S}e(ay),
        \qquad
        P_{\mathcal D}=P_{\mathrm{full}}-P_{\mathcal S}.
\]
By \eqref{eq:full-discrete-lebesgue}, Cauchy's inequality, and orthogonality
on the \(b\)-point grid,
\[
\begin{aligned}
        \sum_{m=0}^{b-1}
        \left|P_{\mathcal D}\!\left(\frac{x+m}{b}\right)\right|
        &\le
        b(\log b+2)
        +b^{1/2}
        \left(
        \sum_{m=0}^{b-1}
        \left|P_{\mathcal S}\!\left(\frac{x+m}{b}\right)\right|^2
        \right)^{1/2} \\
        &=b\left(\log b+2+\sqrt{s}\right).
\end{aligned}
\]
Dividing by \(|\mathcal D|=b-s\) and applying
\cref{lem:transfer-row-sum} proves the claim.
\end{proof}

\begin{corollary}[\(s\) missing digits in the required exponent ranges]
\label{cor:transfer-s-missing-digits-thresholds}
Let
\[
        \mathcal D
        =
        \{0,1,\dots,b-1\}\setminus\mathcal S,
        \qquad
        |\mathcal S|=s,
        \qquad
        1\le s\le b-2.
\]
If
\[
        \frac{b}{b-s}
        \left(\log b+2+\sqrt{s}\right)
        <b^{2/5},
\]
then one may choose an admissible exponent satisfying
\[
        \alpha_{b,\mathcal D}<\frac25.
\]
If
\[
        \frac{b}{b-s}
        \left(\log b+2+\sqrt{s}\right)
        <b^{1/6},
\]
then one may choose an admissible exponent satisfying
\[
        \alpha_{b,\mathcal D}<\frac16.
\]
In particular, for every fixed \(0<\delta<1\) and all sufficiently large
\(b\), it is sufficient to assume, respectively,
\[
        s\le(1-\delta)b^{4/5},
        \qquad
        s\le(1-\delta)b^{1/3}.
\]
\end{corollary}

\begin{proof}
This follows from
\cref{prop:transfer-s-missing-digits-estimate,prop:transfer-controls-hybrid}.
The final assertions use
\[
        \frac{b}{b-s}=1+o(1)
\]
and the fact that \(\log b=o(b^\gamma)\).
\end{proof}

\begin{remark}
Application to the main theorems still requires the standing digit
non-resonance condition and the corresponding positivity criterion for the
projected local factor.  The estimate above is uniform in the locations of
the missing digits; for structured complements, the consecutive-interval
bound may be sharper.
\end{remark}

\section*{Acknowledgments}
\addcontentsline{toc}{section}{Acknowledgments}

Rui Han and Yaghoub Rahimi were supported in part by NSF DMS-2143369. 
OpenAI’s ChatGPT was used as an auxiliary tool for editorial assistance and figure preparation.

\appendix

\section{
Bad-denominator decay for the restricted-digit Fourier transform}

\label{app:digit-bad-major-general}

\begin{proof}[Proof of Lemma \ref{lem:digit-bad-major-general}]
Let
\[
        P_{\mathcal D}(t):=\sum_{a\in\mathcal D}e(at),
        \qquad
        g_{\mathcal D}:=\gcd\{a-a':a,a'\in\mathcal D\}.
\]
The product structure of the restricted-digit Fourier transform gives
\[
        \widehat C_k(\theta)
        =
        \prod_{j=0}^{k-1}P_{\mathcal D}(b^j\theta).
\]

By \cref{rem:digit-resonances}, the resonance set of the one-digit polynomial is
\[
        \mathcal R_{\mathcal D}
        :=
        \{t\in\mathbb T: |P_{\mathcal D}(t)|=r\}
        =
        \{t\in\mathbb T: g_{\mathcal D}t\in\mathbb Z\}.
\]
Hence, away from this finite resonance set, the modulus of
\(P_{\mathcal D}\) is strictly smaller than \(r\).  More precisely, for any
fixed
\[
        0<{\nu}<\frac12,
\]
compactness gives a constant
\[
        {
        \vartheta=\vartheta_{b,\mathcal D,\nu}>0
        }
\]
such that
\begin{equation}
\label{eq:digit-polynomial-away-from-resonances}
        \|g_{\mathcal D}t\|\ge {\nu}
        \qquad\Longrightarrow\qquad
        |P_{\mathcal D}(t)|
        \le
        r(1-{\vartheta}).
\end{equation}

We now show that many of the points \(b^j\theta\) are bounded away from the
resonance set.  {Since resonance is characterized by
\(g_{\mathcal D}t\in\mathbb Z\), the relevant quantity is
\[
        \|g_{\mathcal D}b^j\theta\|
        =
        \left\|
        g_{\mathcal D}\frac{b^j\ell}{d}
        +
        g_{\mathcal D}\frac{b^j\eta}{N}
        \right\|.
\]
We first control its rational part.}
Put
\[
        x_j:=\frac{b^j\ell}{d}.
\]
We claim that
\[
        g_{\mathcal D}x_j\notin\mathbb Z
        \qquad(j\ge0).
\]
Indeed, if \(g_{\mathcal D}x_j\in\mathbb Z\), then
\[
        d\mid g_{\mathcal D}b^j\ell.
\]
Since \((\ell,d)=1\), this implies
\[
        d\mid g_{\mathcal D}b^j.
\]
{By the digit non-resonance assumption in
\cref{def:digit-nonresonance}, every prime divisor of
\(g_{\mathcal D}\) divides \(b\).  On the other hand, by hypothesis there
exists a prime \(p\mid d\) with \(p\nmid b\).  It follows that
\(p\nmid g_{\mathcal D}\) and \(p\nmid b^j\), and hence
\[
        p\nmid g_{\mathcal D}b^j.
\]
Therefore \(d\nmid g_{\mathcal D}b^j\), a contradiction.}
Thus \(g_{\mathcal D}x_j\) is never an integer.  Since it is a rational number
with denominator dividing \(d\), we have
\[
        \|g_{\mathcal D}x_j\|\ge \frac1d
        \qquad(j\ge0).
\]

We claim that every block of
\[
        R:=\left\lceil \log_b(10bd)\right\rceil+2
\]
consecutive indices contains some \(j\) such that
\begin{align}\label{eq:gDx_j_large}
        \|g_{\mathcal D}x_j\|
        \ge
        {\frac1{10b}}.
\end{align}
{
Indeed, suppose otherwise that for some \(J\),
\[
        \|g_{\mathcal D}x_{J+i}\|
        <
        \frac1{10b}
        \qquad (0\le i<R).
\]
For \(0\le i<R-1\), the assumption gives
\[
        b\|g_{\mathcal D}x_{J+i}\|
        <
        \frac1{10}
        <
        \frac12.
\]
Recall that
\[
        \|by\|=b\|y\|
        \qquad\text{whenever}\qquad
        b\|y\|<\frac12.
\]
Since \(x_{J+i+1}=bx_{J+i}\), we therefore have
\[
        \|g_{\mathcal D}x_{J+i+1}\|
        =
        b\|g_{\mathcal D}x_{J+i}\|
        \qquad (0\le i<R-1).
\]
Iterating this one-step identity gives
\[
        \|g_{\mathcal D}x_{J+i}\|
        =
        b^i\|g_{\mathcal D}x_J\|
        \qquad (0\le i<R).
\]

Since \(g_{\mathcal D}x_J\) is a nonintegral rational with denominator
dividing \(d\), we have
\[
        \|g_{\mathcal D}x_J\|\ge\frac1d.
\]
Therefore
\[
\begin{aligned}
        \|g_{\mathcal D}x_{J+R-1}\|
        =
        b^{R-1}\|g_{\mathcal D}x_J\|      
        \ge
        \frac{b^{R-1}}{d}
        \ge
        \frac1{10b},
\end{aligned}
\]
where the last inequality follows from the definition of \(R\).  This
contradicts the assumption and proves the claim.
}

We next pass from the rational part \(x_j\) to the full frequency \(\theta\).
Since \(N\asymp X=b^k\) and \(|\eta|\le L\), we have
\[
        \left|
        \frac{g_{\mathcal D}b^j\eta}{N}
        \right|
        \ll_{b,\mathcal D}
        \frac{Lb^j}{b^k}.
\]
Choose an integer \(J_0\) with
\[
        J_0\asymp_{b,\mathcal D,A}\log\log X
\]
large enough so that
\[
        \frac{g_{\mathcal D}Lb^j}{N}
        \le
        {\frac1{20b}}
        \qquad
        \text{for all }0\le j\le k-J_0.
\]
Then, whenever \(0\le j\le k-J_0\) and
\[
        \|g_{\mathcal D}x_j\|
        \ge
        {\frac1{10b}},
\]
we have
\[
\begin{aligned}
        \|g_{\mathcal D}b^j\theta\|
        &=
        \left\|
            g_{\mathcal D}\frac{b^j\ell}{d}
            +
            g_{\mathcal D}\frac{b^j\eta}{N}
        \right\|                          \ge
        {\frac1{20b}}.
\end{aligned}
\]

Applying \eqref{eq:digit-polynomial-away-from-resonances} with
${\nu=1/(20b)}$,
we obtain
\begin{align}\label{eq:PD_small}
        |P_{\mathcal D}(b^j\theta)|
        \le
        r(1-{\vartheta})
\end{align}
for some $\vartheta>0$, for all such indices \(j\).

By the block argument in \eqref{eq:gDx_j_large}, among the indices
$0\le j\le k-J_0$,
there are at least
\[
        \gg_b \frac{k}{\log(2d)}
\]
indices for which
\[
        \|g_{\mathcal D}b^j\theta\|
        \ge
        {\frac1{20b}}.
\]
On these indices we gain
a factor \(1-{\vartheta}\), see \eqref{eq:PD_small}, while on all remaining indices we use
the trivial bound
\[
        |P_{\mathcal D}(b^j\theta)|\le r.
\]
Consequently,
\[
\begin{aligned}
        |\widehat C_k(\theta)|
        &\le
        r^k(1-{\vartheta})^{c_b k/\log(2d)}     
        \le
        |\mathcal C_k|
        \exp\!\left(
            -c_{b,\mathcal D}\frac{k}{\log(2d)}
        \right).
\end{aligned}
\]
Since \(d\le L=(\log X)^A\), we have
\[
        \log(2d)\ll_A \log\log X.
\]
Also \(k=\log_b X\).  Therefore
\[
        \frac{k}{\log(2d)}
        \gg_{b,A}
        \frac{\log X}{\log\log X}.
\]
Hence
\[
        |\widehat C_k(\theta)|
        \ll_{b,\mathcal D,A}
        |\mathcal C_k|
        \exp\!\left(
            -c_{b,\mathcal D,A}\frac{\log X}{\log\log X}
        \right).
\]
This proves the lemma.
\end{proof}

\section{Evaluation of the projected three-prime local factor}
\label{sec:three-term-local-factor}

\begin{proof}[Proof of Lemma \ref{lem:three-term-local-factor-evaluation}]
The absolute convergence established earlier permits rearrangement of
\(S_{1,\mathrm{full}}\) into an Euler product.  The summand is
multiplicative in the squarefree \(p\)-parts of \(d'\), and the sum over
\(\ell'\) factors by the Chinese remainder theorem.

We first compute the local factor when \(p\nmid d\).  If \(p\nmid d'\),
the local contribution is \(1\).  Suppose \(p\mid d'\) and \(p>2\).
Since \((\ell',d')=1\), we have \(p\nmid\ell'\), and
\[
        \ell d'-2\ell'd
        \equiv -2\ell'd\not\equiv0\pmod p,
        \qquad
        \ell'd-\ell d'
        \equiv \ell'd\not\equiv0\pmod p.
\]
Thus \(p\) survives in both reduced denominators \(d_1,d_2\).
Summing over the \(p-1\) possible nonzero residue classes of \(\ell'\)
modulo \(p\), the contribution from \(p\mid d'\) is
\[
        (p-1)\frac{\mu(p)^3}{\phi(p)^3}
        =
        -\frac{1}{(p-1)^2}.
\]
Hence, for \(p>2\) with \(p\nmid d\), the local factor is
\[
        1-\frac{1}{(p-1)^2}.
\]
For \(p=2\) with \(2\nmid d\), the contribution from \(2\mid d'\) is
\(+1\), so the local factor at \(2\) is \(2\).  Taking \(d=1\) gives
\eqref{eq:S0-three-term}.

We next show that no non-squarefree \(d\) contributes.  Suppose that
\(p^e\Vert d\) with \(e\ge2\).  If \(p\nmid d'\), then \(p^e\)
survives in both reduced denominators \(d_1,d_2\).  If \(p\mid d'\),
then \(\mu(d')\neq0\) forces \(d'\) to be squarefree.  Writing
\(d=p^e d_0\) and \(d'=pm\), with \(p\nmid d_0m\), gives
\[
\frac{\ell d'-2\ell'd}{dd'}
=
\frac{\ell m-2p^{e-1}\ell'd_0}{p^e d_0m},
\]
and
\[
\frac{\ell'd-\ell d'}{dd'}
=
\frac{p^{e-1}\ell'd_0-\ell m}{p^e d_0m}.
\]
Both numerators are nonzero modulo \(p\), so again \(p^e\) survives in
both \(d_1\) and \(d_2\).  Thus
\[
        \mu(d_1)\mu(d_2)=0,
\]
and therefore
\[
        S_{1,\mathrm{full}}\!\left(\frac{\ell}{d}\right)=0.
\]

It remains to compute the local factor when \(p\mid d\) and \(d\) is
squarefree at \(p\).  Write
\[
        d=pd_0,
        \qquad
        p\nmid d_0.
\]
First suppose \(p>2\).  If \(p\nmid d'\), then \(p\) survives in both
\(d_1,d_2\), giving the local contribution
\[
        \frac{1}{(p-1)^2}.
\]

If \(p\mid d'\), write
\[
        d'=pm,
        \qquad
        p\nmid m.
\]
After cancelling one common factor \(p\), the remaining \(p\)-divisibility
is determined by
\[
        \ell m-2\ell'd_0,
        \qquad
        \ell'd_0-\ell m.
\]
As \(\ell'\) ranges over the reduced residue classes modulo \(p\), the
quantity
\[
        t:=\frac{\ell'd_0}{\ell m}\pmod p
\]
ranges over \(\mathbb F_p^\times\).  The first reduced denominator loses
its remaining factor \(p\) exactly when \(t=1/2\), while the second loses
its remaining factor \(p\) exactly when \(t=1\).  These two residue
classes contribute
\[
        \frac{1}{(p-1)^2}
\]
each, while the remaining \(p-3\) residue classes contribute
\[
        -\frac{1}{(p-1)^3}
\]
each.  Thus the contribution from \(p\mid d'\) is
\[
        \frac{2}{(p-1)^2}
        -
        \frac{p-3}{(p-1)^3}
        =
        \frac{p+1}{(p-1)^3}.
\]
Adding the contribution from \(p\nmid d'\), the local factor at a prime
\(p>2\) dividing \(d\) is
\[
        \frac{1}{(p-1)^2}
        +
        \frac{p+1}{(p-1)^3}
        =
        \frac{2p}{(p-1)^3}.
\]
The corresponding local factor in \(S_0\) is
\[
        1-\frac{1}{(p-1)^2}
        =
        \frac{p(p-2)}{(p-1)^2}.
\]
Hence the ratio is
\[
        \frac{2p}{(p-1)^3}
        \cdot
        \frac{(p-1)^2}{p(p-2)}
        =
        \frac{2}{(p-1)(p-2)}.
\]

For \(p=2\mid d\), both the cases \(2\nmid d'\) and \(2\mid d'\)
give local contribution \(1\).  Hence the local factor at \(2\mid d\)
is \(2\), the same as the local factor at \(2\) in \(S_0\).  This proves
\eqref{eq:S1full-local-factorized}.

Since only squarefree divisors \(d\mid b\) contribute, the definition of
\(\mathfrak S_b^{(3)}(n)\) and the Ramanujan identity
\[
        c_d(n)
        =
        \sum_{\substack{1\le\ell\le d\\(\ell,d)=1}}
        e\!\left(\frac{n\ell}{d}\right)
\]
give
\[
        \mathfrak S_b^{(3)}(n)
        =
        S_0
        \sum_{d\mid b}
        \mu^2(d)
        \left(\prod_{p\mid d}\beta_p\right)c_d(n).
\]
By multiplicativity of the Ramanujan sums on squarefree moduli, this is
equivalent to \eqref{eq:S3-Ramanujan-product}.

For a prime \(p\),
\[
        c_p(n)
        =
        \begin{cases}
        p-1, & p\mid n,\\
        -1, & p\nmid n.
        \end{cases}
\]
Therefore
\[
1+\beta_pc_p(n)
=
\begin{cases}
2, & p=2,\ p\mid n,\\
0, & p=2,\ p\nmid n,\\[2pt]
3, & p=3,\ p\mid n,\\
0, & p=3,\ p\nmid n,\\[2pt]
\dfrac{p}{p-2}, & p>3,\ p\mid n,\\[6pt]
1-\dfrac{2}{(p-1)(p-2)}, & p>3,\ p\nmid n.
\end{cases}
\]
All these factors are nonnegative, and for \(p>3\) both possibilities
are strictly positive.  Hence
\[
        \mathfrak S_b^{(3)}(n)>0
\]
if and only if every prime among \(2,3\) that divides \(b\) also divides
\(n\), which is equivalent to
\[
        \gcd(b,6)\mid n.
\]
This proves \eqref{eq:S3-positivity}.  Averaging over
\(a\in\mathcal D\) gives
\[
        \overline{\mathfrak S}_{b,\mathcal D}^{(3)}>0
        \quad\Longleftrightarrow\quad
        \mathcal D\cap\gcd(b,6)\mathbb Z\neq\emptyset.
\]
\end{proof}

\end{document}